\documentclass[11pt]{amsart} 

\usepackage{amsmath}
\usepackage{amssymb,mathrsfs,shuffle}
\usepackage{enumerate}

\usepackage[all]{xy}

\usepackage{tikz,pgf}
\usepackage{tikz-cd}
\usepackage{pgfmath}
\usetikzlibrary{arrows}
\usetikzlibrary{decorations.markings}

\usepackage{color}
\usepackage{tabularx}

\usepackage[hidelinks]{hyperref}

\def\comment#1{{\sf{[#1]}}}

\def\Z{{\mathbb Z}}
\def\Q{{\mathbb Q}}
\def\N{{\mathbb N}}
\def\R{{\mathbb R}}
\def\C{{\mathbb C}}
\def\P{{\mathbb P}}
\def\F{{\mathbb F}}

\def\A{{\mathbb A}}
\def\D{{\mathbb D}}
\def\H{{\mathbb H}}
\def\KK{{\mathbb K}}
\def\LL{{\mathbb L}}

\def\T{{\mathbb T}}
\def\V{{\mathbb V}}
\def\W{{\mathbb W}}

\def\kk{{\Bbbk}}		

\def\cA{{\mathcal A}}
\def\cB{{\mathcal B}}

\def\cE{{\mathcal E}}
\def\cF{{\mathcal F}}
\def\cG{{\mathcal G}}
\def\cH{{\mathcal H}}

\def\cL{{\mathcal L}}
\def\M{{\mathcal M}}

\def\cO{{\mathcal O}}
\def\cP{{\mathcal P}}

\def\U{{\mathcal U}}
\def\cV{{\mathcal V}}

\def\sF{{\mathscr F}}
\def\sL{{\mathscr L}}

\def\sP{{\mathscr P}}
\def\sR{{\mathscr R}}
\def\sS{{\mathscr S}}

\def\h{{\mathfrak h}}

\def\n{{\mathfrak n}}

\def\sL{{\mathfrak{sl}}} 

\def\u{{\mathfrak u}}
\def\fU{{\mathfrak U}}

\def\e{{\varepsilon}}
\def\w{{\omega}}
\def\G{{\Gamma}}

\def\bG{\boldsymbol{\mathrm{G}}}

\def\bO{\boldsymbol{O}}

\def\ee{\mathbf{e}}
\def\ff{\mathbf{f}}

\def\ba{\mathbf{a}}
\def\bb{\mathbf{b}}

\def\bw{\boldsymbol{w}}

\def\bmu{\boldsymbol{\mu}}
\def\blambda{\boldsymbol{\lambda}}

\def\sfS{\mathsf{S}}
\def\sfT{\mathsf{T}}

\def\hhat{\hat{\h}}

\def\gtilde{{\tilde{g}}}
\def\alphatilde{{\tilde{\alpha}}}
\def\gammatilde{{\tilde{\gamma}}}

\def\pitilde{{\tilde{\pi}}}
\def\rhotilde{{\tilde{\rho}}}
\def\phitilde{{\tilde{\varphi}}}

\def\Vtilde{{\widetilde{\V}}}
\def\cLtilde{{\widetilde{\cL}}}
\def\cVtilde{{\widetilde{\cV}}}

\def\That{{\widehat{\T}}}

\def\Zhat{{\widehat{\Z}}}

\def\Tdual{{\check{T}}}
\def\edual{{\check{\ee}}}

\def\ubar{{\overline{u}}}
\def\xbar{{\overline{x}}}

\def\zbar{{\overline{z}}}

\def\cHbar{\overline{\cH}}

\def\Kbar{{\overline{K}}}
\def\Mbar{{\overline{\M}}}

\def\Qbar{{\overline{\Q}}}

\def\cVbar{{\overline{\cV}}}
\def\cEbar{{\overline{\cE}}}

\def\kkbar{{\overline{\kk}}}

\def\alphabar{{\overline{\alpha}}}

\def\Deltabar{{\overline{\Delta}}}

\def\Vdual{\check{V}}

\def\vv{{\vec{\mathsf v}}}

\def\Gm{{\mathbb{G}_m}}

\def\Sp{{\mathrm{Sp}}}

\def\SL{{\mathrm{SL}}}
\def\GL{{\mathrm{GL}}}
\def\PSL{{\mathrm{PSL}}}
\def\MT{{\mathrm{MT}}}

\def\PG{{\mathrm{P}\Gamma}}

\def\MHS{{\mathsf{MHS}}}

\def\Vec{{\mathsf{Vec}}}
\def\Rep{{\mathsf{Rep}}}

\def\tC{{\mathsf{C}}}

\def\tS{{\mathsf{S}}}
\def\tE{{\mathsf{E}}}

\def\ab{{\mathrm{ab}}}
\def\an{{\mathrm{an}}}
\def\et{{\mathrm{\acute{e}t}}}
\def\etl{{\et_\ell}}
\def\cts{{\mathrm{cts}}}

\def\un{{\mathrm{un}}}
\def\cusp{{\mathrm{cusp}}}

\def\op{{\mathrm{op}}}

\def\DR{{\mathrm{DR}}}
\def\cyc{{\mathrm{cyc}}}
\def\conj{{\mathrm{conj}}}

\def\trans{{\mathrm{T}}}			      
\def\per{{\mathrm{per}}}
\def\nilp{{\mathrm{nil}}}

\def\tor{{\mathrm{ell}}}
\def\hor{{\mathrm{par}}}
\def\hyp{{\mathrm{hyp}}}

\def\KDR{{\Q^\ab}}						   
\def\etl{{\et_\ell}}

\def\To{\longrightarrow}
\def\bdot{{\bullet}}
\def\bs{\backslash}
\def\bbs{{\bs\negthickspace \bs}}
\def\ffs{{/\negthickspace /}}
\def\blank{{\phantom{x}}}				

\def\cyl{\odot}                         

\def\(({(\!(}
\def\)){)\!)}

\def\Pminus{{\P^1-\{0,1,\infty\}}}

\def\Ql{{\Q_\ell}}

\def\Fp{{\F_p}}

\def\unit{{\mathbf{1}}}
\def\uu{{\vec{1}}}

\def\Cl{{\mathscr{C}\!\ell}}
\def\Sh{\mathrm{Sh}}

\newcommand{\diag}{\text{diag}}

\newcommand\im{\operatorname{im}} 
\newcommand\id{\operatorname{id}}

\newcommand\tr{\operatorname{tr}}

\newcommand\coker{\operatorname{coker}}

\newcommand\Hom{\operatorname{Hom}}
\newcommand\Ext{\operatorname{Ext}}

\newcommand\End{\operatorname{End}}
\newcommand\Aut{\operatorname{Aut}}

\newcommand\Out{\operatorname{Out}}

\newcommand\Gr{\operatorname{Gr}}

\newcommand\CoInd{\operatorname{CoInd}}
\newcommand\Sym{\operatorname{Sym}}
\newcommand\Gal{\operatorname{Gal}}

\newcommand\Pic{\operatorname{Pic}}
\newcommand\Spec{\operatorname{Spec}}
\newcommand\Spf{\operatorname{Spf}}
\newcommand\Cov{\operatorname{Cov}}

\newcommand\rot{\operatorname{rot}}

\newcommand\comptensor{\operatorname{\widehat{\otimes}}}

\newtheorem{theorem}{Theorem}[section]
\newtheorem{lemma}[theorem]{Lemma}
\newtheorem{proposition}[theorem]{Proposition}
\newtheorem{corollary}[theorem]{Corollary}
\newtheorem{bigtheorem}{Theorem}
\newtheorem{bigproposition}[bigtheorem]{Proposition}

\theoremstyle{definition}
\newtheorem{definition}[theorem]{Definition}
\newtheorem{example}[theorem]{Example}

\theoremstyle{remark}
\newtheorem{remark}[theorem]{Remark}
\newtheorem{question}[theorem]{Question}

\begin{document}

\title[Hecke Actions on Loops and Iterated Shimura Integrals]{Hecke Actions on Loops and Periods of Iterated Shimura Integrals
\\
\smallskip
{\small {\em Actions de correspondances de Hecke sur les lacets et les p\'eriodes des int\'egrales de Shimura it\'er\'ees}
}}

\dedicatory{Dedicated to the memory of Yuri Manin}  

\author{Richard Hain}
\address{Department of Mathematics\\ Duke University\\
Durham, NC 27708-0320}
\email{hain@math.duke.edu}

\thanks{ORCID: {\sf 0000-0002-7009-6971}}

\date{\today}



\maketitle

\begin{center}
{{\sc with an appendix by pham huu tiep}}
\end{center}

\begin{abstract}
We show that the classical Hecke correspondences $T_N$ act on the free abelian group generated by the conjugacy classes of the modular group $\SL_2(\Z)$ and the conjugacy classes of its profinite completion. We show that this action induces a dual action on the ring of class functions of a certain relative unipotent completion of the modular group. This ring contains all iterated integrals of modular forms that are constant on conjugacy classes. It possesses a natural mixed Hodge structure and, after tensoring with $\Ql$, a natural action of the absolute Galois group. Each Hecke correspondence preserves this mixed Hodge structure and commutes with the action of the absolute Galois group. Unlike in the classical case, where Hecke correspondences are acting on modular forms, the algebra generated by these generalized Hecke operators is not commutative.

In the appendix, Pham Tiep proves that, for all primes $p\ge 5$, every irreducible character of $\SL_2(\Z/p^n)/(\pm \id)$ appears in its conjugation action on the group algebra of $\SL_2(\Z/p^n)$, a result needed in the body of the paper.
\bigskip

\noindent{\sc Resum\'e.} Nous montrons que les correspondances de Hecke classiques $T_N$ agissent sur le groupe ab\'elien libre engendr\'e par les classes de conjugaison du groupe modulaire $\SL_2(\Z)$ et les classes de conjugaison de sa compl\'etion profinie. Nous montrons que cette action induit une action duale sur l'anneau des fonctions de classe d'une certaine compl\'etion unipotente relative du groupe modulaire. Cet anneau contient toutes les int\'egrales it\'er\'ees de formes modulaires qui sont constantes sur les classes de conjugaison. Il poss\`ede une structure de Hodge mixte naturelle et, après tensorisation avec $\Ql$, une action naturelle du groupe de Galois absolu. Chaque correspondance de Hecke pr\'eserve cette structure de Hodge mixte et commute avec l'action du groupe de Galois absolu. Contrairement au cas classique, o\`u les correspondances de Hecke agissent sur les formes modulaires, l'alg\`ebre engendr\'ee par ces op\'erateurs de Hecke g\'en\'eralis\'es n'est pas commutative.

En annexe, Pham Tiep prouve que, pour tous les nombres premiers $p\ge 5$, tout caract\`ere irr\'eductible de $\SL_2(\Z/p^n)/(\pm \id)$ appara\^it dans son action par conjugaison sur l'alg\`ebre de groupe de $\SL_2(\Z/p^n)$, un r\'esultat utilis\'e dans le corps de l'article.
\end{abstract}

\tableofcontents

\section{Introduction}

In this paper we show that the action of Hecke correspondences on invariants of modular curves, such as their cohomology groups, lifts to some non-abelian invariants. More precisely, we show that the classical Hecke correspondences $T_N$ ($N\in\N_+$) act on the free abelian groups generated by the conjugacy classes of $\SL_2(\Z)$ and the conjugacy classes of its profinite completion. We show that this action induces a dual action on those iterated integrals of modular forms ({\em iterated Shimura integrals} in Manin's terminology \cite{manin}) that are constant on conjugacy classes, and also on their non-holomorphic generalizations. These form a ring of class functions on $\SL_2(\Z)$ which possesses a natural mixed Hodge structure. Each such Hecke operator preserves this mixed Hodge structure and commutes with the action of the absolute Galois group on the $\ell$-adic analogue of this ring. Unlike in the classical case, the algebra generated by these generalized Hecke operators is not commutative.

The problem of defining a Hecke action on iterated Shimura integrals was posed by Manin in \cite[\S3.3]{manin} where he writes:
\begin{quote}
{\em The problem of extending these results to the iterated case remains a major challenge. One obstacle is that correspondences (in particular, Hecke correspondences) do not act directly on the fundamental groupoid (as opposed to the cohomology) and hence do not act on the iterated integrals which provide homomorphisms of this groupoid.}
\end{quote}
An initial attempt to define a Hecke action on iterated Shimura integrals was made by him in \cite[\S5.2]{manin_mod_symbs}. Restricting our attention to conjugation invariant iterated integrals circumvents the problem of base points.

The overall goal of the project is to use this Hecke action to understand periods of iterated Shimura integrals and extensions in the categories of mixed Hodge structures and $\ell$-adic Galois representations of the form
\begin{equation}
\label{eqn:mot_ext}
0 \to \big(\Sym^{r_1} V_{f_1} \otimes \dots \otimes \Sym^{r_m} V_{f_m}\big) (d) \to E \to \Q(0)\to 0
\end{equation}
that occur in subquotients of the coordinate ring of relative unipotent completions of modular groups. Here $f_1,\dots,f_m$ are Hecke eigen cusp forms, $V_{f_j}$ the simple $\Q$-Hodge structure or $\ell$-adic Galois representation that corresponds to $f_j$, and $\Sym^{r_j} V_{f_j}$ its $r_j$th symmetric power. The extensions (\ref{eqn:mot_ext}) are expected to be the Hodge and $\ell$-adic realizations of Voevodsky motives.

The coordinate rings of such relative completions contain all iterated Shimura integrals. It is known by the work of Francis Brown on {\em multiple modular values} \cite[Ex.~17.6]{brown:mmv} that all of the Hodge extensions of $\Q$ by $V_f(d)$ predicted by the conjectures of Beilinson \cite[Conj.~3.4a]{beilinson} do occur in the coordinate ring of the standard relative completion of $\SL_2(\Z)$. This is implied by the fact that the periods of twice iterated integrals of Eisenstein series can contain non-critical $L$-values of cusp forms, which was proved by Brown in \cite{brown:mmv}. It is hoped that all of the extensions (\ref{eqn:mot_ext}) predicted by Beilinson's conjectures, and not excluded by Brown's observation \cite[\S17]{brown:mmv}, occur in these coordinate rings.

The construction of the Hecke action on conjugacy classes is elementary and natural. Denote the set of conjugacy classes of a discrete (or profinite) group $\G$ by $\blambda(\G)$ and by $\kk\blambda(\G)$ the free $\kk$-module generated by it, where $\kk$ is a commutative ring. The central objects of this paper are $\kk\blambda(\SL_2(\Z))$ and, dually, the functions $\blambda(\SL_2(\Z)) \to \C$ that arise from conjugation-invariant iterated integrals of modular forms.

Before proceeding, it is worth recalling the relation between conjugacy classes in $\SL_2(\Z)$ and closed geodesics on the modular curve, which we regard as an orbifold, or more accurately, a stack. The map
$$
\blambda(\SL_2(\Z)) \to \blambda(\PSL_2(\Z))
$$
is 2-to-1 with fibers $\{\gamma,-\gamma\}$. Apart from the two conjugacy classes of elements of order 4 of $\SL_2(\Z)$, the two preimages of $\mu\in \blambda(\PSL_2(\Z))$ are distinguished by the signs of their traces. The conjugacy classes of non-torsion elements of $\PSL_2(\Z)$ correspond to powers of oriented closed geodesics in the modular curve and powers of the horocycle. So a non-torsion conjugacy class of $\SL_2(\Z)$ corresponds to either a (not necessarily prime) closed geodesic or a non-zero power of the horocycle on the modular curve, together with the sign of its trace. This is explained in more detail in Section~\ref{sec:conj_cl}. 

\begin{bigtheorem}
\label{thm:T_commute}
The classical Hecke correspondences $T_N$, $N\in \N$, act on $\Z\blambda(\SL_2(\Z))$. The operators $T_N$ and $T_M$ commute when $M$ and $N$ are relatively prime. These actions of the $T_N$ descend to $\Z\blambda(\PSL_2(\Z))$.
\end{bigtheorem}

This is proved in Section~\ref{sec:hecke_action}. The basic observation behind the existence of this Hecke action is that if $\pi : Y \to X$ is a finite unramified cover of topological spaces, then there are pushforward and pullback maps
$$
\pi_\ast : \Z\lambda(Y) \to \Z\lambda(X) \text{ and } \pi^\ast : \Z\lambda(X) \to \Z\lambda(Y),
$$
where, for a topological space $X$, $\lambda(X)$ denotes the set of free homotopy classes of maps from the circle to $X$. When $X$ is path connected, $\lambda(X) = \blambda(\pi_1(X,x))$. The pushforward map is simply composition with $\pi$; the pullback map takes a loop in $X$ to the sum of closed loops in $Y$ that cover its preimage under $\pi$. The precise definition can be found in Section~\ref{sec:pullback_def}. (This notion also occurs independently in \cite{rss}, where it is called a {\em transfer map}.) In particular, this gives a definition of pushforward and pullback maps
$$
\pi_\ast : \Z\blambda(\G') \to \Z\blambda(\G) \text{ and } \pi^\ast : \Z\blambda(\G) \to \Z\blambda(\G')
$$
associated with the inclusion $\pi : \G'\hookrightarrow \G$ of a finite index subgroup by taking $X$ and $Y$ to be appropriate models of their classifying spaces $B\G$ and $B\G'$. One finds that
$$
\pi_\ast \pi^\ast (\gamma) = \sum_j \gamma^{m_j}\qquad \gamma \in \blambda(\G),
$$
where the $m_j$ are positive integers that depend on $\gamma$ and whose sum is the degree of $\pi$, which is the index of $\G'$ in $\G$.

When $p$ is a prime number, the (generalized) Hecke operator
$$
T_p : \Z\blambda(\SL_2(\Z)) \to \Z\blambda(\SL_2(\Z))
$$
is the map
$$
\xymatrix{
\Z\blambda(\SL_2(\Z)) \ar[r]^{\pi^\ast} & \Z\blambda(\G_0(p)) \ar[r]^{\pi^\op_\ast} & \Z\blambda(\SL_2(\Z))
}
$$
induced by the inclusions $\pi : \G_0(p) \hookrightarrow \SL_2(\Z)$ and $\pi^\op : \G_0(p) \hookrightarrow \SL_2(\Z)$, where
$$
\pi^\op : \gamma \mapsto \begin{pmatrix} p^{-1} & 0 \cr 0 & 1 \end{pmatrix}\gamma^{-\trans} \begin{pmatrix} p & 0 \cr 0 & 1 \end{pmatrix}.
$$
Here $(\blank)^{-\trans}$ denotes inverse transpose and $\G_0(N)$ is the subgroup of $\SL_2(\Z)$ whose elements are upper triangular mod $N$.

Since $\pi_\ast\circ \pi^\ast$ is not simply multiplication by $\deg\pi$, the classical relation that expresses $T_{p^n}$ as a polynomial in $T_p$ no longer holds and has to be modified. To this end, for each prime number $p$, define
$$
\ee_p :  \Z\blambda(\SL_2(\Z)) \to \Z\blambda(\SL_2(\Z))
$$
to be the map $\pi_\ast\pi^\ast - \id$ associated with the inclusion $\pi : \G_0(p) \hookrightarrow \SL_2(\Z)$. For example, if $\sigma_o$ is the class of
$$
\begin{pmatrix} 1 & 1 \cr 0 & 1 \end{pmatrix}
$$
then
\begin{equation}
\label{eqn:e_p-sigma}
\ee_p(\sigma_o^n) =
\begin{cases}
p\sigma_o^n & p|n, \cr
\sigma_o^{np} & p \nmid n.
\end{cases}
\end{equation}
In particular, $\ee_p$ is {\em not}, in general, multiplication by $p$. The operator $\ee_p$ satisfies the polynomial relation $m_p(\ee_p) = 0$, where
\begin{equation}
\label{eqn:poly}
m_2(x) = x(x+1)(x-2) \text{ and } m_p(x) = x(x^2-1)(x-p) \text{ when $p$ is odd}.
\end{equation}
(See Section~\ref{sec:min_poly}.) It commutes with $\ee_q$ for all primes $q$.

The general shape of the formula for the action of $T_p$ on a conjugacy class $\alpha$ is
$$
T_p(\alpha) = \text{a finite sum of classes of $\GL_2(\Q)$-conjugates of fractional powers of $\alpha$},
$$
where each term of the sum lies in $\SL_2(\Z)$. For example,
\begin{equation}
\label{eqn:T_p-sigma}
T_p(\sigma_o^n) = \sigma_o^{np} +
\begin{cases}
p\sigma_o^{n/p} & p|n, \cr
\sigma_o^n & p \nmid n.
\end{cases}
\end{equation}
The general formula for $T_p(\alpha)$ is given in Section~\ref{sec:formula}.

\begin{bigtheorem}
\label{thm:Tp^n}
The actions of the Hecke correspondences $T_{p^n}$ on $\Z\blambda(\SL_2(\Z))$ satisfy
\begin{equation}
\label{eqn:relation}
T_{p^n}\circ T_p = T_{p^{n+1}} + T_{p^{n-1}}\circ\ee_p.
\end{equation}
\end{bigtheorem}

This relation should be compared with the familiar relation (Prop.~10 in Chapter VII of \cite{serre:arithmetic})
$$
T_{p^n}\circ T_p = T_{p^{n+1}} + pT_{p^{n-1}}R_p
$$
that holds between the classical Hecke operators acting on modular forms, where $R_p$ rescales lattices by $p$. Both may be considered as specializations of the relation
$$
T_{p^n}\circ T_p = T_{p^{n+1}} + T_{p^{n-1}}\circ (R_p\ee_p)
$$
where $R_p$ and $\ee_p$ commute. In the classical case, $\ee_p = p$, whereas in our case, $\ee_p \neq p$ and $R_p = 1$, as rescalling lattices by $p$ acts trivially on $\blambda(\SL_2(\Z))$.

The formulas (\ref{eqn:e_p-sigma}) and (\ref{eqn:T_p-sigma}) for the action of $T_p$ and $\ee_p$ on $\sigma_o$ imply that $T_p$ and $\ee_p$ do not commute. (See Example~\ref{ex:non_commute}.) This and the relation
$$
T_{p^2} =  T_p^2 - \ee_p
$$
imply that $T_p$ does not commute with $T_{p^2}$, unlike in the classical case.

In view of this, it is natural to consider, for each prime number $p$, the quotient
$$
\That_p := \Z\langle T_p,\ee_p\rangle/(m_p(\ee_p))
$$
of the free associative algebra generated by symbols $T_p$ and $\ee_p$, by the two-sided ideal generated by $m_p(\ee_p)$, where $m_p(x)$ is the polynomial (\ref{eqn:poly}). When $n>1$, one can define elements $T_{p^n}$ of $\That_p$ inductively by
$$
T_{p^{n+1}} = T_{p^n}T_p - T_{p^{n-1}}\ee_p.
$$
One can then define $\That$ by
$$
\That := \varinjlim_{N>0} \bigotimes_{p\le N} \That_p.
$$
This is a non-commutative generalization of the classical Hecke algebra $\T$ that acts on $\Z\blambda(\SL_2(\Z))$. The classical Hecke algebra is obtained from it by adding the relations $\ee_p = p$:
$$
\T = \That/(\ee_p - p : p\text{ prime}) \cong \Z[T_p : p\text{ prime}].
$$
This suggests the purely topological question:

\begin{question}
Is the Hecke action $\That \to \End\Z\blambda(\SL_2(\Z))$ injective?
\end{question}

For each $m\ge 0$ one also has the Adams operator $\psi^m : \Z\blambda(\Gamma) \to \Z\blambda(\Gamma)$, which takes the class of $\gamma \in \G$ to the class of $\gamma^m$. At present, it is not know how the Adams operators interact with the Hecke operators.

For each appropriate choice of a base point of the modular curve, there is a natural action of the absolute Galois group on the profinite completion $\SL_2(\Z)^\wedge$ of $\SL_2(\Z)$. (See Section~\ref{sec:fund_gps+galois}.) It induces a Galois action on $\Z\blambda(\SL_2(\Z)^\wedge)$ which commutes with the Adams operators. This action does not depend on the choice of a base point. The group-theoretic description of the action of $\That$ on $\Z\lambda(\SL_2(\Z))$ implies that $\That$ also acts on $\Z\lambda(\SL_2(\Z)^\wedge)$.

\begin{bigtheorem}
The action of $\That$ on $\Z\blambda(\SL_2(\Z))$ lifts to an action of $\That$ on $\Z\blambda(\SL_2(\Z)^\wedge)$. This action commutes with that of the absolute Galois group. Complex conjugation acts on both $\Z\blambda(\SL_2(\Z))$ and $\Z\blambda(\SL_2(\Z)^\wedge)$ as conjugation by $\diag(1,-1)\in \GL_2(\Z)$.
\end{bigtheorem}

Note that it is not necessary to complete the coefficient ring at this stage. If this seems odd, note that the absolute Galois group acts on $\SL_2(\Z)^\wedge$, and thus on its integral group ring $\Z[\SL_2(\Z)^\wedge]$.

Similarly one can define generalized Hecke operators for lattices in higher rank groups. For example, one has the Hecke operators
$$
T_N : \Z\blambda(\GL_n(\Z)) \to \Z\blambda(\GL_n(\Z)) \text{ and } \Z\blambda(\GL_n(\Zhat)) \to \Z\blambda(\GL_n(\Zhat)).
$$
One difference with the rank 1 case is that, when $n\ge 3$, the congruence kernel of $\SL_n(\Z)$ is trivial by \cite{bms}, so that the profinite completion of $\SL_n(\Z)$ is $\SL_n(\Zhat)$ unlike in the case of $\SL_2(\Z)$. (See \cite{melnikov}.)

\subsection{Relative completions of the modular group}

Relative unipotent completion (or relative completion for short) of a discrete or profinite group replaces it by an affine group scheme over a field of characteristic zero. In more classical language, relative completion replaces a discrete or profinite group by a proalgebraic group. (Relative completion is reviewed in Section~\ref{sec:rel_comp}. A gentler introduction can be found in \cite{hain:modular}.)

In characteristic zero, every affine group $G$ is an extension
$$
1 \to U \to G \to R \to 1
$$
of a proreductive group $R$ by a prounipotent group $U$. Every finite group can be regarded as a reductive algebraic group, so every profinite group can be regarded as a proreductive group. The completion $\cG$ of $\SL_2(\Z)$ that we consider in this paper is constructed in Section~\ref{sec:rel_comp_mod_gp}. It is a $\Q$-group that is an extension of the proreductive $\Q$-group
$$
\SL_2 \times \SL_2(\Zhat)
$$
by a prounipotent group $\U$ whose Lie algebra $\u$ is free as a pronilpotent Lie algebra.\footnote{This is a proalgebraic analogue of a result of Mel'nikov \cite{melnikov} which states that the kernel of $\SL_2(\Z)^\wedge \to \SL_2(\Zhat)$ is a countably generated free profinite group.}

After tensoring with $\Qbar$, the abelianization $H_1(\u)$ of $\u$ is canonically dual (as an $\SL_2\times \SL_2(\Zhat)$-module) to
$$
H^1(\u)\otimes \Qbar = \varinjlim_N \bigoplus_{\chi,m}H^1(\G(N),S^m H)_\chi \boxtimes S^m H
$$
Here $\G(N)$ denotes the full level $N$ subgroup of $\SL_2(\Z)$, $H$ denotes the fundamental representation of $\SL_2$, $S^m H$ its $m$th symmetric power, and $\chi$ an irreducible character of $\SL_2(\Z/N)$. The subscript $\chi$ on the cohomology group, signifies its $\chi$ isotypical summand. The group $\SL_2$ acts in the standard way on the second factor $S^m H$ and $\SL_2(\Zhat)$ acts on the first factor via $\chi$. There is a canonical Zariski dense homomorphism $\SL_2(\Z) \to \cG(\Q)$ whose composition with the projection to $\SL_2(\Q) \times \SL_2(\Zhat)$ is the diagonal inclusion.

The coordinate ring $\cO(\cG)$ of $\cG$ is a commutative Hopf algebra which has Betti, de~Rham and $\ell$-adic \'etale incarnations, which we denote by $\cO(\cG^B)$, $\cO(\cG^\DR)$ and $\cO(\cG^\et_\ell)$, respectively. These are Hopf algebras over $\Q$, $\KDR$ and $\Ql$, respectively. Denote the corresponding affine groups by $\cG^B$, $\cG^\DR$ and $\cG^\et_\ell$. Their coordinate rings are related by natural comparison isomorphisms, which are Hopf algebra isomorphisms.

The Hopf algebras
$$
\cO(\cG^B),\ \cO(\cG^\DR),\ \{\cO(\cG^\et_\ell),\ell \text{ prime}\}
$$
are expected to be direct limits of the Betti, de~Rham and \'etale realizations of a directed system of Voevodsky motives. Since this it is not known at the present time, we will regard $\cO(\cG)$ as a set of compatible Betti, $\KDR$-de~Rham and $\ell$-adic \'etale realizations in the sense of Deligne \cite[\S1]{deligne:p1} and Jannsen \cite[I.2]{jannsen}. Each realization is endowed with a natural weight filtration which correspond under the comparison isomorphisms. Each $\ell$-adic \'etale incarnation has a natural action of the absolute Galois group and the de~Rham realization $\cO(\cG^\DR)$ has a natural Hodge filtration. The de~Rham realization $\cO(\cG^\DR)$ contains all of Manin's iterated Shimura integrals \cite{manin} of all levels, but is much larger. The following theorem justifies this point of view.

\begin{bigtheorem}
\label{thm:rel_comp_mot}
The coordinate ring $\cO(\cG^B)$ of the above relative completion of $\SL_2(\Z)$ carries a natural ind mixed Hodge structure with non-negative weights. Its Hodge filtration corresponds to the Hodge filtration of $\cO(\cG^\DR)$ under the comparison isomorphism. Each of the $\ell$-adic incarnations $\cO(\cG^\et_\ell)$ of $\cO(\cG)$ has a natural action of the absolute Galois group $\Gal(\Qbar/\Q)$. The product, coproduct and antipode respect the Hodge and Galois structures.
\end{bigtheorem}

Here we have suppressed the role of the base point, which sometimes plays an important role. Our default choice is the tangential base point $\partial/\partial q$, which corresponds to the first order smoothing of the nodal cubic given by the Tate curve. Precise statements can be found in Section~\ref{sec:fund_gps+galois}.

\subsection{Class functions and the dual Hecke action}

The ring $\Cl(\cG)$ of class functions on $\cG$ is, by definition, the subspace of $\cO(\cG)$ that consists of those functions that are invariant under conjugation. Its Betti, de~Rham and $\ell$-adic realizations will be denoted $\Cl(\cG^B)$, $\Cl(\cG^\DR)$ and $\Cl(\cG^\et_\ell)$, respectively. Elements of $\Cl(\cG^B)$ restrict to class functions $\SL_2(\Z)\to \Q$. In Section~\ref{sec:conj_invar_ints} we show that $\Cl(\cG)$ is very large. In particular, it contains the subring of conjugation-invariant iterated Shimura integrals, which is not finitely generated. The Hodge and Galois structures on $\cO(\cG)$ described in Theorem~\ref{thm:rel_comp_mot} restrict to $\Cl(\cG)$.

\begin{bigproposition}
\label{prop:class_fns_mot}
The ring $\Cl(\cG^B)$ of class functions on the Betti realization of $\cG$ carries a natural ind mixed Hodge structure. Each of its $\ell$-adic incarnations $\Cl(\cG^\et_\ell)$ has a natural action of $\Gal(\Qbar/\Q)$. Neither of these structures depends on the choice of a base point. In addition, the Adams operators $\psi^m$, $m\in \N$ are morphisms of ind MHS and commute with the Galois action.
\end{bigproposition}

The weight graded quotients of $\Cl(\cG^B)$ are sums of Hodge structures of the form
$$
\big(\Sym^{r_1} V_{f_1} \otimes \dots \otimes \Sym^{r_m} V_{f_m}\big) (d),
$$
where $V_f$ is the $\Q$ Hodge structure of a Hecke eigenform $f$. Consequently, the tannakian subcategory of the category $\MHS$ of mixed Hodge structures generated by $\Cl(\cG^B)$ determines many elements of
$$
\Ext^1_\MHS\big(\Q,(\Sym^{r_1} V_{f_1} \otimes \dots \otimes \Sym^{r_m} V_{f_m}) (d)\big).
$$
The hope is that the extensions that occur conform to Beilinson's conjectures, subject to Brown's constraint \cite[\S17]{brown:mmv}.

Our main result asserts that Hecke correspondences act on $\Cl(\cG)$.

\begin{bigtheorem}
\label{thm:dual_hecke}
Each Hecke correspondence $T_N$ induces a (dual) Hecke operator
$$
\Tdual_N : \Cl(\cG) \to \Cl(\cG)
$$
and each $\ee_p$ induces a dual operator $\edual_p : \Cl(\cG) \to \Cl(\cG)$. More precisely, each $\Tdual_N$ and $\edual_p$ acts compatibly on the Betti, de~Rham and $\ell$-adic realizations of $\Cl(\cG)$. Each $\Tdual_N$ and $\edual_p$ acts on $\Cl(\cG^B)$ as a morphism of mixed Hodge structures and Galois equivariantly on $\Cl(\cG^\et_\ell)$. The dual operators satisfy
\begin{equation}
\label{eqn:dual_idents}
\langle \alpha, \Tdual_N F \rangle = \langle T_N \alpha, F\rangle \text{ and } \langle \alpha, \edual_p F \rangle = \langle \ee_p \alpha, F\rangle
\end{equation}
for all $F\in \Cl(\cG^B)$ and $\alpha \in \blambda(\SL_2(\Z))$. The operators $\Tdual_N$ and $\Tdual_M$ commute when $N$ and $M$ are relatively prime. For all primes $p$ we have the relation
\begin{equation}
\label{eqn:dual_relation}
 \Tdual_p \circ \Tdual_{p^n} = \Tdual_{p^{n+1}} + \edual_p \circ \Tdual_{p^{n-1}}.
\end{equation}
dual to (\ref{eqn:relation}).
\end{bigtheorem}

Define the {\em dual Hecke algebra} $\That^\op$ to be the opposite ring of $\That$. It is generated by the $\Tdual_p$ and $\edual_p$. The previous result says that $\That^\op$ acts compatibly on all realizations of $\Cl(\cG)$.

\subsection{The Hecke action on motivic periods}

The Hecke algebra $\T$ acts on the ring of motivic periods (in the sense of Brown \cite{brown:mot_periods}) of $\Cl(\cG)$. In this case these are, by definition, formal linear combinations of symbols $[\Cl(\cG);\alpha,F]$, where $\alpha \in \lambda(\SL_2(\Z))$ and $F \in \Cl_\KDR(\cG^\DR)$, that are subject to some basic relations. Such periods can be thought of as unevaluated iterated integrals. They form a ring $\cP(\Cl(\cG))$. There is a ring homomorphism
$$
\per : \cP(\Cl(\cG)) \to \C
$$
which takes the motivic period $[\Cl(\cG);\alpha,F]$ to the complex number $\langle F,\alpha \rangle$ obtained by evaluating $F$ on $\alpha$. The Hecke operators act on $\cP(\Cl(\cG))$ by the formula
$$
T_p[\Cl(\cG);\alpha,F] := [\Cl(\cG);T_p(\alpha),F] = [\Cl(\cG);\alpha,\Tdual_p(F)].
$$
The Adams operators $\psi^m$ also act on $\cP(\Cl(\cG))$ via the formula
$$
\psi^m[\Cl(\cG);\alpha,F] := [\Cl(\cG);\alpha^m,F] = [\Cl(\cG);\alpha,\psi^m F].
$$

The most basic elements of $\Cl(\cG^\DR)$ that do not come from class functions on $\SL_2\times \SL_2(\Zhat)$ correspond to the class functions
\begin{equation}
\label{eqn:F_f_phi}
F_{f,\varphi} : \alpha \mapsto \Big\langle \int_\alpha \w_f(\varphi),\alpha \Big\rangle,
\end{equation}
where $f$ is a modular form of $\SL_2(\Z)$ of weight $2n$ and level 1, $\w_f$ is the corresponding $S^{2n-2}H$ valued 1-form on the modular curve, $\varphi : S^{2n-2}H \to \cO(\SL_2)$ is an $\SL_2$-invariant map, where $\SL_2$ acts on its coordinate ring by conjugation, and $\w_f(\varphi) := \varphi\circ\w_f$. One computes the value of this integral by first integrating $\w_f(\varphi)$ over the loop $\alpha$ to obtain an element of $\cO(\SL_2)$ and then evaluates the result on $\alpha$. If $\varphi$ is $\KDR$-de Rham and $f$ is a normalized eignform, then $F_{f,\varphi} \in \Cl(\cG^\DR_\Qbar)$.

More generally, the functions defined by evaluating the cyclic iterated integrals
$$
\alpha \mapsto
\sum_{\sigma \in C_r}
\Big\langle \int_\alpha 
\w_{f_{\sigma(1)}}(\varphi_{\sigma(1)})\,\w_{f_{\sigma(2)}}(\varphi_{\sigma(2)}) \cdots \w_{f_{\sigma(r)}}(\varphi_{\sigma(r)}), \alpha
\Big\rangle
$$
of such forms correspond to elements of $\Cl(\cG^\DR)$, where $C_r$ denotes the cyclic group generated by $(1,2,\dots,r)$. These constructions are discussed in detail in Section~\ref{sec:cyclic}.

Computing the action of Hecke correspondences on the ring motivic periods $\cP(\Cl(\cG))$ appears to be a deep and difficult problem. However, we do work out one class of examples in Section~\ref{sec:hecke_period} where we compute the action of $T_p$ on certain periods of the class functions $F_{f,\varphi}$ defined above in (\ref{eqn:F_f_phi}), where $f$ has weight $2n$ and level 1. We show that if $\alpha \in \SL_2(\Z)$ acts transitively on $\P^1(\Fp)$, then
$$
T_p [\Cl(\cG);\alpha,F_{f,\varphi}] = \frac{\psi^{p+1}}{p^{n-1}(p+1)}[\Cl(\cG);\alpha,F_{T_p(f),\varphi}].
$$
In particular, if $f$ is a Hecke eigenform, then $[\Cl(\cG);\alpha,F_{f,\varphi}]$ will be an ``eigenperiod" of $T_p$ whose ``eigenvalue'' is a multiple of $\psi^{p+1}$. This formula implies that a normalized Hecke eigenform $f =\sum_k a_k q^k$ of weight $2n$ and level 1 can be recovered from the Hecke action on $\cP(\Cl(\cG))$ as
$$
a_p = p^{n-1}(p+1) \frac{T_p[\Cl(\cG);\alpha,F_{f,\varphi}]}{\psi^{p+1}[\Cl(\cG);\alpha,F_{f,\varphi}]}
= p^{n-1}(p+1) \frac{\langle T_p F_{f,\varphi},\alpha \rangle} {\langle F_{f,\varphi},\alpha^{p+1}\rangle }.
$$

\subsection{Mumford--Tate groups}

It is important to know that the mixed Hodge structure on $\Cl(\cG^B)$ is as rich as possible --- that it generates an interesting subcategory of the tannakian category of mixed Hodge structures. The richness of a mixed Hodge structure is measured by its {\em Mumford--Tate group}.

Recall that the Mumford--Tate group $\MT_V$ of a $\Q$ mixed Hodge structure $V$ is the image of the homomorphism
$$
\pi_1(\MHS_\Q,\w) \to \Aut \w(V)
$$
where $\MHS_\Q$ is the category of graded polarizable $\Q$ mixed Hodge structures, $\w$ is the fiber functor that takes a MHS to its underlying $\Q$ vector space $V^B$, and $\pi_1(\MHS_\Q,\w)$ is its tannakian fundamental group.

One goal of Brown's program \cite{brown:mmv} of {\em mixed modular motives} is to understand the Mumford--Tate group of $\cO(\cG^B)$. Proposition~\ref{prop:class_fns_mot} and Theorem~\ref{thm:dual_hecke} imply that the action of the Adams and Hecke operators bound the Mumford--Tate group of $\Cl(\cG)$.

\begin{question}
Is $\MT_{\Cl(\cG^B)}$ the group of automorphisms of $\Cl(\cG^B)$ that commute with all Hecke and Adams operators?
\end{question}

This raises the question of whether the MHS on $\Cl(\cG^B)$ is as rich as the MHS on $\cO(\cG^B)$. In other words, is the natural homomorphism
$$
\MT_{\cO(\cG^B)} \to \MT_{\Cl(\cG^B)}
$$
an isomorphism? One might optimistically conjecture that it is based partly on the fact that, when $\cV$ is the unipotent completion of the fundamental group of a smooth affine curve, the restriction mapping
$$
\MT_{\cO(\cV)} \to \MT_{\Cl(\cV)}
$$
is an isomorphism when restricted to prounipotent radicals (i.e., to $W_{-1}\MT$). This follows from \cite[Cor.~3]{hain:goldman} and \cite[Thm.~4.3]{kawazumi-kuno}.

\subsection{Overview}

The paper is in four parts. The first two are elementary. The action of ``unramified correspondences" on conjugacy classes of fundamental groups is constructed in Part~1 in both the discrete and profinite settings. The action of the Hecke correspondences on $\Z\blambda(\SL_2(\Z))$ is constructed in Part~2 and the basic relations between them are established. In Section~\ref{sec:action} we give an explicit formula for the action of $T_p$ on the elliptic and parabolic conjugacy classes of $\SL_2(\Z)$, compute the minimal polynomial of $\ee_p$, and describe the action of $T_p$ on hyperbolic classes. This part concludes with a discussion of the Hecke action on the class functions of $\SL_2(\Zhat)$.

The main task of Part~3 is to construct the relative unipotent completion of $\SL_2(\Z)$ needed in later sections. In order to accommodate the Hecke action, it is necessary to use a version of relative completion that is larger than the ones used to date in, for example, \cite{brown:mmv,hain:modular}. Its various incarnations --- Betti, de~Rham and $\ell$-adic \'etale --- and the comparison isomorphisms between them are constructed in the final section of this part, Section~\ref{sec:rel_comp_mod_gp}. We construct the natural action of the absolute Galois group on each of its $\ell$-adic incarnations as well as the canonical mixed Hodge structure on its coordinate ring. This part begins with a detailed discussion of the modular curve and local systems and connections over it needed in later sections. For example, we construct the action of the absolute Galois group on the profinite completion of $\SL_2(\Z)^\wedge$ in Section~\ref{sec:fund_gps+galois}. Section~\ref{sec:rel_comp} is a terse review of relative completion in which we make minor improvements to results in the existing literature.

In Part~4, we show that the Hecke correspondences act compatibly on all realizations of the ring of class functions $\Cl(\cG)$ of the relative completion $\cG$ of $\SL_2(\Z)$ constructed in Part~3. We show that the mixed Hodge and Galois structures on $\cO(\cG)$ induce a mixed Hodge structure on $\Cl(\cG^B)$ and a Galois action on each of its $\ell$-adic realizations . Each generalized Hecke operator preserves these structures.

An algebraic description of $\Cl(\cG)$ is given in Section~\ref{sec:conj_invar_ints} as well as techniques for constructing non-trivial elements of it. These techniques are applied in Section~\ref{sec:examples} to construct explicit elements of $\Cl(\cG)$ from modular forms. This section concludes with a computation of the Hecke action on the periods of $\Cl(\cG)$ associated with Hecke eigenforms.

The ring $\Cl(\cG)$ has several natural filtrations, including the weight filtration, the length filtration, the filtration by level, and a filtration coming from the representation theory of $\SL_2$, which we call the {\em modular filtration}. These are described in Section~\ref{sec:filtrations} along with their finiteness properties and their behaviour under the Hecke operators. 

\bigskip

\noindent{\bf Acknowledgments:} This work originated during a sabbatical visit to the Institute for Advanced Study in 2014--15, where I had many inspiring discussions with Francis Brown on periods of iterated integrals of modular forms and his ``multiple modular value'' program. I am especially grateful to him for his interest in this work and the many stimulating discussions we had at IAS and elsewhere. I would like to thank both IAS and Duke University for support during the sabbatical.

Special thanks goes to Pham Tiep who enthusiastically worked on, and almost completely resolved, the problem of determining the characters of $\SL_2(\Z/N)$ that appear in its ``adjoint representation''. His result appears as Theorem~\ref{thm:tiep} and is proved in the appendix. It is needed to ensure that there are large families of class functions on $\SL_2(\Z)$ that come from modular forms of higher level. I am also grateful to Florian Naef who long ago communicated Example~\ref{ex:naef}. This helped point me in the right direction when I was trying to understand the class functions on  relative completions of $\SL_2(\Z)$.

Finally, I am grateful to an anonymous reviewer of an early version for pointing out that the operators $\ee_p$ do indeed satisfy polynomial relations and to the other referees whose numerous helpful comments and corrections resulted in significant improvements to the paper. I would also like to thank V\~{o} Qu\^{o}c Bao for pointing out an error in an earlier proof of Proposition~\ref{prop:ext_DR}.

\section{Preliminaries}
\label{sec:prelims}

\subsection{Path multiplication and iterated integrals}

We use the topologist's convention (which is the opposite of the algebraist's convention) for path multiplication. Two paths $\alpha, \beta:[0,1]\to X$ in a topological space $X$ are composable when $\alpha(1)=\beta(0)$. The product $\alpha \beta$ of two composable paths first traverses $\alpha$ and then $\beta$. With this convention, $\pi_1(X,x)$ acts on the left of a pointed universal covering of $X$ and on the right of the fiber $V_x$ of a local system $\V$ over $X$.

The torsor of paths in a topological space $X$ from $x\in X$ to $y\in X$ will be denoted $\pi(X;x,y)$.

So that the homomorphism from the topological to the \'etale fundamental group of a variety is an isomorphism, we define the \'etale fundamental group of a scheme $X$ with respect to a fiber functor $b$ from \'etale coverings of $X$ to finite sets to the opposite of the usual definition: $\pi_1^\et(X,b)$ is the group of {\em right} automorphisms of $b$.

\subsection{Iterated integrals}

The iterated integral of smooth 1-forms $\w_1,\dots,\w_r$ on a manifold $M$ over a piecewise smooth path $\alpha : [0,1]\to M$ is defined by
$$
\int_\alpha \w_1\w_2\dots \w_r
= \int_{\Delta^r} f_1(t_1)\dots f_r(t_r)dt_1dt_2\dots dt_r.
$$
where $f_j(t)dt = \alpha^\ast \w_j$ and $\Delta^r$ is the ``time ordered'' $r$-simplex
$$
\Delta^r = \{(t_1,\dots,t_r)\in \R^n : 0 \le t_1 \le t_2 \le \cdots \le t_r \le 1\}.
$$
An exposition of the basic properties of iterated integrals can be found
in \cite{chen:bams} and \cite{hain:bowdoin}.

\subsection{Hodge theory}

All mixed Hodge structures will be $\Q$ mixed Hodge structures and graded polarizable unless stated otherwise. The category of {\em graded-polarizable} $\Q$ mixed Hodge structures will be denoted by $\MHS$. The category of graded-polarizable $\R$-mixed Hodge structures will be denoted by $\MHS_\R$. An ind (respectively, pro) MHS is, by definition an object of ind-$\MHS$ (respectively, pro-$\MHS$). We will refer to them as {\em graded polarizable} ind (or pro) MHS.

\subsection{Orbifolds and stacks}
\label{sec:orbifolds_vs_stacks}

By an {\em orbifold} (resp., a {\em complex analytic orbifold}) we mean a stack in the category of topological spaces (resp., complex analytic varieties). In general, the word {\em stack} will refer to a Deligne--Mumford stack. We will typically denote the complex analytic orbifold associated with a stack $X$ defined over a subfield of $\C$ by $X^\an$. The associated orbifold will often be denoted by $X(\C)$.

\subsection{Modular groups}

Suppose that $N \ge 1$. We will use the standard notation
\begin{align*}
\G_0(N) &= \{\gamma \in \SL_2(\Z) : \gamma \text{ is upper triangular mod } N\}
\cr
\G_1(N) &= \{\gamma \in \SL_2(\Z) : \gamma \text{ is upper triangular unipotent mod } N\}
\cr
\G(N) &= \{\gamma \in \SL_2(\Z) : \gamma \text{ is congruent to the identity mod } N\}.
\end{align*}
The orbifolds $\G_0(N)\bbs \h$, $\G_1(N)\bbs \h$ and $\G(N)\bbs \h$ have (orbifold) fundamental groups $\G_0(N)$, $\G_1(N)$ and $\G(N)$, respectively. They are the complex analytic orbifolds associated with the modular curves $Y_0(N)$, $Y_1(N)$ and $Y(N)$, which will be regarded as stacks. Their standard compactifications will be denoted $X_0(N)$, $X_1(N)$ and $X(N)$, respectively.

A convenient model of the homotopy type of the orbifold $\G\bbs \h$, where $\G$ is a subgroup of $\SL_2(\Z)$, is $\G\bs \hhat$, where $\G$ acts diagonally on
\begin{equation}
\label{eqn:hhat}
\hhat := \h \times E\GL_2^+(\Q)
\end{equation}
and $E\GL_2^+(\Q)$ is a space on which the subgroup $\GL_2^+(\Q)$ of elements of $\GL_2(\Q)$ with positive determinant acts freely and properly discontinuously.

\subsection{Number fields}
\label{sect:Qbar}

For us, $\Qbar$ denotes the algebraic closure of $\Q$ in $\C$. The $N$th roots of unity in $\C$ will be denoted $\bmu_N$ and the group of all roots of unity by $\bmu_\infty$. The maximal abelian extension $\Q^\ab$ of $\Q$ is $\Q(\bmu_\infty)$.

\part{Unramified correspondences}

In this part we define unramified correspondences and show that they act on the free abelian group generated by conjugacy classes in the fundamental group of the source of the correspondence. We also give a group-theoretic description of this action. The material in this part is elementary.

\section{Quick review of covering space theory}
\label{sec:cov_sp}

For clarity we give a quick review of the relevant facts we shall need from covering space theory. This will also serve to fix notation. In all discussions of covering spaces and unramified correspondences, the topological spaces will be assumed to be locally contractible. Thus all such discussions apply to manifolds, complex algebraic varieties and to the geometric realizations of simplicial sets.

Suppose that $X$ is a connected topological space and that $\pi: Y \to X$ is a finite covering of degree $d$. Initially we do not assume that $Y$ is connected. Fix a base point $x_o \in X$. For $y\in \pi^{-1}(x_o)$, we shall identify $\pi_1(Y,y)$ with its image under the injection $\pi_\ast : \pi_1(Y,y) \to \pi_1(X,x_o)$.

There is a natural right action
$$
\pi^{-1}(x_o) \times \pi_1(X,x_o) \to \pi^{-1}(x_o)
$$
of $\pi_1(X,x_o)$ on the fiber over $\pi$ over $x_o$. We shall denote it by
$$
(y,\gamma) \mapsto y\cdot \gamma.
$$
It is characterized by the property that $y\cdot \gamma = y'$ if and only the unique lift of a loop in $X$ based at $x_o$ that represents $\gamma$ starts at $y\in\pi^{-1}(x_o)$ ends at $y'$. If $y$ is in the fiber over $x_o$, then
$$
y = y\cdot \gamma \text{ if and only if } \gamma \in \pi_1(Y,y)
$$
and
$$
y' = y\cdot \gamma \text{ implies that } \pi_1(Y,y') = \gamma^{-1} \pi_1(Y,y) \gamma.
$$

Now suppose that $Y$ is connected. Fix a base point $y_o$ of $Y$ that lies over $x_o$. Then $\pi_1(Y,y_o)$ has index $d$ in $\pi_1(X,x_o)$ and the right $\pi_1(X,x_o)$-action on $\pi^{-1}(x_o)$ is transitive. Since $\pi_1(Y,y_o)$ stabilizes $y_o$, there is a natural isomorphism
$$
\pi_1(Y,y_o) \bs \pi_1(X,x_o) \to \pi^{-1}(x_o),\quad \pi_1(Y,y_o) \gamma \mapsto y_o\cdot \gamma.
$$
of right $\pi_1(X,x_o)$-sets.

The kernel of the right action $\pi_1(X,x_o) \to \Aut \pi^{-1}(x_o)$ is the normal subgroup of $\pi_1(X,x_o)$ that corresponds to the Galois closure
$$
\pitilde : Z \to Y \to X
$$
of $\pi$. Set
$$
G = \Aut(Z/X) \text{ and } H = \Aut(Z/Y).
$$

The group $G$ (and hence $H$ as well) acts on the {\em left} on $Z$. The fiber of $\pitilde$ over $x_o$ is a left $G$-torsor. The choice of a base point $z_o$ of $Z$ that lies over $x_o$ trivializes this left $G$-torsor. So we can identify the fiber of $\pitilde$ over $x_o$ with $G$ via the map
$$
G \to \pitilde^{-1}(x_o) \text{ defined by } g \mapsto gz_o
$$
The choice of $z_o$ also determines a surjective homomorphism $\rho_{z_o} : \pi_1(X,x_o) \to G$ which is characterized by the property that
$$
\rho_{x_o}(\gamma) z_o = z_o\cdot \gamma
$$
for all $\gamma \in \pi_1(X,x_o)$. In addition, $\pi_1(Y,y_o)$ is the inverse image of $H$ under $\rho_{x_o}$.

The pointed covering $\pi : (Y,y_o) \to (X,x_o)$ is naturally identified with the covering $H\bs (Z,z_o) \to (X,x_o)$. This means that we can identify the fiber of $\pi : Y \to X$ over $x_o$ with $H\bs G$. The map induces an isomorphism of right $\pi_1(X,x_o)$-sets
$$
H\bs G \to \pi^{-1}(x_o), \quad H\rho(x_o)(\gamma) \mapsto y_o \cdot \gamma.
$$

\section{Pushforward and pullback}

As in the introduction, the set of homotopy classes of (unbased) loops $S^1 \to X$ in a topological space $X$ will be denoted by $\lambda(X)$. The set of conjugacy classes of a group $\G$ will be denoted by $\blambda(\G)$. If $X$ is path connected, then $\lambda(X) = \blambda(\pi_1(X,x_o))$ for all $x_o \in X$. If $\G$ is a discrete group, then $\blambda(\G) = \lambda(B\G)$, where $B\G$ denotes the classifying space of $\G$.

Suppose that $\kk$ is a commutative ring. The free $\kk$-module generated by $\lambda(X)$ will be denoted by $\kk\lambda(X)$ and the free $\kk$-module generated by $\blambda(\G)$ will be denoted by $\kk\blambda(\G)$. In both cases, the conjugacy class of the identity will be denoted by $1$.

Every continuous map $f : X \to X'$ between topological spaces induces a function
$$
f_\ast : \lambda(X) \to \lambda(X')
$$
and thus a $\kk$-module map $\kk\lambda(X) \to \kk\lambda(X')$. It takes the free homotopy class of the loop $\alpha : S^1 \to X$ to the free homotopy class of the loop $f\circ \alpha : S^1 \to X'$.

Likewise, every group homomorphism $\phi : \G \to \G'$ induces a function $\phi_\ast : \blambda(\G) \to \blambda(\G')$ and a $\kk$-module map $\kk\blambda(\G) \to \kk\blambda(\G')$. The main task in this section is to show that, when $\pi : Y \to X$ is a finite unramified covering, there is a pullback map
$$
\pi^\ast : \kk \lambda(X) \to \kk\lambda(Y).
$$
Applying this to the special case of the covering $B\G' \to B\G$ associated with the inclusion of a finite index subgroup $\G'$ of a discrete group $\G$ implies that there is a pullback map
$$
\kk\blambda(\G') \to \kk\blambda(\G).
$$
We will give an algebraic formula for this map, which will imply that the pullback map is also defined when $\G'$ is an open subgroup of a profinite group $\G$.

\subsection{Pullback of loops along unramified coverings}
\label{sec:pullback_def}

Suppose that $\pi : Y \to X$ is an unramified covering of finite degree. For the time being, we will not assume that either space is connected.

Suppose that $\alpha : S^1 \to X$ is a loop in $X$. The pullback
$$
\xymatrix{
Y\times_X S^1 \ar[r]^\alphatilde \ar[d] & Y \ar[d]^\pi \cr
S^1 \ar[r]_\alpha & X
}
$$
of $\pi$ along $\alpha$ is a covering of $S^1$ and thus a disjoint union of oriented circles
$$
Y\times_X S^1 = \bigsqcup_{s\in\sS_\alpha} S^1_s
$$
whose components are indexed by a finite set $\sS_\alpha$. Denote the restriction of $\alphatilde$ to the component $s\in \sS_\alpha$ by $\alphatilde_s$. The homotopy lifting properties of unramified coverings implies that its class in $\lambda(Y)$ depends only on the class of $\alpha$ in $\lambda(X)$. The pullback of $\alpha$ is defined by
$$
\pi^\ast \alpha = \sum_{s\in\sS_\alpha} \alphatilde_s.
$$

The following basic property of pullback is an immediate consequence of the definition.

\begin{lemma}
\label{lem:comp}
If $\pi_X : Y \to X$ and $\pi_Y: Z \to Y$ are finite coverings, then the diagram
$$
\xymatrix{
\Z\lambda(Z) & \Z\lambda(Y)\ar[l]_(.45){\pi_Y^\ast} \cr
& \Z\lambda(X) \ar[ul]^{(\pi_X\pi_Y)^\ast} \ar[u]_{\pi_X^\ast}
}
$$
commutes.
\end{lemma}

Denote the connected $r$-fold covering map of the circle by $\pi_r : S^1 \to S^1$. Orient the circles so that it is orientation preserving.

\begin{lemma}
\label{lem:pullback}
The pullback of the covering $\pi_d$ along $\pi_e$ is a disjoint union of $g := \gcd(d,e)$ circles. In the pullback diagram
$$
\xymatrix{
\bigsqcup_g S^1 \ar[r] \ar[d] & S^1 \ar[d]^{\pi_d}\cr
S^1 \ar[r]_{\pi_e} & S^1
}
$$
the top horizontal map has degree $e/g$ on each component and the left hand vertical map has degree $d/g$ on each component.
\end{lemma}

\begin{proof}
This is an exercise in covering space theory using the elementary fact that $d\Z \cap e\Z = g\Z$.
\end{proof}

As an immediate corollary, we obtain a formula for the pullback along $\pi_d : S^1 \to S^1$ of a multiple of the positive generator $\sigma$ of $\pi_1(S^1)$. Denote the positive generator of the fundamental group of the domain of $\pi_d$ by $\mu$.

\begin{corollary}
\label{cor:pullback}
For all integers $n$, we have
$$
\pi^\ast (\sigma^n) = \gcd(d,n) \mu^{n/\gcd(d,n)} \text{ and } \pi_\ast (\mu) = \sigma^d.
$$
Consequently, $\pi_\ast \pi^\ast(\sigma^n) = \gcd(d,n) \sigma^{nd/\gcd(d,n)}$.
\end{corollary}

\subsection{An algebraic description of the pullback map}
\label{sec:alg_pullback}

When $X$ and $Y$ are connected, the pullback map admits a description in terms of the induced map on fundamental groups. Choose a base point $x_o$ of $X$ and $y_o$ of $Y$ that lies over $x_o$. We will use the notation and conventions established in Section~\ref{sec:cov_sp}.

Suppose that $\alpha \in \pi_1(X,x_o)$. We will abuse notation and also denote its conjugacy class by $\alpha$. The subgroup $\langle \alpha \rangle$ of $\pi_1(X,x_o)$ generated by $\alpha$ acts on $\pi^{-1}(x_o)$ on the right: $\alpha$ takes $y\in \pi^{-1}(x_o)$ to $y'$ if and only if the unique lift of $\alpha$ to a path in $Y$ that starts at $y$ ends at $y'$.

The set $\sS_\alpha$ above is the set of $\langle \alpha \rangle$-orbits $\pi^{-1}(x_o)/\langle \alpha \rangle$. For each $s\in \pi^{-1}(x_o)/\langle \alpha \rangle$, choose $\gamma_s \in \pi_1(X,x_o)$ such that $y_o\cdot \gamma_s$ is in the orbit. The cardinality $d_s$ of the corresponding orbit is given by
\begin{align}
\label{eqn:def_ds}
d_s
&= \min\{k \in \Z : k > 0,\ y_o\cdot \gamma_s \alpha^k = y_o \cdot \gamma_s\}\cr
&= \min\{k \in \Z : k > 0,\ \alpha^k \in \gamma_s^{-1}\pi_1(Y,y_o) \gamma_s\}.
\end{align}
The $\langle \alpha \rangle$-orbit that contains $y_o\cdot\gamma_s$ is illustrated in Figure~\ref{fig:orbit}. It determines the loop
$$
\alphatilde_{s,1} \alphatilde_{s,2} \cdots \alphatilde_{s,d_s} \in \lambda(Y),
$$
where $\alphatilde_{s,j}$ is the unique lift of $\alpha$ to $Y$ that starts at $y_o\cdot\gamma_s\alpha^{j-1}$. This is the lift of the loop $\gamma_s^{-1} \alpha^{d_s} \gamma_s$ in $X$ to a loop in $Y$ based at $y_o\cdot \gamma_s$. 
\begin{figure}[ht!]
\begin{tikzpicture}[scale=1.5]
\filldraw	(1,0)		circle	[radius=1pt]
			(0.5,0.87)	circle	[radius=1pt]
			(-0.5,0.87)	circle	[radius=1pt]
			(0.5,-0.87)	circle	[radius=1pt]
			(-1,0)		circle	[radius=1pt]
			(0.5,0.87)	circle	[radius=1pt]
			(-0.5,-0.87)circle	[radius=1pt];
\draw[decoration={markings,mark=at position 0.5 with {\arrow[scale=2]{>}}},postaction={decorate}]
	(1,0) arc [start angle = 0, end angle = 60, radius=1];
\draw[decoration={markings,mark=at position 0.5 with {\arrow[scale=2]{>}}},postaction={decorate}]
	(0.5,0.87) arc [start angle = 60, end angle = 120, radius=1];
\draw[decoration={markings,mark=at position 0.5 with {\arrow[scale=2]{>}}},postaction={decorate}]
	(-1,0) arc [start angle = 180, end angle = 240, radius=1];
\draw[decoration={markings,mark=at position 0.5 with {\arrow[scale=2]{>}}},postaction={decorate}]
	(0.5,-0.87) arc [start angle = -60, end angle = 0, radius=1];
\draw[dotted] (-0.5,0.87) arc [start angle = 120, end angle = 180, radius = 1];
\draw[dotted] (-0.5,-0.87) arc [start angle = 240, end angle = 300, radius = 1];
\node[right] at (1.0,0) {$y_o\cdot\gamma_s=y_o \cdot \gamma_s\alpha^{d_s}$};
\node[above right] at (0.5,0.87) {$y_o \cdot \gamma_s\alpha$};
\node[left] at (-1,0) {$y_o \cdot \gamma_s\alpha^{j-1}$};
\node[below left] at (-0.5,-0.87) {$y_o \cdot\gamma_s\alpha^j$};
\node[above left] at (-0.5,0.87) {$y_o \cdot \gamma_s\alpha^2$};
\node[below right] at (0.5,-0.87) {$y_o \cdot \gamma_s\alpha^{d_s-1}$};
\end{tikzpicture}
\caption{The $\langle \alpha \rangle$-orbit that contains $y_o\cdot\gamma_s$}
\label{fig:orbit}
\end{figure}
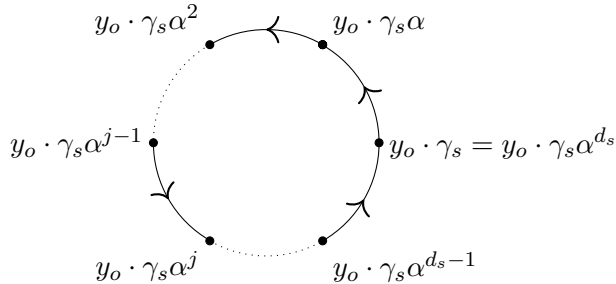

With this notation, the formula for the pullback of $\alpha$ is
\begin{equation}
\label{eqn:pullback}
\pi^\ast(\alpha) = \sum_{s\in\sS_\alpha} \alphatilde_{s,1} \alphatilde_{s,2} \cdots \alphatilde_{s,d_s}
= \sum_{s\in \sS_\alpha} \gamma_s \alpha^{d_s} \gamma_s^{-1} \in \Z\lambda(Y).
\end{equation}

This formula can be converted into a group-theoretic description of the pullback map. The right action of $\pi_1(X,x_o)$ on $\pi^{-1}(x_o)$ induces an isomorphism of right $\pi_1(X,x_o)$-sets
$$
\pi_1(Y,y_o)\bs \pi_1(X,x_o) \to \pi^{-1}(x_o).
$$
The set of $\langle \alpha \rangle$ orbits is thus the double coset space
$$
\sS_\alpha = \pi_1(Y,y_o)\bs \pi_1(X,x_o)/\langle \alpha \rangle
$$
and $\{\gamma_s : s\in\sS_\alpha\}$ is a set of double coset representatives. We conclude that $\pi^\ast(\alpha)$ is the image of
\begin{equation}
\label{eqn:pushforward}
\sum_{s\in \pi_1(Y,y_o)\bs \pi_1(X,x_o)/\langle \alpha \rangle} \gamma_s \alpha^{d_s} \gamma_s^{-1} \in \Z\pi_1(Y,y_o)
\end{equation}
in $\Z\lambda(Y)$, where $d_s$ is defined by (\ref{eqn:def_ds}).

\subsection{The pullback map for finite index subgroups}
\label{sec:gp_pullback}

This algebraic description of the pullback map allows us to extend the definition of pullback to the case of the inclusion $j : \G' \hookrightarrow \G$ of a finite index subgroup. There is no restriction on the group $\G$; it can be discrete, profinite, algebraic or a Lie group. However, when $\G$ is an algebraic or topological group, we will require that $\G'$ be an open subgroup. We will need the profinite case when discussing Galois equivariance of the Hecke action on conjugacy classes of various profinite completions of $\SL_2(\Z)$.

We now define the pullback map $j^\ast : \Z\blambda(\G) \to \Z\blambda(\G')$ using the group-theoretical description of the pullback (\ref{eqn:pushforward}) given above. To simplify and unify the discussion, we will regard a discrete group $\G$ to be a topological group with the discrete topology. In all cases, the set $\G'\bs \G$ is finite (and discrete as $\G'$ has finite index and is open in $\G$).

Suppose that $\alpha \in \G$. Denote by $\langle \alpha \rangle$ the cyclic subgroup of $\G$ generated by $\alpha$. Set
$$
\sS_\alpha = \G'\bs \G /\langle \alpha \rangle.
$$
This is a discrete finite space. Choose a representative $\gamma_s \in \G$ of each double coset $s$. For each $s$, set
$$
d_s = \min\{k \in \Z : k > 0,\ \alpha^k \in \gamma_s^{-1}\G' \gamma_s\}.
$$
This is well defined as $\G'$ has finite index in $\G$. Define $j^\ast(\alpha) \in \Z\blambda(\G')$ to be the image of
$$
\sum_{s\in \sS_\alpha} \gamma_s \alpha^{d_s} \gamma_s^{-1} \in \Z \G'
$$
under the quotient map $\Z[\G'] \to \Z\blambda(\G')$.

This formula implies that the construction is natural in the following sense.

\begin{proposition}
Suppose that we have a commutative diagram
$$
\xymatrix{
\G' \ar[d]_{\phi'} \ar[r]^j & \G \ar[d]^\phi \cr
\hat{\G}' \ar[r]^{\hat{\jmath}} & \hat{\G}
}
$$
of groups where $j$ and $\hat{\jmath}$ are inclusions of finite index open subgroups. If the induced map $\phi : \G'\bs \G \to \hat{\G}'\bs \hat{\G}$ is a bijection, then the diagram
$$
\xymatrix{
\Z\blambda(\G) \ar[d]_{\phi'_\ast} \ar[r]^{j^\ast} & \Z\blambda(\G') \ar[d]^{\phi_\ast} \cr
\Z\blambda(\hat{\G}) \ar[r]^{\hat{\jmath}^\ast} & \Z\blambda(\hat{\G}')
}
$$
commutes.
\end{proposition}

One important case for us is where $\G$ and $\G'$ are discrete and $\phi$ and $\phi'$ are profinite completion. In this case $\hat{\jmath}$ is injective as $\G'$ has finite index in $\G$.\footnote{Note, however, that profinite completion is not, in general, left exact. See \cite{deligne:Sp_hat}.}

\section{Unramified correspondences act on conjugacy classes}

In this section, we define {\em unramified correspondences} and show that they act on conjugacy classes. Later we will see that Hecke correspondences are unramified correspondences and therefore act on $\Z\blambda(\SL_2(\Z))$.

\subsection{Unramified correspondences}

By an {\em unramified correspondence} between two locally contractible topological spaces $X$ and $Y$, we mean a (not necessarily connected) space $U$ and maps
\begin{equation}
\label{eqn:corresp}
\xymatrix{
& U \ar[dl]_f\ar[dr]^g \cr
X && Y
}
\end{equation}
where $f$ is a covering projection of finite degree. We will call $X$ the {\em source} and $Y$ the {\em target} of the correspondence. We will typically denote such a correspondence by a roman letter, such as $F$.

The composition of this correspondence with the unramified correspondence $G$ given by the diagram
\begin{equation}
\label{eqn:corresp2}
\xymatrix{
& V \ar[dl]_h\ar[dr]^k \cr
Y && Z
}
\end{equation}
is the unramified correspondence $G\circ F$ given by the diagram
$$
\xymatrix{
& U\times_Y V \ar[dl]_\phi\ar[dr]^\kappa \cr
X && Z
}
$$
where $\phi$ and $\kappa$ are the compositions of the two natural projections from $U\times_Y V$ to $U$ and $V$ with $f$ and $k$ respectively. The map $\phi$ is a covering map of finite degree as it is the composition of $f$ with the pullback of the covering map $h$ along $g$.

Correspondences with the same source and target can be added. The sum of the correspondence (\ref{eqn:corresp}) with the correspondence
$$
\xymatrix{
& U' \ar[dl]_{f'}\ar[dr]^{g'} \cr
X && Y
}
$$
is the correspondence
$$
\xymatrix{
& U \sqcup U' \ar[dl]_{f \sqcup f'}\ar[dr]^{g \sqcup g'} \cr
X && Y
}
$$
where $\sqcup$ denotes disjoint union. The unramified correspondences from $X$ to $Y$ form an abelian monoid which we denote by $\Hom^+(X,Y)$. The multiplication
$$
\circ : \Hom^+(Y,Z) \times \Hom^+(X,Y) \to \Hom^+(X,Z)
$$
is bilinear. Additive inverses can be added formally to $\Hom^+(X,Y)$ in the standard way to obtain an abelian group $\Hom(X,Y)$. The multiplication map extends to a bilinear mapping $\Hom(Y,Z) \times \Hom(X,Y) \to \Hom(X,Z)$.

When both maps $f$ and $g$ in the correspondence $F$ depicted in (\ref{eqn:corresp}) are finite unramified coverings, we can reverse the roles of the source and target to obtain a new unramified correspondence, which we call the {\em adjoint} of $F$ and denote by $F^\vee$. For example, the adjoint $F^\vee$ of (\ref{eqn:corresp}) is
$$
\xymatrix{
& U \ar[dl]_g\ar[dr]^f \cr
Y && X
}
$$
If both projections in the correspondence $G$ above are also finite unramified coverings, it and $G\circ F$ will have adjoints, and these will satisfy
$$
(G\circ F)^\vee = F^\vee \circ G^\vee.
$$

\subsection{Unramified correspondences act on $\Z\lambda(X)$}

It is standard that correspondences act on sheaves, homology, cohomology, etc. For example, the correspondence (\ref{eqn:corresp}) acts on covariant objects by $g_\ast \circ f^\ast$ and on contravariant objects by $f_\ast \circ g^\ast$. {\em Unramified} correspondences act on $\Z\lambda(\blank)$ as we can pullback conjugacy classes along finite unramified coverings:

The unramified correspondence (\ref{eqn:corresp}) induces the map
$$
g_\ast\circ f^\ast : \Z\lambda(X) \to  \Z\lambda(Y).
$$

\begin{proposition}
\label{prop:comp}
The map
$$
\Hom(X,Y) \to \Hom(\Z\lambda(X),\Z\lambda(Y))
$$
is compatible with multiplication of unramified correspondences.
\end{proposition}

\begin{proof}
We will reduce this to the case where the spaces are circles. We take our correspondences to be (\ref{eqn:corresp}) and (\ref{eqn:corresp2}). Using the definition of pushforward and pullback and the linearity of the action of correspondences under disjoint union of coverings, we can (and will) reduce to the case where $X$,$Y$,$Z$, $U$ and $V$ are each just one copy of $S^1$. Set
$$
W = U\times_Y V.
$$
This will be a disjoint union of circles $W_j$. Denote the restriction of $r$ to $W_j$ by $r_j$. We can further assume that $\alpha \in \lambda(X)$ is the positive generator of $\pi_1(X) = \pi_1(S^1)$.

The two correspondences and their composition are summarized in the diagram
$$
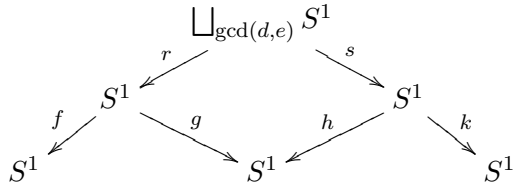
\xymatrix@C=16pt @R=12pt{
&& \bigsqcup_{\gcd(d,e)} S^1 \ar[dl]_(.6)r\ar[dr]^(.55)s \cr
& S^1 \ar[dl]_f \ar[dr]^g  && S^1 \ar[dl]_h \ar[dr]^k  \cr
S^1 && S^1 && S^1
}
$$

Pick a generator of the fundamental group of each circle. We can choose these such that the degrees of $f$, $h$ and $r$ are positive. Denote the chosen generators of the various circles according to the following table:
$$
\begin{array}{c|cccccc}
\text{space} & X & Y & Z & U & V & W_j \cr
\hline
\text{generator} & \alpha & \beta & \gamma & \mu & \nu & \w_j
\end{array}
$$
Denote the degrees (with respect to these chosen generators) of the maps in the diagram by
$$
\begin{array}{c|ccccccc}
\text{map} & f & g & h & k \cr
\hline
\text{degree} & p & d & e & q 
\end{array}
$$
Then, by Lemma~\ref{lem:pullback}, the degree of the restriction of $p$ to each component of $W_j$ of $W$ is $e/gcd(d,e)$ and the degree of the restriction of $q$ to each $W_j$ is $d/gcd(d,e)$. Then
\begin{multline*}
k_\ast h^\ast g_\ast f^\ast (\alpha) = k_\ast h^\ast g_\ast(\mu) = k_\ast h^\ast (\beta^d) = \gcd(d,e)\,k_\ast(\nu^{d/\gcd(d,e)})\cr = \gcd(d,e)\,\gamma^{dq/\gcd(d,e)}.
\end{multline*}
On the other hand, using Lemma~\ref{lem:comp}, we have
\begin{multline*}
(ks)_\ast(fr)^\ast(\alpha) = k_\ast s_\ast r^\ast (\mu) = k_\ast s_\ast (\zeta_1 + \dots + \zeta_g) = \gcd(d,e)\,k_\ast(\nu^{d/\gcd(d,e)})
\cr
= \gcd(d,e)\,\gamma^{dq/\gcd(d,e)}.
\end{multline*}
\end{proof}

\section{The profinite case}

The analogue of the unramified correspondence (\ref{eqn:corresp}) in the profinite case, is a diagram
\begin{equation}
\label{eqn:gp_corresp}
\xymatrix{
& \{\G''_k\}_{k\in K} \ar[dl]_{\{\phi_k\}}\ar[dr]^{\{\psi_k\}} \cr
\G && \G'
}
\end{equation}
where $\G$ and $\G'$ are profinite groups, $\{\G''_k\}_{k\in K}$ is a finite set of open subgroups of $\G$ with corresponding inclusions $\phi_k : \G''_k \hookrightarrow \G$, and where $\psi_k : \G_k'' \to \G'$, $k\in K$ are continuous group homomorphisms. Such a diagram will be called a {\em group correspondence}. It can be regarded as the formal sum of the {\em basic} group correspondences
$$
\xymatrix{
& \G_k'' \ar[dl]_{\phi_k}\ar[dr]^{\psi_k} \cr
\G && \G'
}
$$
indexed by $k\in K$. Two basic correspondences
$$
\xymatrix{
& \G'' \ar[dl]_{\phi}\ar[dr]^{\psi} \cr
\G && \G'
}
\quad
\xymatrix{
& \G'' \ar[dl]_{\phi'}\ar[dr]^{\psi'} \cr
\G && \G'
}
$$
are {\em equivalent} if there are $\gamma \in \G$ and $\mu\in \G'$ such that
$$
\phi' = (\text{conjugation by $\gamma$}) \circ \phi \text{ and } \psi' = (\text{conjugation by $\mu$}) \circ \psi.
$$
The correspondence (\ref{eqn:gp_corresp}) induces the map
$$
\sum_{k\in K} (\psi_k)_\ast\circ \phi_k^\ast : \Z\blambda(\G) \to \Z\blambda(\G').
$$
where the pullback maps $\phi_k^\ast$ are defined as in Section~\ref{sec:gp_pullback}. The induced map depends only on the equivalence classes of its basic constituents.

As in the discrete case, we can formally define $\Hom(\G,\G')$, the abelian group of equivalence classes of group correspondences from $\G$ to $\G'$. Composition of equivalence classes of group correspondences, which is defined in the obvious way, is bilinear.

\subsection{Galois equivariance}
\label{sec:gal_equiv}

Suppose that $X$ is a geometrically connected scheme over a field $K$. Suppose that $\xbar$ is a geometric point of $X$ that lies over a $K$-rational point. Denote the absolute Galois group $\Gal(\Kbar/K)$ of $K$ by $G_K$. It acts continuously on $\pi_1^\et(X\times_K \Kbar,\xbar)$ and therefore on the profinite set $\blambda(\pi_1^\et(X\times_K \Kbar,\xbar))$. Define $\lambda^\et(X) = \blambda(\pi_1^\et(X\times_K \Kbar,\xbar))$.

\begin{lemma}
\label{lem:G-equiv}
The action of $G_K$ on $\pi_1^\et(X\times_K \Kbar,\xbar)$ induces an action on $\lambda^\et(X)$. This action does not depend on the choice of the base point $\xbar$. \qed
\end{lemma}

This Galois action can also be constructed using the canonical outer Galois action $G_K \to \Out \pi_1^\et(X\times_K \Kbar,\xbar)$, where $\xbar$ is any geometric point of $X$.  

An unramified correspondence of $K$-schemes is defined to be a correspondence
\begin{equation}
\label{eqn:et_corresp}
\xymatrix{
& U \ar[dl]_f\ar[dr]^g \cr
X && Y
}
\end{equation}
of $K$-schemes where $f$ is \'etale. Denote it by $F$. Applying the construction above to the geometric \'etale fundamental groups, we see that such a correspondence induces a map $F_\ast : \Z\lambda^\et(X) \to \Z\lambda^\et(Y)$.

\begin{theorem}
\label{thm:gal_equiv}
The map $F_\ast : \Z\lambda^\et(X) \to \Z\lambda^\et(Y)$ induced by the \'etale correspondence (\ref{eqn:et_corresp}) is $G_K$-equivariant.
\end{theorem}

\begin{proof}
We will assume that $Y$ is geometrically connected and that $U$ is irreducible over $K$. The general case is left to the reader. Write
$$
U\times_K\Kbar = \bigsqcup_{c\in C} U_c
$$
where each $U_c$ is a connected $\Kbar$-scheme. The set $C$ is a transitive finite $G_K$ set. After changing base to $\Kbar$, the correspondence $F$ becomes a sum of the correspondences $F_c$
$$
\xymatrix{
& U_c \ar[dl]_{f_c}\ar[dr]^{g_c} \cr
X\times_K\Kbar && Y\times_K\Kbar
}
$$
where $f_c$ and $g_c$ denote the restrictions of $f$ and $g$ to $U_c$. Consequently
$$
F_\ast = \sum_{c\in C} (F_c)_\ast :  \Z\lambda^\et(X) \to \Z\lambda^\et(Y).
$$

Fix $\sigma \in G_K$. Since the diagram
$$
\xymatrix{
X\times_K \Kbar \ar[d]_{\id\times \sigma} & \ar[l]_(.4){f_c} U_c \ar[d]^\sigma \ar[r]^(.45){g_c} & Y\times_K \Kbar \ar[d]^{\id\times \sigma} \cr
X\times_K \Kbar & \ar[l]^(.4){f_{\sigma(c)}} U_{\sigma(c)} \ar[r]_{g_{\sigma(c)}} & Y\times_K \Kbar
}
$$
commutes, it follows that $\sigma\circ(F_c)_\ast = \big(F_{\sigma(c)}\big)_\ast \circ \sigma$, so that
$$
F_\ast \circ \sigma = \sum_{c\in C} (F_c)_\ast\circ \sigma = \sum_{c\in C} \sigma\circ \big(F_{\sigma(c)}\big)_\ast = \sigma \circ \sum_{c\in C} \big(F_c\big)_\ast = \sigma \circ F_\ast.
$$
\end{proof}

\subsection{A comparison theorem}

Suppose that $K$ is a subfield of $\C$. Fix an embedding $\Kbar \hookrightarrow \C$. If $Z$ is a scheme over $K$ and $\zbar \in Z(\C)$, there is a natural homomorphism
$$
\pi_1(Z(\C),\zbar) \to \pi_1^\et(Z\times_K \Kbar,\zbar)
$$
where we regard the complex points $Z(\C)$ of $Z$ as a topological space via the complex topology. This homomorphism becomes an isomorphism after taking the profinite completion of the topological fundamental group of $Z(\C)$. There is therefore a natural comparison map
$$
c_Z : \lambda(Z(\C)) \to \lambda^\et(Z).
$$

The unramified correspondence (\ref{eqn:et_corresp}), denoted $F$, induces the unramified correspondence
\begin{equation}
\label{eqn:top_corresp}
\xymatrix{
& U(\C) \ar[dl]_f\ar[dr]^g \cr
X(\C) && Y(\C)
}
\end{equation}
of complex analytic varieties which we denote by $F^\an$.

\begin{proposition}
\label{prop:lambda_comp}
The diagram
$$
\xymatrix{
\Z\lambda(X(\C))\ar[d]_{c_X} \ar[r]^{F^\an_\ast} & \Z\lambda(Y(\C))\ar[d]^{c_Y} \cr
\Z\lambda^\et(X) \ar[r]^{F_\ast} & \Z\lambda^\et(Y)
}
$$
commutes.
\end{proposition}

\section{The dual action on class functions}
\label{sec:dual_action}

Suppose that $\G$ is a topological group (e.g., a discrete or profinite group) and that $\kk$ is a commutative topological ring such as $\Z$ or a field of any characteristic endowed with the discrete topology. Denote the set of continuous functions $\G \to \kk$ that are constant on each conjugacy class of $\G$ by $\Cl_\kk(\G)$. This definition will be extended to affine groups schemes in Section~\ref{sec:class_fns}.

When $\G$ is finite and $\kk$ is a field, $\Cl_\kk(\G)$ is spanned by the characters of representations. This is typically not the case when $\G$ is a lattice in an algebraic group, as can be seen by considering the case of a lattice in a non-abelian unipotent $\Q$-group $U$. In this case, the non-constant characters are pullbacks of characters of finite quotients of $\G$. Class functions on $\G$ that are not characters can be obtained by restricting non-constant class functions on $U$ to $\G$. For a general construction of class functions on unipotent groups, see Remark \ref{rem:class_fns}(\ref{item:unipt}).

Denote the group correspondence (\ref{eqn:gp_corresp}) by $F$.

\begin{proposition}
For all commutative rings $\kk$, the group correspondence $F$ induces a function
$$
\check{F}: \Cl_\kk(\G') \to \Cl_\kk(\G).
$$
It is defined by
$$
\langle \check{F}(\chi),\gamma \rangle = \langle \chi, F(\gamma) \rangle
$$
for all $\gamma \in \blambda(\G)$ and $\chi \in \Cl_\kk(\G')$.
\end{proposition}

\begin{proof}
It suffices to consider the case where $F$ is an elementary group correspondence
$$
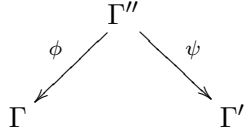
\xymatrix{
& \G'' \ar[dl]_{\phi}\ar[dr]^{\psi} \cr
\G && \G'
}
$$
We have to show that $\phi_\ast\circ\psi^\ast$ takes continuous class functions on $\G'$ to continuous class functions on $\G$. It is clear that $\psi^\ast$ takes continuous class functions on $\G'$ to continuous class functions on $\G''$. So we have to show that $\phi_\ast$ takes continuous class functions on $\G''$ to continuous class functions on $\G$.

Suppose that $\chi : \Gamma'' \to \kk$ is a continuous class function. This means that it takes the constant value $\chi(K)$ on some finite index subgroup $K$ of $\G''$. Let $N$ be the intersection of the conjugates of $\G''$ in $\G$. It is a finite index normal subgroup of $\G$. If $\gamma \in N$, then $\phi^\ast(\gamma) = d\gamma \in \blambda(\G'')$, where $d$ is the index of $\G''$ in $\G$. Consequently, if $\gamma \in K\cap N$ then
$$
\langle \phi_\ast \chi ,\gamma \rangle = \langle \chi, \phi^\ast \gamma \rangle = d \langle \chi, \gamma \rangle = d\chi(K).
$$
Thus $\phi_\ast \chi$ is constant on $K\cap N$. Since $K\cap N$ has finite index, $\phi_\ast \chi$ is continuous.
\end{proof}

If a group $G$ acts on the correspondence $F$ as in the setup of Lemma~\ref{eqn:gp_corresp}, then it acts on $\Cl_\kk(\G)$ and $\Cl_\kk(\G')$. Lemma~\ref{eqn:gp_corresp} then implies that $\check{F}$ is $G$-equivariant.

We can now apply this in the setup of Section~\ref{sec:gal_equiv}.

\begin{corollary}
The \'etale correspondence (\ref{eqn:et_corresp}) induces a $G_K$-invariant linear function
$$
\Cl_\kk(\pi_1^\et(Y\times_K \Kbar)) \to \Cl_\kk(\pi_1^\et(X\times_K \Kbar)).
$$
\end{corollary}

\part{The Hecke action on conjugacy classes of $\SL_2(\Z)$}

In this part we define an action of each Hecke correspondence $T_N$ on $\Z\blambda(\SL_2(\Z))$ and determine the basic relations that hold between them. This part is mainly elementary topology. All modular curves in this part will be regarded as complex analytic orbifolds. In particular, $\M_{1,1}$ will denote the complex analytic orbifold $\SL_2(\Z)\bbs \h$.

\section{Conjugacy classes in $\SL_2(\Z)$}
\label{sec:conj_cl}

As remarked in the introduction, the map $\blambda(\SL_2(\Z)) \to \blambda(\PSL_2(\Z))$ is 2 to 1. The fibers of this map are $\{\gamma,-\gamma\}$. Except when $\tr(\gamma) = 0$ (in which case $\gamma$ has order 4), these are distinguished by the trace map.

If $X$ is a compact hyperbolic manifold, then $\lambda(X)-\{1\}$ can be identified with the closed, oriented (not necessarily prime) geodesics in $X$. If $X$ is a complete, connected hyperbolic surface with finitely generated homology, then
\begin{multline*}
\lambda(X) = \{1\} \cup \{\text{closed, oriented geodesics in $X$}\}
\cr
 \cup \{\text{powers of oriented horocycles}\}.
\end{multline*}

The modular curve $\M_{1,1}$ is the orbifold quotient $\SL_2(\Z)\bbs \h$ of the upper half plane $\h$.\footnote{We use the convention of \cite[\S3]{hain:lectures}. With this convention, the orbifold fundamental group of $\G\bbs \h$ is $\G$. A closed geodesic on $\G\bbs \h$ will be a $\G$-invariant geodesic in $\h$ and two distinct points on it that lie in the same $\G$-orbit.}  Since $-\id$ acts trivially on $\h$, we also have the orbifold $\PSL_2(\Z)\bbs \h$. In both cases the orbifold has a hyperbolic metric.

Each element of $\SL_2(\C)$ that is not a scalar matrix has one or two fixed points on $\P^1(\C)$. These are its projectivized eigenspaces, where $\SL_2(\Z)$ acts on $\C^2$ in the obvious way. Except for $\pm \id$, all elements $\gamma$ of $\SL_2(\Z)$ have either one fixed point (necessarily in $\P^1(\Q)$), or two real fixed points, or a complex conjugate pair of non-real fixed points. These correspond to the three types of non-trivial elements $\gamma$ of $\SL_2(\Z)-\{\pm \id\}$:
\smallskip

\begin{tabular}{>{\bf}lll}
elliptic & $|\tr(\gamma)| \in \{0,1\}$ & finite order, fix a point in $\h$ \\
parabolic & $|\tr(\gamma)| = 2$ & fix a single fixed point that lies in $\P^1(\Q)$\\
hyperbolic & $|\tr(\gamma)| > 2$ & fix two distinct points of $\P^1(\R)$.
\end{tabular}
\newline
This decomposition of $\SL_2(\Z)$ gives a partition
$$
\blambda(\SL_2(\Z)) = \{\pm \id\} \sqcup \blambda(\SL_2(\Z))^\tor \sqcup \blambda(\SL_2(\Z))^\hor \sqcup \blambda(\SL_2(\Z))^\hyp
$$
of $\blambda(\SL_2(\Z))$ into $\{\pm \id\}$, elliptic, parabolic and hyperbolic classes and gives the decomposition
\begin{multline}
\label{eqn:decomp}
\Z\blambda(\SL_2(\Z)) = \Z \{\pm \id\} \oplus \Z\blambda(\SL_2(\Z))^\tor \oplus \Z\blambda(\SL_2(\Z))^\hor \cr \oplus \Z\blambda(\SL_2(\Z))^\hyp.
\end{multline}
There is a similar decomposition of $\blambda(\PSL_2(\Z))$. Elements of $\blambda(\PSL_2(\Z))^\hyp$ are represented by closed geodesics in the modular curve and elements of $\blambda(\PSL_2(\Z))^\hor$ by powers of the horocycle.

\section{The Hecke action on $\Z\blambda(\SL_2(\Z))$}
\label{sec:hecke_action}

In this section we prove Theorems~\ref{thm:T_commute} and \ref{thm:Tp^n}. For an integer $N\ge 1$, define $\Cov_N$ to be the set of isomorphism classes of {\em pairs} of lattices
$$
\Cov_N = \{\Lambda \supset \Lambda' \text{ of index } N\}\ffs\C^\ast
$$
in $\C$, where $\Lambda'$ has index $N$ in $\Lambda$. It is an orbifold model of the moduli space of $N$-fold coverings $E'\to E$ of elliptic curves. The pair of lattices $(\Lambda,\Lambda')$ corresponds to the covering $\C/\Lambda' \to \C/\Lambda$. The connected components of $\Cov_N$ are indexed by the isomorphism classes of quotients of $\Z^2$ of order $N$. In particular, $\Cov_N$ is connected if and only if $N$ is square free.

The Hecke correspondence $T_N$ is the unramified correspondence
\begin{equation}
\label{eqn:TN}
\xymatrix{
& \Cov_N \ar[dl]_{\pi_N}\ar[dr]^{\pi_N^\op} \cr
\M_{1,1} && \M_{1,1}
}
\end{equation}
where $\pi_N$ and $\pi_N^\op$ are the two projections $\pi_N: (\Lambda,\Lambda') \mapsto \Lambda$ and $\pi_N^\op : (\Lambda,\Lambda') \mapsto \Lambda'$. It induces the map $(\pi_N^\op)_\ast \circ \pi_N^\ast$ on $\Z\blambda(\SL_2(\Z))$ that we shall (by abuse of notation) denote by
$$
T_N : \Z\blambda(\SL_2(\Z)) \to \Z\blambda(\SL_2(\Z)).
$$

For each prime number $p$, define
$$
\ee_p : \Z\blambda(\SL_2(\Z)) \to \Z\blambda(\SL_2(\Z))
$$
to be $(\pi_p)_\ast\circ \pi_p^\ast - \id$. The restriction of $\ee_p$ to the image of $\Z\blambda(\G(p))$ is multiplication by $p$, but in general the value of $\ee_p$ on conjugacy classes of elements not in $\G(p)$ is more complicated. One has
$$
\ee_p(\alpha) = - \alpha + \sum \alpha^{d_j},
$$
where $\sum_j d_j = p+1$ and each $d_j$ divides $|\PSL_2(\F_p)|$, which is $p(p^2-1)/2$ when $p$ is odd and 6 when $p=2$. This can be proved using the formula in Section~\ref{sec:alg_pullback} and the fact that the fibers of $Y_0(p) \to \M_{1,1}$ are isomorphic to $\P^1(\F_p)$.

\begin{theorem}
\label{thm:hecke_relns}
Each Hecke correspondence $T_N$ acts on $\Z\blambda(\SL_2(\Z))$. This action preserves the decomposition (\ref{eqn:decomp}) and satisfies the identities
\begin{align*}
T_M \circ T_N &= T_{MN} & \text{when }\gcd(M,N)=1, \cr
T_{p^n} \circ T_p &= T_{p^{n+1}} + T_{p^{n-1}} \circ \ee_p & \text{when } p \text{ prime and } n\ge 1.
\end{align*}
The operators $\ee_p$ and $\ee_q$ commute for all pairs of prime numbers $p$ and $q$.
\end{theorem}

In Example~\ref{ex:non_commute}, we will see that $\ee_p$ does not commute with $T_p$ and that $T_{p^2}$ does not commute with $T_p$.

The two identities in the theorem follow from refinements of the standard arguments that one finds in \cite[VII,\S5]{serre:arithmetic}. The first is a direct consequence of the following easily proved fact.

\begin{proposition}
If the positive integers $M$ and $N$ are relatively prime, then the diagram
$$
\xymatrix{
\Cov_{MN} \ar[r]^{\tau_M}\ar[d]_{\tau_N} & \Cov_N \ar[d]^{\pi_N} \cr
\Cov_M \ar[r]^{\pi_M} & \M_{1,1}
}
$$
of orbifolds is a pullback square, where $\tau_K : (\Lambda,\Lambda')\mapsto (\Lambda, K^{-1}\Lambda')$. The isomorphism $\Cov_M\times_{\M_{1,1}} \Cov_N \to \Cov_{MN}$ is defined by
$$
(\Lambda\supset\Lambda',\Lambda\supset\Lambda'') \mapsto (\Lambda\supset\Lambda'\cap \Lambda'')
$$
on lattices, or by fibered product
$$
(E'\to E,E''\to E) \mapsto (E'\times_E E'' \to E)
$$
of coverings. \qed
\end{proposition}

The rest of this section is devoted to proving the second identity. We begin with some useful background.

\subsection{The involution $\iota_N$}
\label{sec:iota}

The orbifold $\Cov_N$ has an involution $\iota_N$ that is defined by
$$
\iota_N : [\Lambda\supset \Lambda'] \mapsto [\Lambda'\supset N\Lambda].
$$
It is an involution because $[\Lambda\supset \Lambda'] = [N\Lambda\supset N\Lambda']$. In more algebro-geometric language, $\iota_N$ takes the isomorphism class of the isogeny $E' \to E$ to the class of the dual isogeny $\Pic^0 E \to \Pic^0 E'$. This description is equivalent to the first by Abel's theorem: the Abel--Jacobi map $X \to \Pic^0 X$ is an isomorphism for all elliptic curves $X$. Apparently $\iota_N$ is called the {\em Fricke involution}.

The relevance of the involution is that $\pi_N^\op = \pi_N\circ \iota_N$ in the diagram (\ref{eqn:TN}), so that
\begin{equation}
\label{eqn:hecke_iota}
T_N = (\pi_N)_\ast \circ (\iota_N)_\ast \circ \pi_N^\ast.
\end{equation}

Recall that
\begin{equation}
\G_0(N) = \SL_2(\Z)\cap g_N\SL_2(\Z) g_N^{-1} = \Big\{\gamma \in \SL_2(\Z) : \gamma \equiv \begin{pmatrix} \ast & \ast \cr 0 & \ast\end{pmatrix} \bmod N\Big\}
\end{equation}
where
\begin{equation}
\label{eqn:def_gN}
g_N = \begin{pmatrix} N & 0 \cr 0 & 1\end{pmatrix} \in \GL_2(\Q).
\end{equation}
One component of $\Cov_N$ is the moduli space $\Cov_N^\cyc$ of cyclic $N$-fold coverings $E' \to E$. There are natural identifications
$$
\G_0(N)\bbs \h \cong \Cov_N^\cyc \cong \G_0(N)^\op\bbs \h,
$$
where
$$
\G_0(N)^\op = g_N \G_0(N) g_N^{-1} = \Big\{\gamma \in \SL_2(\Z) : \gamma \equiv \begin{pmatrix} \ast & 0 \cr \ast & \ast\end{pmatrix} \bmod N\Big\}.
$$
In the first isomorphism, the $\G_0(N)$-orbit of $\tau$ corresponds to the isomorphism class of the pair of lattices
$$
[\Z + \Z\tau \supset \Z + N\Z\tau].
$$
In the second, the $\G_0(N)^\op$-orbit of $\tau$ corresponds to the pair
$$
[\Z + \Z\tau \supset N\Z + \Z\tau].
$$
So the diagram
$$
\xymatrix{\G_0(N)\bbs \h \ar[r]\ar[d]_\wr & \G_0(N)^\op \bbs \h \ar[d]^\wr \cr
\Cov_N^\cyc \ar[r]_{\iota_N} & \Cov_N^\cyc
}
$$
commutes, where the vertical arrows are the identifications described above and where the top map is induced by $g_N : \tau \to N\tau$.

Denote the projection $\G_0(N)^\op\bbs \h \to \SL_2(\Z)\bbs \h$ by $\pi^\op$. Observe that $\pi^\op_N\circ \iota_N = \pi'_N$.

Recall the definition of $\hhat$ from (\ref{eqn:hhat}) in Section~\ref{sec:prelims}. When $N$ is a prime number $p$, every covering is cyclic. This, the relation (\ref{eqn:hecke_iota}) and the fact that $g_p \in \GL_2(\Q)^+$ acts on $\hhat$ establishes:

\begin{proposition}
\label{prop:T_p}
The operator $T_p$ is realized by the unramified correspondence
$$
\xymatrix{
& \G_0(p)\bbs \hhat \ar[dl]_\pi \ar[r]^(.42){\iota_p} & \G_0(p)^\op \bbs \hhat \ar[dr]^{\pi^\op} \cr
\SL_2(\Z)\bbs \hhat &&& \SL_2(\Z)\bbs \hhat
}
$$
where the top map is induced by $g_p : \hhat \to \hhat$ and $\pi$ and $\pi^\op$ are the natural covering projections. It is self dual in the sense that $T_p$ equals its adjoint $T_p^\vee$.
\end{proposition}

\begin{remark}
The Cartan involution of $\SL_2(\R)$ takes $g\in \SL_2(\R)$ to its inverse transpose $g^{-\trans}$. The automorphism
$$
g \mapsto (g_Ng g_N^{-1})^{-\trans} = g_N^{-1} g^{-\trans} g_N
$$
of $\SL_2(\R)$ restricts to an automorphism of $\G_0(N)$. It induces the orbifold isomorphism
$$
\G_0(N) \bbs \h \to \G_0(N)\bbs \h
$$
that is covered by the map $\tau \mapsto -1/N\tau$ of $\h$ to itself. It takes the cyclic isogeny $E' \to E$ to the dual isogeny $\Pic^0 E \to \Pic^0 E'$.
\end{remark}

\subsection{Proof of the second identity}

Fix a prime number $p$ and a positive integer $n$. We now establish the second relation in Theorem~\ref{thm:hecke_relns}. This is the relation that relates $T_{p^n}\circ T_p$ to $T_{p^{n+1}}$ and $T_{p^{n-1}}$.

The correspondence $T_{p^n} \circ T_p$ is the unramified correspondence
\begin{equation}
\label{eqn:Tp^n}
\xymatrix@C=16pt@R=12pt
{
&& \Cov_p\times_{\M_{1,1}} \Cov_{p^n} \ar[dr]\ar[dl] \cr
& \Cov_p \ar[dl]_\pi\ar[dr]^{\pi^\op} && \Cov_{p^n} \ar[dl]_{\pi'}\ar[dr]^{\pi''} \cr
\M_{1,1} && \M_{1,1} && \M_{1,1}
}
\end{equation}
where the maps are defined by
$$
\xymatrix@C=12pt@R=12pt
{
&& (\Lambda\supset \Lambda' \supset \Lambda'') \ar[dl]\ar[dr] \cr
& (\Lambda \supset \Lambda') \ar[dl]\ar[dr] && (\Lambda',\Lambda'') \ar[dl]\ar[dr] \cr
\Lambda && \Lambda' && \Lambda''
}
$$
The middle square is a pullback along $\pi^\op: \Cov_p \to \M_{1,1}$ and $\pi':\Cov_{p^n} \to \M_{1,1}$. To establish the formula, we need to relate it to $\Cov_{p^{n\pm 1}}$.

Note that for all non-zero integers $N$, the lattices $\Lambda$ and $N\Lambda$ are isomorphic, and thus determine the same (orbi) point in $\M_{1,1}$. A sublattice $\Lambda''$ of a lattice $\Lambda$ of index $p^{n+1}$ is either contained in $p\Lambda$ or is not. Consequently, we can decompose $\Cov_{p^{n+1}}$ as the disjoint union
\begin{equation}
\label{eqn:cov_decomp}
\Cov_{p^{n+1}} = \Cov_{p^{n+1}}^0 \sqcup \Cov_{p^{n-1}},
\end{equation}
where $\Cov_{p^{n+1}}^0$ is the moduli space of pairs of lattices $(\Lambda,\Lambda'')$, where $\Lambda''$ has index $p^{n+1}$ and $\Lambda''\not\subseteq p\Lambda$. The other component consists of all index $p^{n-1}$ sublattices $\Lambda''$ of $p\Lambda$. These have index $p^{n+1}$ in $\Lambda$.

The next lemma is the analogue for correspondences of the usual argument used to compute $T_p\circ T_{p^n}$ on the level of points.

\begin{lemma}
There is a natural isomorphism of the pullback $\Cov_p\times_{\M_{1,1}} \Cov_{p^n}$ in (\ref{eqn:Tp^n}) with $\Cov_{p^{n+1}}^0 \sqcup \pi^\ast\Cov_{p^{n-1}}$. The diagram
$$
\xymatrix{
\Cov_{p^{n+1}}^0 \sqcup \pi^\ast\Cov_{p^{n-1}} \ar[r]\ar[d] & \Cov_{p^n} \ar[d]^{\pi'} \cr
\Cov_p \ar[r]_{\pi^\op} & \M_{1,1}
}
\qquad
\xymatrix{
(\Lambda \supset \Lambda' \supset \Lambda'') \ar[r]\ar[d] & (\Lambda'\supset \Lambda'')\ar[d] \cr
(\Lambda\supset \Lambda') \ar[r] & \Lambda'
}
$$
is a pullback square.
\end{lemma}

\begin{proof}
The pullback $\Cov_p \times_{\M_{1,1}} \Cov_{p^n}$ consists of all triples $(\Lambda \supset \Lambda' \supset \Lambda'')$, where $\Lambda'$ has index $p$ in $\Lambda$ and $\Lambda''$ has index $p^n$ in $\Lambda'$. This can be decomposed into two subspaces, one consisting of the triples where $\Lambda''$ is contained in $p\Lambda$ and the other where it is not. If $\Lambda'' \not\subseteq p\Lambda$, then the image of $\Lambda''$ in $\Lambda/p\Lambda$ has index $p$. This implies that there is a unique sublattice $\Lambda'$ of $\Lambda$ index $p$ that contains $\Lambda''$. This implies that one component is $\Cov_{p^{n+1}}^0$. The other component consists of all isomorphism classes of sequences
$$
(\Lambda \supset \Lambda ' \supset p\Lambda \supseteq \Lambda'')
$$
of lattices. But this component is clearly the pullback of $\Cov_{p^{n-1}}$, the space of all $(p\Lambda\supseteq \Lambda'')$ along the projection $\pi: \Cov_p  \to \M_{1,1}$.
\end{proof}

As a consequence, the unramified correspondence $T_{p^n}\circ T_p$ is the sum of the two correspondences
$$
\xymatrix{
&\Cov_{p^{n+1}}^0 \ar[dl]\ar[dr] \cr
\M_{1,1} && \M_{1,1}
}
\qquad
\xymatrix{
&\pi^\ast\Cov_{p^{n-1}} \ar[dl]\ar[dr] \cr
\M_{1,1} && \M_{1,1}
}
$$

The decomposition in (\ref{eqn:cov_decomp}) implies that the first is $T_{p^{n+1}} - T_{p^{n-1}}$. The second correspondence is the composite
$$
\xymatrix{
& \Cov_p  \ar[dl]_\pi \ar[dr]^\pi && \Cov_{p^{n-1}} \ar[dl]_{\pi'}\ar[dr]^{\pi''} \cr
\M_{1,1} && \M_{1,1} && \M_{1,1}
}
$$
which is $T_{p^{n-1}}\circ (\pi_\ast\pi^\ast) = T_{p^{n-1}} + T_{p^{n-1}}\circ\ee_p$. This completes the proof of the identity. 

The commuting of $\ee_p$ and $\ee_q$ can be proved directly or deduced from the special cases
$$
\ee_p = T_p^2 - T_{p^2} \text{ and } \ee_q = T_q^2 - T_{q^2}
$$
of the relation we have just proved and the fact that $T_{p^n}$ commutes with $T_{q^m}$ when $p\neq q$. This completes the proof of the theorem.

\section{The action of $T_p$ on $\Z\blambda(\SL_2(\Z))$}
\label{sec:action}

In this section we explicitly compute the action of $T_p$ on the elliptic and parabolic elements of $\blambda(\SL_2(\Z))$ and compute the minimal polynomial $m_p(x)$ of $\ee_p$. We also make some general comments about computing the value of $T_p$ on hyperbolic elements.

Throughout this section $p$ is a fixed prime number. Let $\H$ be the local system over $\M_{1,1}$ whose fiber over the point corresponding to the elliptic curve $E$ is $H_1(E;\Z)$.  Denote $\H\otimes\Fp$ by $\H_\Fp$ and its projectivization by $\P(\H_\Fp)$. The projection
$$
\pi : \P(\H_\Fp) \to \M_{1,1}
$$
is a covering projection that is isomorphic to the covering $Y_0(p) \to \M_{1,1}$. The point $L \in \P(\H_\Fp)$ corresponds to the covering $E'\to E$ of $E$ where $H_1(E')$ is the inverse image of the one dimensional subspace $L$ of $H_1(E;\F_p)$.

Another way to think of the fiber of $Y_0(p) \to \M_{1,1}$ over $x_o$ and the right $\SL_2(\Z)$-action on it is to note that the fiber is
$$
\G_0(p)\bs \SL_2(\Z) \cong B(\Fp)\bs \SL_2(\Fp) \cong \P^1(\Fp)
$$
where $B$ is the upper triangular Borel subgroup; the stabilizer of $(0,1)\in (\F_p)^2$.

The right action of $\gamma \in \SL_2(\Z)$ on $\pi^{-1}(x_o)$ is given by matrix multiplication on the right:
$$
[u,v] \mapsto [(u,v)\gamma].
$$

The following result is elementary and well-known. Since it plays a key role, we provide a short proof.

\begin{proposition}
\label{prop:coset_reps}
The distinct conjugates of $\G_0(p)$ in $\SL_2(\Z)$ are
$$
\gamma_j^{-1} \G_0(p) \gamma_j,\quad j=0,\dots,p
$$
where $\gamma_0 = \id$ and
$$
\gamma_j = \begin{pmatrix} 0 & -1 \cr 1 & j \end{pmatrix},\quad j = 1,\dots,p.
$$
\end{proposition}

\begin{proof}
The stabilizer in $\SL_2(\Fp)$ of $[0,1] \in \P^1(\Fp)$ is the Borel $B=B(\Fp)$. Since
$$
[0,1]\gamma_j = [1,j],\quad j=1,\dots,p,
$$
stabilizer of $[1,j]\in \P^1(\Fp)$ is $\gamma_j^{-1}B\gamma_j$. This proves the result as the conjugates of $\G_0(p)$ in $\SL_2(\Z)$ are the inverse images of the conjugates of $B$ in $\SL_2(\Fp)$. 
\end{proof}

\subsection{The general formula}
\label{sec:formula}

Suppose that $\alpha \in \SL_2(\Z)$. We will abuse notation and also denote its image in $\blambda(\SL_2(\Z))$ by $\alpha$. The set $\sS_\alpha$ of $\langle \alpha \rangle$ orbits on the fiber of $Y_0(p) \to \M_{1,1}$ is $\P^1(\Fp)/\langle \alpha \rangle$.

Let $d_j$ be the size of the $\langle \alpha\rangle$ that contains $[0,1]\gamma_j$. Choose a set of orbit representatives $[0,1]\gamma_k$ and let $S_\alpha$ be the corresponding set of indices $k \in \{0,\dots,p\}$. The following formulas are immediate consequences of the formula (\ref{eqn:pushforward}) and the discussion in Section~\ref{sec:alg_pullback}, Proposition~\ref{prop:T_p} and the fact that if $j$ and $k$ are in the same $\alpha$ orbit, then $\gamma_j\alpha^{d_j}\gamma_j^{-1}$ and $\gamma_k\alpha^{d_k}\gamma_k^{-1}$ are conjugate in $\G_0(p)$.

\begin{proposition}
\label{prop:formula}
With this notation
$$
\pi^\ast(\alpha) = \sum_{k\in S_\alpha} \gamma_k \alpha^{d_k} \gamma_k^{-1}
= \sum_{j=0}^p (\gamma_j \alpha^{d_j} \gamma_j^{-1})/d_j \in \Z\lambda(\G_0(p)),
$$
$$
\ee_p(\alpha) = -\alpha + \sum_{k\in S_\alpha} \alpha^{d_k}
= -\alpha + \sum_{j=0}^p \alpha^{d_j}/d_j \in \Z\blambda(\SL_2(\Z))
$$
and
$$
T_p(\alpha) = \sum_{k\in S_\alpha} (g_p\gamma_k) \alpha^{d_k}(g_p\gamma_k)^{-1}
= \sum_{j=0}^p (g_p\gamma_j \alpha^{d_j} (g_p\gamma_j)^{-1})/d_j  \in \Z\blambda(\SL_2(\Z)).
$$
\end{proposition}

All $d_k$ equal one if and only if $\alpha$ acts trivially on $\P^1(\Fp)$. That is, when $\pm \alpha \in \G(p)$.

\begin{corollary}
\label{cor:e_p}
If $\pm \alpha \in \G(p)$, then
$$
\pi^\ast(\alpha) = \sum_{j=0}^p \gamma_j\alpha \gamma_j^{-1} \in \Z\lambda(\G_0(p)),
$$
and
$$
T_p(\alpha) = \sum_{j=0}^p g_p\gamma_j\alpha \gamma_j^{-1}g_p^{-1} \in \Z\blambda(\SL_2(\Z))
$$
Moreover $\ee_p(\alpha) = p\alpha$ if and only if $\pm\alpha \in \G(p)$.
\end{corollary}

\subsection{Computing $\ee_p$}
\label{sec:compute_ep}

Suppose that $\alpha \in \SL_2(\Z)$. In this section we will regard $\alpha$ as acting on $\P^1(\Fp)$. We already know how to compute $\pi^\ast(\alpha)$ when $\alpha$ acts trivially.  We now suppose that $\alpha$ acts non-trivially on $\P^1(\F_p)$.

To understand $\pi^\ast\alpha$, we need to understand the cycle decomposition of $\alpha$ acting on $\P^1(\Fp)$. This can be understood using linear algebra. The first observation is that the fixed points of $\alpha$ correspond to eigenspaces of $\alphabar$, its reduction mod $p$. This means that there are at most 2 fixed points.

To understand this better, consider the characteristic polynomial
$$
x^2 - tx + 1 \in \Fp[x]
$$
of $\alpha$, where $t$ is its trace. It has discriminant $\Delta = t^2-4$. There are 3 cases, namely:
$$
\left(\frac{\Delta}{p}\right) = 0,1,-1,           
$$
which correspond to $\Delta$ being 0 mod $p$, a non-zero square or a non-square mod $p$. In the first case, $t\equiv \pm 2$ mod $p$ so that the characteristic polynomial of $\alpha$ is $(x\pm 1)^2$ mod $p$.

\subsubsection{$\Delta \equiv 0 \bmod p$ and $\pm\alpha \notin \G(p)$} Since the characteristic polynomial of $\alpha$ is $(x\pm 1)^2$ mod $p$ and since $\pm\alpha \notin \G(p)$, it is conjugate to
$$
\pm \begin{pmatrix} 1 & u \cr 0 & 1 \end{pmatrix}
$$
mod $p$, where $u\not\equiv 0\bmod p$. So $\alpha$ fixes a unique point of $\P^1(\Fp)$ (viz, its eigenspace). As a permutation of $\P^1(\Fp)$ it is the product of a 1-cycle and a $p$-cycle. Consequently, $\pi^\ast(\alpha)$ is the sum of two loops. If the eigenspace of $\alpha$ is $[0,1]\gamma_j$, then $\alpha \in \gamma_j^{-1} \G_0(p)\gamma_j$ and $\alpha^p \in \gamma_k^{-1}\G_0(p)\gamma_k$ for any $k\neq j$. In this case
$$
\pi^\ast(\alpha) = [\gamma_j\alpha\gamma_j^{-1}] + [\gamma_k \alpha^p \gamma_k^{-1}]
$$
and
$$
\pi_\ast\pi^\ast(\alpha) = \alpha + \alpha^p \text{ and } \ee_p(\alpha) = \alpha^p.
$$

\subsubsection{$\Delta$ is non-zero and not a square mod $p$}
\label{sec:non-square}
In this case $\alpha$ has no eigenvectors in $H_\Fp$, so $\alpha$ has no fixed points. Since $\Delta$ is not a square mod $p$,
$$
\kk = \Fp(\Delta)
$$
is a field of order $p^2$. Denote the Galois involution of $\kk$ by $y\mapsto \bar{y}$.

The group $\kk^\times/\F_p^\times$ is cyclic of order $p+1$. Let $r,\bar{r} \in \kk^\times$ be the eigenvalues of $\alpha$.

\begin{proposition}
If $m$ is the order of $r$ in $\kk^\times/\F_p^\times$, then $m>1$ and $\alpha$ acts on $\P^1(\Fp)$ as a product of $(p+1)/m$ disjoint $m$-cycles. Consequently
$$
\pi_\ast\pi^\ast(\alpha) = (p+1) \alpha^m/m \text{ and } \ee_p(\alpha) = - \alpha + (p+1)\alpha^m/m .
$$
\end{proposition}

\begin{proof}
Fix an eigenvector $v\in H_\kk$ of $\alpha$ with eigenvalue $r$. Then $\bar{v}$ is an eigenvector with eigenvalue $\bar{r}$. We can identify $H_\Fp$ with the Galois invariants
$$
\{yv+\bar{y}\bar{v} : y \in \kk\} \subset H_\kk.
$$
The $\Fp$-linear map $\kk \to H_\Fp$ that takes $y\in \kk$ to $yv+\bar{y}\bar{v}$ is an isomorphism of $\F_p$ vector spaces. It therefore induces a bijection
$$
\kk^\times/\Fp^\times \overset{\simeq}{\To} \P(H_\Fp)\cong \P^1(\Fp).
$$
Since $v\alpha = r v$, the automorphism of $\P^1(\Fp)$ induced by $\alpha$ corresponds to the multiplication by $r$ map $\kk^\times/\Fp^\times \to \kk^\times/\Fp^\times$. The result follows.
\end{proof}

\subsubsection{$\Delta$ odd and $p=2$} This case is similar to the previous one in that $\alpha$ will not have eigenvectors in $H_{\F_2}$. This corresponds to the case where we identify $\P^1(\F_2)$ with $\P(\F_4)$ and $\alpha$ acts as multiplication by a generator of $\F_4^\times$. That is, it acts as a 3-cycle on $\P^1(\F_2)$.

\subsubsection{$\Delta$ a non-zero square mod $p$, with $p$ odd} In this case, $\alpha$ has two distinct eigenvalues in $\Fp$ and thus 2 distinct fixed points in $\P^1(\Fp)$. Since $\alpha$ acts non-trivially on $\P^1(\Fp)$, the eigenvalues cannot be $\pm 1$.

\begin{proposition}
If $\alpha$ has eigenvalue $\lambda \in \Fp^\times$, then $\alpha$ acts on $\P^1(\Fp)$ as a product of two 1-cycles and $(p-1)/m$ disjoint $m$-cycles, where $m > 1$ is the order of $\lambda^2$ in $\Fp^\times$. Consequently,
$$
\pi_\ast\pi^\ast(\alpha) = 2\alpha + (p-1)\alpha^m/m \text{ and } \ee_p(\alpha) = \alpha + (p-1)\alpha^m/m.
$$
\end{proposition}

\begin{proof}
Here $\alpha$ is conjugate to
$$
\begin{pmatrix} \lambda & 0 \cr 0 & \lambda^{-1} \end{pmatrix}.
$$
This fixes $0,\infty \in \P^1(\Fp)$. Identify their complement with $\Gm(\Fp)$. Then $\alpha$ acts on it by multiplication by $\lambda^2$. Since this action is faithful, it will act on $\Gm(\Fp)$ with $(p-1)/m$ orbits of length $m$. Since $\lambda\neq \pm 1$, $m>1$.
\end{proof}

\subsection{The minimal polynomial of $\ee_p$}
\label{sec:min_poly}

As we have seen in Section~\ref{sec:compute_ep}, each $\gamma \in \SL_2(\Z)$ acts on $\P^1(\Fp)$ with either 0, 1, 2 or $p+1$ fixed points and it acts on the complement of these as a product of disjoint cycles, all of the same length. This makes it easy to compute the minimal polynomial of $\ee_p$.

\begin{lemma}
If $\alpha$ acts on $\P^1(\Fp)$ with $f$ fixed points and $d$ disjoint $m$-cycles, then the ideal of polynomials $h(x)$ with the property that $h(\pi_\ast\pi^\ast) (\alpha)=0$ is generated by
$$
h_\alpha(x) :=
\begin{cases}
(x-p-1) & d=0,\cr
(x-f)(x-p-1) & d > 0.
\end{cases}
$$
\end{lemma}

\begin{proof}
If $d=0$, then $\alpha$ acts trivially on $\P^1(\Fp)$, so that $\pi_\ast\pi^\ast(\alpha) = (p+1)\alpha$. This satisfies the polynomial $x-p-1$. Suppose now that $d>0$. Then
$$
\pi_\ast\pi^\ast(\alpha) = f\alpha + d\alpha^m \text{ and } (\pi_\ast\pi^\ast)^2(\alpha) = f^2\alpha + (p+f+1)d\,\alpha^m.
$$
These do not satisfy any linear polynomial, but do satisfy
$$
(\pi_\ast\pi^\ast)^2(\alpha) -(p+f+1)\pi_\ast\pi^\ast(\alpha) + (p+1)f \alpha = 0.
$$
That is, it satisfies the polynomial $(x-p-1)(x-f)$, but not any polynomial of lower degree.
\end{proof}

We can now compute the minimal polynomial of $\ee_p$.

\begin{proposition}
\label{prop:min_poly}
The minimal polynomial $m_p(x)$ of $\ee_p$ is
$$
m_p(x) =
\begin{cases}
x(x+1)(x-2) & p=2,\cr
x(x^2-1)(x-p) & p\text{ odd}.
\end{cases}
$$
\end{proposition}

\begin{proof}
The minimal polynomial of the restriction of $\ee_p$ to the conjugacy class $\alpha$ is $m_\alpha(x) := h_\alpha(x+1)$. The minimal polynomial of $\ee_p$ is the lowest common multiple of these.

If $p=2$ and $\alpha$ acts non-trivially on $\P^1(\Fp)$, then $f \in \{0,1\}$. Thus $m_2(x)$ is the lcm of $(x-2)$, $x(x-2)$ and $(x+1)(x-2)$, which is $m_2(x)$.

If $p$ is odd, then $f \in \{0,1,2,p+1\}$. The previous lemma and the results in Section~\ref{sec:compute_ep} imply that $m_p(x)$ is the lcm of $(x-p)$, $(x+1)(x-p)$, $x(x-p)$ and $(x-1)(x-p)$, so that $m_p(x) = x(x^2-1)(x-p)$.
\end{proof}

\subsection{Computing $T_p(\alpha)$ on parabolic and elliptic elements}

Here we compute the action of $T_p$ on parabolic and elliptic elements of $\SL_2(\Z)$.

\subsubsection{Parabolic elements}
The horocycle about the cusp of $\M_{1,1}$ corresponds to the conjugacy class $\sigma_o$ of the matrix
$$
\begin{pmatrix}
1 & 1 \cr 0 & 1
\end{pmatrix}.
$$

\begin{proposition}
Suppose $\e \in \{\pm 1\}$. For all $n\in \Z$ we have
$$
\ee_p(\e\sigma_o^n) =
\begin{cases}
\e^{p}\sigma_o^{np}&  p \nmid n, \cr
\e p\sigma_o^{n} & p|n
\end{cases}
$$
and
$$
T_p(\e\sigma_o^n) =
\begin{cases}
\e\sigma_o^{np} + \e^p \sigma_o^n & p \nmid n, \cr
\e\sigma_o^{np} + \e p \sigma_o^{n/p} & p|n.
\end{cases}
$$
\end{proposition}

\begin{proof}
The modular curve $Y_0(p)$ has two cusps and the projection $Y_0(p) \to \M_{1,1}$ has orbifold degree $p$ in the neighbourhood of one cusp and is a local isomorphism in a neighbourhood of the other. These cusps correspond to the conjugacy classes of
$$
\begin{pmatrix}
1 & 1 \cr 0 & 1
\end{pmatrix}
\text{ and } 
\begin{pmatrix}
0 & -1 \cr 1 & 0
\end{pmatrix}
\begin{pmatrix}
1 & p \cr 0 & 1
\end{pmatrix}
\begin{pmatrix}
0 & -1 \cr 1 & 0
\end{pmatrix}^{-1}
=
\begin{pmatrix}
1 & 0 \cr -p & 1
\end{pmatrix}
$$
in $\G_0(p)$. Denote their classes in $\blambda(\G_0(p))$ by $\sigma_o$ and $\rho_o$, respectively. Then
$$
\pi^\ast (\e\sigma_o^n) = \e \sigma_o^n + \e^{p/g}g\rho_o^{n/g},
$$
where $g = \gcd(n,p)$. The first formula follows as $\pi_\ast(\rho_o) = \sigma_o^p$. Since $g_p\rho_o g_p^{-1} = \sigma_o$ and $g_p\sigma_o g_p^{-1} = \sigma_o^p$ in  $\blambda(\SL_2(\Z))$, we have
$$
T_p(\e\sigma_o^n) = \e \sigma_o^{np} + \e^{p/g}g\sigma_o^{n/g}
$$
\end{proof}

\begin{example}
\label{ex:non_commute}
These formulas imply that
$$
T_p\circ \ee_p (\sigma_o) = \sigma_o^{p^2} + p\sigma_o \text{ and } \ee_p\circ T_p(\sigma_o) = (p+1)\sigma_o^p.
$$
In particular, $T_p$ and $\ee_p$ do not commute. This implies that $T_p$ and $T_{p^2}$ do not commute either since
$$
[T_{p^2},T_p] = [T_p^2 - \ee_p, T_p] = [T_p,\ee_p] \neq 0.
$$
\end{example}

\subsubsection{Elliptic elements}

We'll refer to the conjugacy class of a torsion element as a {\em torsion (conjugacy) class}. Hecke correspondences take torsion classes to sums of torsion classes. For this reason, the action of the $T_p$ on them is easily computed.

The order of a torsion class is defined to be the order of any one of its elements.

\begin{proposition}
\label{prop:fte_order}
If $p$ is a prime number, and $\gamma\in \blambda(\SL_2(\Z))$ is a torsion class, then $T_p = \id + \ee_p$ and
$$
T_p(\gamma) =
\begin{cases}
(p+1)\gamma & \gamma = \pm\id\ \text{for all }p \cr
\gamma^2 + \gamma & |\gamma| = 4 \text{ and } p=2 \cr
2 \gamma + (p-1)\gamma^2/2 &  |\gamma| = 4 \text{ and } p \equiv 1 \bmod 4  \cr
(p+1)\gamma^2/2 & |\gamma| = 4 \text{ and } p\equiv 3 \bmod 4  \cr
\gamma + \gamma^3 & |\gamma|=3,6 \text{ and }p=3 \cr
2\gamma + (p-1)\gamma^3/3 & |\gamma|=3,6 \text{ and } p\equiv 1 \bmod 3 \cr
(p+1)\gamma^3/3 & |\gamma|=3,6 \text{ and } p\equiv -1 \bmod 3\cr
\end{cases}
$$
\end{proposition}

\begin{proof}
First observe that since there is only one conjugacy class of subgroups of $\SL_2(\Z)$ of order 1, 2, 3, 4 and 6, the conjugacy class of an element of $\SL_2(\Z)$ of finite order is determined by the angle through which it rotates the tangent space of $\h$ at any one of its fixed points. Since $g_p : \h \to \h$ is conformal, this implies that if $\mu \in \G_0(p)$, then $g_p \mu g_p^{-1} \in \SL_2(\Z)$ and that $\mu$ and $g_p \mu g_p^{-1}$ are conjugate in $\SL_2(\Z)$. The definition of $T_p$ on $\Z\blambda(\SL_2(\Z))$ then implies that $T_p = \pi_\ast \pi^\ast = \id + \ee_p$ on the finite order conjugacy classes of $\SL_2(\Z)$.

Regarding the computation of $T_p$: the first case is immediate as $\pm \id$ acts trivially on $\P^1(\Fp)$.

The first interesting case is classes of order 4. There are two conjugacy classes of them in $\SL_2(\Z)$. They are switched by multiplication by $-1$. Since the minimal polynomial of an element of order 4 of $\SL_2(\Z)$ is $x^2+1$, every element $\gammatilde$ of order 4 of $\SL_2(\Z)$ has trace 0 and discriminant $\Delta = -4$. Since $\gammatilde^2 = -\id$, it acts on $\P^1(\F_p)$ as a product of 1- and 2-cycles. When $p=2$, $\Delta \equiv 0$ and $\gammatilde$ has one fixed point. Consequently
$$
T_2(\gamma) = \pi_\ast \pi^\ast(\gamma) = \gamma + \gamma^2.
$$
When $p\equiv 1 \bmod 4$, $\Delta$ is a non-zero square mod $p$, so that $\gammatilde$ acts on $\P^1(\F_p)$ with exactly 2 fixed points (its eigenspaces). In this case,
$$
T_p(\gamma) = \pi_\ast\pi^\ast (\gamma) = 2 \gamma + (p-1)\gamma^2/2.
$$
And when $p \equiv 3 \bmod 4$, $\Delta$ is not a square mod $p$, and so $\gammatilde$ acts on $\P^1(\F_p)$ as a product of $(p+1)/2$ disjoint 2-cycles. Consequently
$$
T_p(\gamma) = \pi_\ast\pi^\ast (\gamma) = (p+1)\gamma^2/2.
$$

Next we consider classes of order 3 and 6. In $\SL_2(\Z)$, there are two conjugacy classes of elements of order 3, and two of order 6; $\gamma$ has order 3 if and only if $-\gamma$ has order 6. Each of these classes is determined by the order of a representative $\gammatilde \in \SL_2(\Z)$ and the angle of rotation it induces in the tangent space of any of its fixed points.

The minimal polynomial of an element $\gammatilde$ of $\SL_2(\Z)$ of order 3 is $x^2+x+1$. It thus has trace $-1$ and discriminant $\Delta = -3$. When $p=3$, $\Delta \equiv 0$, which implies that $\gammatilde$ acts on $\P^1(\F_3)$ with 1 fixed point and one 3-cycle. In this case,
$$
T_3(\gamma) = \pi_\ast\pi^\ast (\gamma) = \gamma + \gamma^3.
$$

Quadratic reciprocity implies that when $p>3$, $-3$ is a square mod $p$ if and only if $p\equiv 1 \bmod 3$. In this case, $\gammatilde$ acts on $\P^1(\F_p)$ with 2 fixed points (eigenspaces) and $(p-1)/3$ orbits of length 3, so that
$$
T_p(\gamma) = \pi_\ast\pi^\ast (\gamma) = 2\gamma + (p-1)\gamma^3/3 = 2\gamma + (p-1)\id.
$$
When $p \equiv -1 \bmod 3$, the eigenvalues of $\gammatilde$ lie in a quadratic extension of $\Fp$, and the action of $\gammatilde$ on $\P^1(\Fp)$ has no fixed points. So, in this case,
$$
T_p(\gamma) = \pi_\ast\pi^\ast (\gamma) = (p+1)\gamma^3/3 = (p+1)\id /3.
$$

Since the action of $\gamma$ and $-\gamma$ on $\P^1(\Fp)$ are identical, the analysis for elements of order 6 is similar and is omitted.
\end{proof}

The previous result implies that, unlike in the general case, $T_p$ and $\ee_p$ commute when restricted to the finite order conjugacy classes. This allows us to give a quick computation of the action $T_{p^n}$ (and thus all $T_N$) on the torsion classes. It is an immediate consequence of Theorem~\ref{thm:hecke_relns} and Proposition~\ref{prop:fte_order}.

\begin{corollary}
If $\gamma \in \blambda(\SL_2(\Z))$ is a torsion conjugacy class, then
$$
T_{p^n}(\gamma) = (\id + \ee_p + \ee_p^2 + \cdots + \ee_p^n)(\gamma).
$$
\end{corollary}

\section{The Hecke action in higher level, higher rank, etc}

One can also define generalized Hecke operators in higher level. Since we are mainly focused on level 1 and to avoid technical complications, we will not say much. However, it is worth mentioning that if $Y$ is one of $Y_0(m)$, $Y_1(m)$ or $Y(m)$ (definitions recalled in Section~\ref{sec:prelims}) and if $\gcd(m,N)=1$, then we have the unramified correspondence
$$
\xymatrix{
& Y\times_{\M_{1,1}} \Cov_N \ar[dl]_\pi \ar[dr]^{\pi'} \cr
Y && Y
}
$$
and we can define $T_N : \Z\lambda(Y) \to \Z\lambda(Y)$ as $\pi'_\ast\circ \pi^\ast$. The point here is that, since $m$ and $N$ are relative prime, an $N$-fold covering $E'\to E$ of elliptic curves induces an isomorphism $E'[m]\to E[m]$ of $m$-torsion subgroups. This ensures that the projection $\pi'$ is well defined and allows the definition of $T_N$.

One case of potential importance in the study of multiple zeta values is where $Y=Y(2)$, the moduli stack of elliptic curves with a full level 2 structure. The associated coarse moduli space is the thrice punctured sphere $\Pminus$. So for all odd primes $p$, we have Hecke operators
$$
T_{p^n} : \Z\lambda(\Pminus) \to \Z\lambda(\Pminus).
$$

\begin{remark}
It is clear that one can similarly define an action of Hecke operators on $\Z\blambda(\G)$ in more general situations, such as where $\G$ is a lattice, such as $\GL_n(\Z)$ or $\Sp_g(\Z)$. More generally, one can define Hecke-like operators
$$
T_f : \Z\blambda(\G_g) \to  \Z\blambda(\G_h)
$$
where $\G_g$ denotes the mapping class group associated to a closed surface of genus $g$ and where $f : S \to T$ is the topological model of an unramified covering of compact oriented surfaces of genus $g$ by one of genus $h$.
\end{remark}

\section{The Hecke action on $\Z\blambda(\SL_2(\Zhat))$ and its characters}
\label{sec:dual_Hecke_SL_2}

Fix a prime number $p$. To better understand $T_{p^n} : \Z\blambda(\SL_2(\Zhat)) \to \Z\blambda(\SL_2(\Zhat))$ we write
$$
\SL_2(\Zhat) = \SL_2(\Z_p)\times \SL_2(\Z_{p'}),
$$
where
$$
\SL_2(\Z_{p'}) := \prod_{\ell \neq p} \SL_2(\Z_\ell)
$$
and $\ell$ ranges over the prime numbers $\neq p$. The second identity of Theorem~\ref{thm:hecke_relns} reduces the task of understanding the action of $T_{p^n}$ to understanding the actions of $T_p$ and $\ee_p$.

There are natural inclusions and projections
\begin{equation}
\label{eqn:incln-projn}
\Z\blambda(\SL_2(\Z_p)) \hookrightarrow \Z\blambda(\SL_2(\Zhat)) \to \Z\blambda(\SL_2(\Z_p)).
\end{equation}
Since $\SL_2(\Z_{p'})$ is contained in the closure of $\G_0(p)$ in $\SL_2(\Zhat)$, it follows that both $T_p$ and $\ee_p$ commute with both maps in (\ref{eqn:incln-projn}).

The product decomposition further implies that
$$
\blambda(\SL_2(\Zhat)) = \blambda(\SL_2(\Z_p))\times \blambda(\SL_2(\Z_{p'}))
$$
which yields a canonical isomorphism
$$
\Z\blambda(\SL_2(\Zhat)) \cong \Z\blambda(\SL_2(\Z_p))\otimes \Z\blambda(\SL_2(\Z_{p'})).
$$

Recall the definition of $g_p$ from equation (\ref{eqn:def_gN}).

\begin{proposition}
\label{prop:hecke_decomp}
If $\alpha \in \blambda(\SL_2(\Z_p))$ and $\alpha' \in \blambda(\SL_2(\Z_{p'}))$, then
$$
T_p(\alpha\otimes \alpha') = T_p(\alpha) \otimes g_p\alpha'g_p^{-1}
\text{ and }
\ee_p(\alpha\otimes \alpha') = \ee_p(\alpha)\otimes \alpha'.
$$
In particular,
$$
T_p(1\otimes \alpha') = (p+1)\otimes g_p\alpha'g_p^{-1} \text{ and } \ee_p(1\otimes \alpha') = p \otimes \alpha'. \qed
$$
\end{proposition}

Suppose that $\kk$ is a commutative ring. Endow it with the discrete topology. As in Section~\ref{sec:dual_action}, $\Cl_\kk(\SL_2(\Z))$ denotes the group
$$
\Cl_\kk(\SL_2(\Zhat)) = \varinjlim_N \Cl_\kk(\SL_2(\Z/N\Z)) 
$$
of continuous $\kk$-valued class functions on $\SL_2(\Zhat)$. The Hecke correspondence $T_N$ and the operator $\ee_p$  induce dual operators
$$
\Tdual_N : \Cl(\SL_2(\Zhat)) \to \Cl(\SL_2(\Zhat)) 
\text{ and }
\edual_p : \Cl(\SL_2(\Zhat)) \to \Cl(\SL_2(\Zhat)).
$$

The next result is an immediate consequence of Theorem~\ref{thm:hecke_relns}.

\begin{proposition}
When $m$ and $n$ are relatively prime, the dual Hecke operators $\Tdual_n$ and $\Tdual_m$ commute. When $n\ge 1$ we have
$$
\Tdual_p \circ \Tdual_{p^n} = \Tdual_{p^{n+1}} + \edual_p \circ \Tdual_{p^{n-1}}.
$$
\end{proposition}

The ring of class functions on $\SL_2(\Zhat)$ decomposes:
$$
\Cl_\kk(\SL_2(\Zhat)) = \Cl_\kk(\SL_2(\Z_p)) \otimes \Cl_\kk(\SL_2(\Z_{p'})).
$$
Combining this with Proposition~\ref{prop:hecke_decomp}, we obtain:

\begin{proposition}
For all $\chi \in \Cl_\kk(\SL_2(\Z_p))$ and $\chi' \in \Cl_\kk(\SL_2(\Z_{p'}))$, we have
$$
\Tdual_p(\chi\otimes \chi') = \Tdual_p(\chi)\otimes g_p^\ast \chi'
$$
where $(g_p^\ast \chi')(\mu) := \chi'(g_p\mu g_p^{-1})$.
\end{proposition} 

\begin{remark}
The results of this section can be easily generalized to groups such as $\GL_n(\Zhat)$.
\end{remark}

\part{Relative unipotent completions of modular groups}
\label{part:rel_compln}

The primary goal of this section is to explain the various incarnations of the relative completion of $\SL_2(\Z)$ that we will be working with and to explain the various structures on them, such as mixed Hodge structures and Galois actions. As we shall see in the next part, these structures pass to the rings of class functions on each, which correspond under the canonical comparison maps.

This discussion necessitates a discussion of the modular curve, both as a stack over $\Q$ and as a complex analytic orbifold. Even though this material is surely very well known, much of it classically, we give a quick review of it to fix notation and normalizations, and also for the convenience of the reader. An expanded version of some of the topics can be found in Section~4 of \cite{hain-matsumoto:mem}. Unlike in the previous part, in this part $\M_{1,1}$ will be regarded as a stack over $\Q$ and $\M_{1,1}^\an$ will denote the corresponding complex analytic orbifold.

\section{The modular curve as an orbifold and as a stack}

\subsection{The modular curve as a stack over $\Q$}

The moduli space $\M_{1,\uu/\Q}$ of smooth elliptic curves together with a non-zero abelian differential is the scheme $\Spec \Q[u,v,D^{-1}]$, where $D$ is (up to a factor of 4) the discriminant $u^3 - 27 v^2$ of the plane cubic
\begin{equation}
\label{eqn:ell}
y^2 = 4 x^3 - ux -v.
\end{equation}
In other words, $\M_{1,\uu}$ is the complement in $\A^2_\Q$ of the discriminant locus $D=0$. The point $(u,v)$ corresponds to the elliptic curve defined by the equation above together with the abelian differential $dx/y$. The points of $D^{-1}(0)-\{0\}$ correspond to the nodal cubic plus the choice of a non-zero logarithmic 1-form on its smooth locus.

The multiplicative group $\Gm$ acts on $\A^2_\Q -\{0\}$ by
$$
t : (u,v) \mapsto (t^{-4}u,t^{-6}v)
$$
This action takes the point of $\M_{1,\uu}$ that corresponds to $(E,\w)$ to the point that corresponds to $(E,t\w)$, where $E$ is an elliptic curve (possibly nodal) and $\w$ is a translation invariant differential on $E$.

The moduli stack $\M_{1,1/\Q}$ is the stack quotient
$$
\M_{1,1/\Q} = \M_{1,\uu}\ffs \Gm = (\A^2_\Q - D^{-1}(0))\ffs \Gm
$$
and its Deligne--Mumford compactification $\Mbar_{1,1}$ is
$$
\Mbar_{1,1/\Q} = (\A^2_\Q-\{0\})\ffs \Gm
$$

The universal elliptic curve over $\A^2_\Q - \{0\}$ is the subscheme of
$$
(\A^2_\Q - \{0\}) \times \P^2_\Q
$$
defined by $y^2 z = 4 x^3 - uxz^2 -vz^3$. The $\Gm$-action lifts to this curve:
$$
t : (u,v,x,y,z) \mapsto (t^{-4}u,t^{-6}v,t^{-2}x,t^{-3}y,t^{-6}z).
$$
The stack quotient is the universal elliptic curve $\cE_{/\Q}$ over $\Mbar_{1,1/\Q}$.

The relative dualizing sheaf of the universal elliptic curve over $\A^2_\Q - \{0\}$ is the trivial line bundle $(\A^2_\Q - \{0\})\times \A_\Q^1$. The point $(u,v,s)$ corresponds to the section that takes the elliptic curve (\ref{eqn:ell}) to $s dx/y$. Denote it by $\cLtilde$. The $\Gm$-action on $\A^2_\Q$ lifts to $\cLtilde$:
$$
t : (u,v,s) \mapsto (t^{-4}u,t^{-6}v,st).
$$
Observe that the $(u,v) \mapsto (u,v,u)$ is a $\Gm$-invariant section of $\cLtilde^4$. It corresponds to the section $u(dx/y)^4$. Similarly, $(u,v) \mapsto (u,v,v)$ is the $\Gm$-invariant section of $\cLtilde^6$ that corresponds to $v(dx/y)^6$.

The line bundle $\cL$ over $\Mbar_{1,1/\Q}$ is the stack quotient of $\cLtilde$ by this $\Gm$-action. Note that $u$ descends to a section of $\cL^4$ over $\Mbar_{1,1}$ and $v$ to a section of $\cL^6$. These correspond to appropriate normalizations (see below) of the Eisenstein series of weights 4 and 6, respectively.

\subsection{The modular curve as a complex analytic orbifold}
\label{sec:orbi_models}

The complex analytic orbifold $\M_{1,1}^\an$ associated to $\M_{1,1}$ is the orbifold quotient
$$
\SL_2(\Z) \bbs \h
$$
of the upper half plane $\h$ by the standard action of $\SL_2(\Z)$. For $\tau \in \h$, set $\Lambda_\tau = \Z \oplus \Z\tau$. The point $\tau \in \h$ corresponds to the elliptic curve
$$
(E_\tau,0) := (\C/\Lambda_\tau,0),
$$
together with the symplectic basis $\ba,\bb$ of $\Lambda_\tau \cong H_1(E_\tau;\Z)$, where $\ba$ and $\bb$ correspond to the elements of the symplectic basis $1,\tau$ of $\Lambda_\tau \cong H_1(E_\tau;\Z)$.

The quotient of $\h$ by the group
$$
\begin{pmatrix} 1 & \Z \cr 0 & 1\end{pmatrix}
$$
is a punctured disk $\D^\ast$ with coordinate $q = e^{2\pi i \tau}$. Since the stabilizer of $\infty$ in $\SL_2(\Z)$ is
$$
\G_\infty = \{\pm 1\} \times \begin{pmatrix} 1 & \Z \cr 0 & 1\end{pmatrix},
$$
there is an orbifold map $C_2\bbs \D^\ast \to \M_{1,1}^\an$, where $C_2 = \{\pm 1\}$ acts trivially on $\D^\ast$. The analytic orbifold $\Mbar_{1,1}^\an$ is the orbifold obtained by glueing a copy of $C_2\bbs \D$ onto $\M_{1,1}^\an$ via this map.

The line bundle $\cL^\an \to \M_{1,1}^\an$ is the quotient of $\C \times \h$ by the action
$$
\begin{pmatrix} a & b \cr c & d \end{pmatrix} : (z,\tau) \mapsto \big((c\tau + d)z,(a\tau + b)/(c\tau + d)\big).
$$
Its restriction to $\D^\ast$ is trivial and thus extends naturally to $\D$. This is the canonical extension of $\cL^\an$ to $\Mbar_{1,1}^\an$.

\subsection{The space of lattices and $\M_{1,\uu}^\an$}
\label{sec:lattices}

Denote the set of lattices in $\C$ by $\sR$. The group $\C^\ast$ acts on it; $z\in \C^\ast$ takes the lattice $\Lambda$ to $z\Lambda$. The space of lattices can be identified with $\SL_2(\Z)\bs \GL_2^+(\R)$ and the right $\C^\ast$-action with right multiplication. The map
$$
\C^\ast \times \h \to \sR\qquad (z,\tau) \mapsto z\Lambda_\tau
$$
is a $\C^\ast$-equivariant covering map, where $\C^\ast$ acts on $\C^\ast \times \h$ by multiplication on the first factor. The covering group is $\SL_2(\Z)$, which acts on $\C^\ast \times \h$ by
$$
\begin{pmatrix} a & b \cr c & d\end{pmatrix} : (z,\tau) \mapsto \big((c\tau+d)z,(a\tau+b)/(c\tau+d)\big).
$$
This map descends to a $\C^\ast$-equivariant bijection $\sR \cong \SL_2(\Z)\bs (\C^\ast \times \h)$, which gives $\sR$ the structure of a complex manifold. This descends to an analytic orbifold isomorphism $\M_{1,1}^\an = \SL_2(\Z)\bbs\h \cong \sR\ffs\C^\ast$.

\subsection{The comparison}

The goal of this section is to recall the comparison map between the complex analytic and algebraic constructions of $\M_{1,1}$. More precisely, we prove the following result. The proof is classical, but we include a sketch of it as later we need the precise isomorphism.

Throughout this subsection, we regard $\M_{1,\uu}(\C)$ as the complex manifold $\C^2-D^{-1}(0)$.

\begin{proposition}
The map $\sR \to \M_{1,\uu}(\C)$ defined by taking a lattice $\Lambda$ to $(\C/\Lambda,2\pi i dz)$ is a $\C^\ast$-equivariant biholomorphism. It induces an orbifold isomorphism $\M_{1,1}^\an \to \M_{1,1}(\C)$.
\end{proposition}

The proof uses Eisenstein series and the Weierstrass $\wp$ function. Suppose that $k\ge 2$. The (normalized) Eisenstein series of weight $2k$ is defined by the series
\begin{equation}
\label{eqn:eisenstein}
\bG_{2k}(\tau) = \frac{1}{2}\frac{(2k-1)!}{(2\pi i)^{2k}}
\sum_{\substack{\lambda \in \Lambda_\tau\cr\lambda \neq 0}}
\frac{1}{\lambda^{2k}}
= -\frac{B_{2k}}{4k} + \sum_{n=1}^\infty \sigma_{2k-1}(n)q^n,
\end{equation}
where $\sigma_k(n)=\sum_{d|n}d^k$. It converges absolutely to a modular form of weight $2k$. The Weierstrass $\wp$-function associated to a lattice $\Lambda$ in $\C$ is defined by
\begin{equation}
\label{def:eisenstein}
\wp_\Lambda(z) := \frac{1}{z^2} +
\sum_{\substack{\lambda \in \Lambda\cr \lambda \neq 0}}
\bigg[\frac{1}{(z-\lambda)^2} - \frac{1}{\lambda^2} \bigg].
\end{equation}

For $\tau \in \h$, set $\wp_\tau(z) = \wp_{\Lambda_\tau}(z)$. The function $\C \to \P^2(\C)$ defined by
$$
z \mapsto [(2\pi i)^{-2}\wp_\tau(z),(2\pi i)^{-3}\wp_\tau'(z),1]
$$
induces an embedding of $\C/\Lambda_\tau$ into $\P^2(\C)$ as the plane cubic with affine equation
\begin{equation}
\label{eqn:plane_cubic}
y^2 = 4x^3 - g_2(\tau)x - g_3(\tau),
\end{equation}
where
$$
g_2(\tau) = 20 \bG_4(\tau) \text{ and } g_3(\tau) = \frac{7}{3} \bG_6(\tau).
$$
The abelian differential $dx/y$ pulls back to the 1-form $2\pi i dz$ on $\C/\Lambda_\tau$.


The holomorphic map $\C^\ast \times \h \to \C^2 - D^{-1}(0)$ defined by
$$
(s,\tau) \mapsto \big(s^{-4}g_2(\tau),s^{-6}g_3(\tau)\big)
$$
is surjective (as every smooth elliptic curve over $\C$ is of the form $\C/\Lambda)$ and is constant on the orbits of the $\SL_2(\Z)$-action on $\C^\ast \times \h$. It therefore induces a holomorphic map
$$
\sR = \SL_2(\Z)\bs (\C^\ast \times \h) \to \C^2 - D^{-1}(0)
$$
which is a biholomorphism as one can construct a holomorphic inverse by taking the point that corresponds to the pair $(E,\w)$ to its period lattice
$$
\Lambda = \Big\{\int_\gamma \w : \gamma \in H_2(E;\Z)\Big\}.
$$

\section{Fundamental groups and the Galois action on $\SL_2(\Z)^\wedge$}
\label{sec:fund_gps+galois}

In this section we construct the action of the absolute Galois group $\Gal(\Qbar/\Q)$ on $\Z\blambda(\SL_2(\Z)^\wedge)$. In order to do this, we need to first construct the action of $\Gal(\Qbar/\Q)$ on the geometric fundamental group of $\M_{1,1}$ with a suitable base point. 

If $X$ is a smooth variety defined over a subfield $\kk$ of $\C$, there is a canonical isomorphism
$$
\pi_1^\et(X/\kk,\xbar) \stackrel{\simeq}{\longrightarrow} \Gal(M_X/\kk),
$$
where $M_X$ is the algebraic closure of $\kk(X)$ in the field of meromorphic functions on a universal covering of $X^\an$ and $\xbar = \Spec \overline{\kk(X)}$, the geometric generic point. The inclusion of the base point corresponds to the choice of an embedding of $M_X \hookrightarrow \overline{\kk(X)}$. In the case of the modular curve $\M_{1,1}$, we find that $\Gal(M_X/\Qbar)$ is isomorphic to $\PSL_2(\Z)^\wedge$, as $-\id$ acts trivially on the orbifold universal covering $\h$ of $\M_{1,1}^\an$.

For this reason, we need to take an alternative approach to constructing $\pi_1^\et(\M_{1,1},x_o)$ and the Galois action on it. We could appeal to the work of Noohi \cite{noohi} on fundamental groups of stacks. Instead we take a more elementary approach which is also used in Section~\ref{sec:loc_sys} to describe algebraic connections on vector bundles over $\M_{1,1}$. We exploit the facts that $\SL_2(\Z)$ is a quotient of $B_3$, the braid group on 3-strings, and that $\M_{1,\uu/\Qbar}$ is a smooth scheme, so that its \'etale fundamental group is the profinite completion of its topological fundamental group, which is well-known to be $B_3$.

\subsection{Braid groups}

Recall that the braid group $B_n$ is the fundamental group of the space of monic polynomials of degree $n$ with complex coefficients that have non-vanishing discriminant. This space of polynomials retracts onto the space of polynomials
$$
T^n + a_{n-2} T^{n-2} + \cdots + a_1 T + a_0
$$
with distinct roots that sum to 0. The center of the braid group is infinite cyclic and generated by a ``full twist''.

The space $\M_{1,\uu}(\C) = \C^2 - D^{-1}(0)$ is the space of polynomials $4x^3 - ux - v$. By setting $T= 4^{1/3}x$, it follows that its fundamental group is isomorphic to $B_3$. Fix a base point $(u_o,v_o)$ of $\M_{1,\uu}(\C)$. Let $E_o$ be the corresponding elliptic curve. The monodromy action of $B_3$ on $H_1(E_o;\Z)$ defines a homomorphism $B_3 \to \SL(H_1(E_o;\Z))$.

The following is well known. A proof can be found in \cite[\S8]{hain:lectures}.

\begin{proposition}
The braid group $B_3$ is isomorphic to the fundamental group of the complement of the trefoil knot in $S^3$. It has presentation
$$
B_3 \cong \langle \sigma,\mu : \sigma^2 = \mu^3 \rangle.
$$
The center of $B_3$ is infinite cyclic and generated by the full twist $\sigma^2$. For a suitable choice of symplectic basis of $H_1(E_o;\Z)$, the monodromy homomorphism $B_3 \to \SL_2(\Z)$ is defined by
$$
\sigma \mapsto S:= \begin{pmatrix} 0 & -1 \cr 1 & 0 \end{pmatrix},\ \mu \mapsto U:= \begin{pmatrix} 0 & -1 \cr 1 & 1 \end{pmatrix}.
$$
It is surjective. The full twist is mapped to $-\id$, so that the kernel is generated by the square of a full twist.
\end{proposition}

\begin{corollary}
The group $B_3$ is an extension
\begin{equation}
\label{eqn:extn_B3}
0 \to \Z \to B_3 \to \SL_2(\Z) \to 1
\end{equation}
where the kernel is generated by the square of a full twist.
\end{corollary}

The pure braid group $P_n$ is the kernel of the natural homomorphism $B_n \to \Sigma_n$ onto the permutation group of the roots of the polynomial corresponding to the base point. It is the fundamental group of the complement of the ``discriminant'' divisor $\prod_{j<k} (\lambda_j - \lambda_k)$ in the hyperplane $\lambda_1 + \dots + \lambda_n = 0$ in $\C^n$.

Define $\PG(2)$ to be the quotient of $\G(2)$ by its center $\langle-\id\rangle$. Using the fact that the 2-torsion points of an elliptic curve $y^2 = f(x)$ correspond to the roots of the cubic polynomial $f(x)$ and the fact that $\Sigma_3$ is isomorphic to $\SL_2(\F_2)$, one can easily show that the composite
$$
B_3 \to \SL_2(\Z) \to \SL_2(\F_2)
$$
is the natural homomorphism to $\Sigma_3$. It follows that $P_3$ is an extension
\begin{equation}
\label{eqn:P3_extension}
0 \to \Z \to P_3 \to \PG(2) \to 1
\end{equation}
where the kernel is generated by a full twist. This is the group-theoretic incarnation of the $\C^\ast$ torsor
$$
\{(\lambda_1,\lambda_2,\lambda_3) : \sum \lambda_j = 0,\ \lambda_j \text{ distinct}\} \to \Pminus.
$$
From this (or otherwise), it follows that $\PG(2)$ is isomorphic to the fundamental group of $\Pminus$ and is therefore free of rank 2. The section $t \mapsto (0,1,t)$ induces a splitting on fundamental groups.

\begin{corollary}
\label{cor:P3_prod}
The sequence (\ref{eqn:P3_extension}) is split exact, which implies that $P_3$ is isomorphic to the product of a free group of rank 2 with the infinite cyclic group generated by a full twist.
\end{corollary}

\begin{remark}
Recall from Section~\ref{sec:lattices} that $\sR$ denotes the space of lattices in $\C$. One can define Hecke operators on $\sR$ as in \cite[VII.\S5]{serre:arithmetic}. The action of $T_N$ on $\Z\blambda(\SL_2(\Z))$ lifts to an action of $T_N$ on $\Z\blambda(B_3)$. The action of $t\in \C^\ast$ on $\sR$ corresponds to Serre's operator $R_t$. The relations (\ref{eqn:relation}) lift, but will also involve the operator $R_p$. Note, however, that each $R_t$ acts trivially on $\blambda(B_3)$.
\end{remark}

\subsection{Profinite completion}
In this section, $\G^\wedge$ will denote the profinite completion of the discrete group $\G$.

\begin{proposition}
\label{prop:pro-fte_B3}
The profinite completion
$$
0 \to \Zhat \to B_3^\wedge \to \SL_2(\Z)^\wedge \to 1
$$
of the exact sequence (\ref{eqn:extn_B3}) is exact.
\end{proposition}

\begin{proof}
It is elementary to show that profinite completion is right exact. To prove left exactness here, it suffices to prove that the kernel of $B_3^\wedge \to \PSL_2(\Z)^\wedge$ is topologically generated by a full twist. Since $P_3$ has finite index in $B_3$, it suffices to show that the kernel of $P_3^\wedge \to \PG(2)^\wedge$ is topologically generated by a full twist. But this follows immediately from Corollary~\ref{cor:P3_prod} as the sequence (\ref{eqn:P3_extension}) is split exact and as the profinite completion of a product is the product of the completions.
\end{proof}

\subsection{Base points}

In order to specify the Galois action on the \'etale/orbifold fundamental group of $\M_{1,1/\Qbar}$, we need to specify a base point. We will also specify the corresponding base point in the Betti case. In both cases, we will have occasion to use tangential base points as defined by Deligne in \cite[\S15]{deligne:p1}.

\subsubsection{The Betti case}
\label{sec:betti_pi}

One can use any orbifold map $b$ from a simply connected space $B$ to $\M_{1,1}^\an$ as a base point. A diagram
$$
\xymatrix{
B \ar[r]^b\ar[d]_\phi & \M_{1,1}^\an \cr
B' \ar[ur]_{b'}
}
$$
in which $B$ and $B'$ are simply connected determines an isomorphism $\phi_\ast : \pi_1(\M_{1,1}^\an,b) \to \pi_1(\M_{1,1}^\an,b')$. 

Standard and useful choices of base points include:
\begin{enumerate}

\item the quotient map $p : \h \to \M_{1,1}^\an$,

\item the map $\{\tau\} \to \h \to \M_{1,1}^\an$, where $\tau\in \h$,

\item the map $i(y_o,\infty) \to \h \to \M_{1,1}^\an$ from a segment of the imaginary axis, where $y_o > 0$.

\end{enumerate}
The corresponding fundamental groups will be denoted $\pi_1(\M_{1,1}^\an,p)$, $\pi_1(\M_{1,1}^\an,\tau)$, and $\pi_1(\M_{1,1},\partial/\partial q)$, respectively.\footnote{Note that the image of $i(y_o,\infty)$ in the $q$-disk is the segment $(0,e^{-2\pi y_o})$ of the positive real axis, which lies in the direction of the tangent vector $\partial/\partial q$ under the standard identification of the real and holomorphic tangent spaces.}  The inclusions $i(y_o,\infty) \to \h$ and $\{\tau\} \to \h$ induce natural isomorphisms
$$
\pi_1(\M_{1,1},\partial/\partial q) \cong \pi_1(\M_{1,1},p) \text{ and } \pi_1(\M_{1,1},\tau) \cong \pi_1(\M_{1,1},p).
$$
The action of $\SL_2(\Z)$ on $\h$ induces a natural isomorphism $\pi_1(\M_{1,1}^\an,p) \to \SL_2(\Z)$ and therefore an isomorphism
\begin{equation}
\label{eqn:pi_1-iso}
\pi_1(\M_{1,1}^\an,\partial/\partial q) \overset{\simeq}{\To} \SL_2(\Z).
\end{equation}

\subsubsection{Complex conjugation}
\label{sec:conjn}

This acts on $\M_{1,\uu}$ via $(u,v) \mapsto (\overline{u},\overline{v})$. The real curves (\ref{eqn:plane_cubic}) with positive discriminant are of the form
$$
y^2 = 4(x-a)(x-b)(x-c),
$$
where $a,b,c$ are distinct real numbers satisfying $a+b+c=0$. Fix one and denote the corresponding point of $\M_{1,\uu}(\R)$ by $p_o$. Its image in $\M_{1,1}^\an$ is the image of a point $q_o = e^{-2\pi y_o}$ on the positive real axis of the $q$-disk under the quotient map $\D^\ast \to \M_{1,1}^\an$. (The image of the locus of real curves with negative discriminant in $\M_{1,1}$ is the image of the negative real axis of the $q$-disk.)

Identify $\pi_1(\M_{1,\uu},p_o)$ with the group of 3-string braids in $\C$ whose endpoints lie in the subset $\{a,b,c\}$ of $\R$. Complex conjugation acts on it by taking a braid to its complex conjugate.

\begin{proposition}
The action of complex conjugation on $\pi_1(\M_{1,\uu},p_o)$ induces an action on its quotient $\pi_1(\M_{1,1}^\an,\partial/\partial q)$, which we identify with $\SL_2(\Z)$, as above. Complex conjugation acts on $\SL_2(\Z)$ as conjugation by $\diag(1,-1)$.
\end{proposition}

\begin{proof}
Complex conjugation takes the $m$th power of the full twist to its inverse. Since the kernel of $B_3 \to \SL_2(\Z)$ is generated by the square of a full twist, complex conjugation induces an automorphism of $\SL_2(\Z)$.

The two standard generators $s_1$ and $s_2$ of $B_3$ map to the generators
$$
\begin{pmatrix} 1 & 1 \cr 0 & 1 \end{pmatrix} \text{ and } \begin{pmatrix} 1 & 0 \cr -1 & 1 \end{pmatrix}
$$
of $\SL_2(\Z)$. Complex conjugation maps $s_j$ to $s_j^{-1}$. The induced action on them is via conjugation by $\diag(1,-1)$.
\end{proof}

\begin{corollary}
Complex conjugation acts on $\blambda(\SL_2(\Z))$ via conjugation by $\diag(1,-1)$.
\end{corollary}

\subsubsection{The \'etale case}

For each choice of geometric point $\xbar$ of $\M_{1,\uu/\Q}$, one can define $\pi_1^\et(\M_{1,\uu/\Q},\xbar)$, which is an extension
$$
1 \to \pi_1^\et(\M_{1,\uu/\Qbar},\xbar) \to \pi_1^\et(\M_{1,\uu/\Q},\xbar) \to \Gal(\Qbar/\Q) \to 1.
$$
If $\xbar$ lies above $x\in \M_{1,\uu}(\Q)$, we get a splitting of this sequence and therefore a natural action of the absolute Galois group $\Gal(\Qbar/\Q)$ on the geometric \'etale fundamental group $\pi_1^\et(\M_{1,\uu/\Qbar},\xbar)$ of $\M_{1,\uu}$. By \cite[XII, Cor.~5.2]{SGA}, the geometric \'etale fundamental group is canonically isomorphic to the profinite completion of the topological fundamental group of $\M_{1,\uu}(\C)$, so that we have isomorphisms\footnote{Here and below, we are regarding the $\Q$-point $x$ as lying in $\M_{1,\uu}(\C)$, so that it can be used as a base point for the topological fundamental group.}
$$
\pi_1^\et(\M_{1,\uu/\Qbar},\xbar) \cong \pi_1(\M_{1,\uu}(\C),x)^\wedge \cong B_3^\wedge.
$$

We can use $\xbar$ as an \'etale base point of $\M_{1,1/\Q}$. For simplicity, we suppose that the elliptic curve corresponding to $\xbar$ has automorphism group $\{\pm \id\}$. The $\Gm$ orbit $\bO_\xbar$ of $\xbar$ in $\M_{1,\uu}$ is isomorphic to $\Gm_{/\Qbar}$ and has geometric \'etale fundamental group isomorphic to $\Zhat(1)$. It is topologically generated by a full twist. The \'etale double covering $\bO_\xbar'$ of the orbit also has fundamental group $\Zhat(1)$ and is topologically generated by the square of a full twist. Define the \'etale fundamental group of $\M_{1,1/\Q}$ by
$$
\pi_1^\et(\M_{1,1/\Q},\xbar) := \pi_1^\et(\M_{1,\uu/\Q},\xbar)/\pi_1^\et(\bO_\xbar',\xbar)
$$
and its geometric \'etale fundamental group by
$$
\pi_1^\et(\M_{1,1/\Qbar},\xbar) := \pi_1^\et(\M_{1,\uu/\Qbar},\xbar)/\pi_1^\et(\bO_\xbar',\xbar)
$$
so that there is an exact sequence
$$
1 \to \pi_1^\et(\bO_\xbar',\xbar) \to \pi_1^\et(\M_{1,\uu/\Qbar},\xbar) \to \pi_1^\et(\M_{1,1/\Qbar},\xbar) \to 1
$$
There is a natural exact sequence
$$
1 \to \pi_1^\et(\M_{1,1/\Qbar},\xbar) \to \pi_1^\et(\M_{1,1/\Q},\xbar) \to \Gal(\Qbar/\Q) \to 1.
$$

Proposition~\ref{prop:pro-fte_B3} implies that there is a natural isomorphism
$$
\pi_1^\et(\M_{1,1/\Qbar},\xbar) \cong \pi_1(\M_{1,1}^\an,x)^\wedge \cong \SL_2(\Z)^\wedge.
$$
As remarked in the introduction of Section~\ref{sec:fund_gps+galois}, this should also follow from general results of Noohi \cite{noohi} on fundamental groups of stacks.

\subsubsection{The Tate curve as base point} In this section we construct the \'etale analogue of the base point $\partial/\partial q$ of $\M_{1,1}$ and use it to construct an action of the absolute Galois group on $\SL_2(\Z)^\wedge$. This we do by constructing an \'etale base point of $\M_{1,\uu/\Q}$ and its Betti analogue. Their projections to $\M_{1,1}$ are the Betti and \'etale versions of $\partial/\partial q$. The \'etale analogue of $\partial/\partial q$ corresponds to a map from the \'etale universal covering of the formal punctured $q$-disk to $\M_{1,1}$ and is constructed from the Tate curve.

The Tate elliptic curve \cite[Chapt.~V]{silverman} is defined by $Y^2 + XY = X^3 + a_4(q) X + a_6(q)$, where $a_4(q)$ and $a_6(q)$ are in $\Z[[q]]$. It has discriminant $\Delta(q)$, the normalized cusp form of weight 12. Since $\Delta \equiv q$ mod $q^2$, its pullback to $\Z[q]/(q^2)$ is a smooth elliptic curve with good reduction at all primes $p$. This curve can be regarded as the fiber of the universal elliptic curve over $\partial/\partial q$.

After a change of variables \cite[V\S1]{silverman}, the pullback of the Tate curve to $\Q[[q]]$ has affine equation $y^2 = 4x^2 + g_2(q) x + g_3(q)$. It corresponds to a map
$$
\Spf \Q[[q]] \to \A^2_\Q-\{0\}
$$
which restricts to a map
$$
\Spf \Q\((q\)) \to \M_{1,\uu/\Q},
$$
where $\kk\((q\))$ denotes the ring of formal Laurent series in the indeterminate $q$ with coefficients in the field $\kk$. When $\kk$ is a field of characteristic zero, its algebraic closure is
$$
\kkbar\((q^{1/N}:N\ge 1\)) := \varinjlim_N \kkbar\((q^{1/n}: n\le N \)),
$$
the field of formal Puiseux series in $q$ which is generated by compatible $N$th roots $q^{1/N}$ of $q$.

For all $\Q$-algebras, the map above extends to a ``formal geometric point''
$$
\vv : \Spf \kkbar\((q^{1/N}:N\ge 1\)) \to \M_{1,\uu/\kk}
$$
of $\M_{1,\uu/\kk}$. Under the convention in Section~\ref{sec:prelims}, the right automorphisms of the associated fiber functor from finite \'etale covers to the category of finite sets is $\pi_1^\et(\M_{1,\uu/\kk},\vv)$.

Since $\Gal\big(\Qbar\((q^{1/N}:N\ge 1\))/\Q\((q^{1/N}:N\ge 1\))\big)$ is isomorphic to $\Gal(\Qbar/\Q)$, there is a split exact sequence
$$
1 \to \pi_1^\et(\M_{1,\uu/\Qbar},\vv) \to \pi_1^\et(\M_{1,\uu/\Q},\vv) \to \Gal(\Qbar/\Q) \to 1
$$
and a natural $\Gal(\Qbar/\Q)$-action on $\pi_1^\et(\M_{1,\uu/\Qbar},\vv)$.

The corresponding analytic construction is to consider the map $\D^\ast \to \M_{1,\uu}^\an$ defined on the $q$-disk by $q \mapsto (g_2(q),g_3(q))$. The lift of $\partial/\partial q$ is the restriction of this map to the positive real axis. We will denote it by $\vv^\an$. It projects to the base point $\partial/\partial q$ of $\M_{1,1}^\an$ defined above.

\begin{proposition}
There is a natural isomorphism
$$
\pi_1^\et(\M_{1,\uu/\Qbar},\vv) \cong \pi_1(\M_{1,\uu}^\an,\vv^\an)^\wedge.
$$
Consequently, there is a natural $\Gal(\Qbar/\Q)$-action on the profinite completion of $\pi_1(\M_{1,\uu}^\an,\vv^\an)^\wedge$.
\end{proposition}

\begin{proof}
It suffices to give a sequence of maps of ``base points'' that interpolates between $\vv$ and $\vv^\an$ and thus between their fiber functors. Denote by $\C\{\!\{q\}\!\}$ the ring of power series in $q$ that correspond to germs at the origin of holomorphic functions on the $q$-disk. For each $N\ge 1$, set
$$
q^{1/N} = e^{2\pi i \tau /N} \in \cO(\h).
$$
This is a compatible set of $N$th roots of $q$. Adjoin them to $\C\{\!\{q\}\!\}$ to obtain the algebraically closed field
$$
\C\{\!\{q^{1/N}:N\ge 1\}\!\} := \varinjlim_N \C\{\!\{q^{1/N}: n\le N\}\!\}
$$
of convergent Puiseux series. It imbeds naturally into the ring of germs at $q=0$ of continuous functions on the positive real axis of $\D$. The fiber functor $\vv^\an$ obtained by pulling back finite \'etale covers of $\M_{1,\uu}^\an$ to the positive real axis of $\D^\ast$ is isomorphic to the fiber functor $b$ associated to base changing the covers to $\C\{\!\{q^{1/N}:N\ge 1\}\!\}$. This gives an isomorphism 
$$
\pi_1^\et(\M_{1,\uu/\C},b) \cong \pi_1(\M_{1,\uu}^\an,\vv^\an)^\wedge
$$
Algebraically closed base change, \cite[XIII,Prop.~4.6]{SGA}, gives the isomorphism
$$
\pi_1^\et(\M_{1,\uu/\Qbar},b) \cong \pi_1^\et(\M_{1,\uu/\C},b).
$$

To complete the proof, observe that the fiber functor $b$ on the category of finite \'etale coverings of $\M_{1,\uu/\Qbar}$ is isomorphic to the fiber functor obtained by base changing covers to $\Qbar\((q^{1/N}:N\ge 1\))$ via the maps
$$
\xymatrix{
\Qbar\((q^{1/N}:N\ge 1\)) \ar[r] & \C\((q^{1/N}:N\ge 1\)) & \C\{\!\{q^{1/N}:N\ge 1\}\!\} \ar[l]
}
$$
of algebraically closed fields. This gives isomorphisms
$$
\pi_1^\et(\M_{1,\uu/\Qbar},\vv) \cong \pi_1^\et(\M_{1,\uu/\Qbar},\vv\otimes \C) \cong \pi_1^\et(\M_{1,\uu/\Qbar},b).
$$
\end{proof}

\begin{remark}
This construction is a minor variant of those used by Ihara and Matsumoto \cite{ihara-matsumoto} and Nakamura \cite{nakamura:galois} to study the $\Gal(\Qbar/\Q)$-action on the profinite braid groups $B_n^\wedge$. Nakamura studied the $\Gal(\Qbar/\Q)$-action on $B_3^\wedge$ in detail in \cite{nakamura:limits}.
\end{remark}

\begin{corollary}
\label{cor:galois_action}
The $\Gal(\Qbar/\Q)$-action on $\pi_1(\M_{1,\uu/\Qbar},\vv)$ induces a $\Gal(\Qbar/\Q)$-action on $\SL_2(\Z)^\wedge$ via the isomorphisms
$$
\SL_2(\Z)^\wedge \cong \pi_1(\M_{1,1}^\an,\partial/\partial q)^\wedge \cong \pi_1(\M_{1,\uu/\Qbar},\vv)/\Zhat(1).
$$
\end{corollary}

\begin{remark}
Each choice of an elliptic curve $E/\Q$ and a symplectic basis of $H_1(E^\an;\Z)$ determines an isomorphism
$$
\phi : T(E) \overset{\sim}{\To} \Zhat\oplus \Zhat
$$
of the Tate module $T(E)=\varprojlim_N E(\Qbar)[N]$ of $E$ with $\Zhat^2$. This isomorphism determines a homomorphism $\rho : \Gal(\Qbar/\Q) \to \GL_2(\Zhat)$. Composing with the conjugation action of $\GL_2(\Zhat)$ on $\SL_2(\Zhat)$ defines an action of $\Gal(\Qbar/\Q)$ on $\blambda(\SL_2(\Zhat))$. Changing the framing $\phi$ conjugates $\rho$ by an element $\gamma$ of $\SL_2(\Zhat)$ and induces the inner automorphism $\alpha \mapsto \gamma\alpha\gamma^{-1}$ of $\SL_2(\Zhat)$. The action of the Galois group on $\blambda(\SL_2(\Zhat))$ is unchanged as, for all $\sigma \in \Gal(\Qbar/\Q)$ and $\alpha \in \SL_2(\Zhat)$, we have
$$
[(\gamma \rho(\sigma) \gamma^{-1}) (\gamma \alpha \gamma^{-1}) (\gamma \rho(\sigma)^{-1} \gamma^{-1})] = [\rho(\sigma)\alpha\rho(\sigma)^{-1}] \in \blambda(\SL_2(\Zhat)).
$$
That is, the action of the Galois group on $\blambda(\SL_2(\Zhat))$ does not depend on the choice of $E$ or on the framing $\phi$. The fact that the Tate module $H_\Zhat$ of $E_{\partial/\partial q}$ is $\Zhat(0) \oplus \Zhat(1)$ implies that the Galois action on $\Z\blambda(\SL_2(\Zhat))$ factors through the cyclotomic character $\chi : \Gal(\Qbar/\Q) \to \Zhat^\times$.
\end{remark}

\subsection{Generalized Hecke operators are Galois equivariant}

We can now establish the Galois equivariance of our Hecke operators.

\begin{theorem}
\label{thm:Hecke_Gal_equiv}
For each geometric base point $b$ of $\M_{1,1/\Q}$, the action of the absolute Galois group $\Gal(\Qbar/\Q)$ on $\pi_1^\et(\M_{1,1/\Qbar},b)$ induces an action of $\Gal(\Qbar/\Q)$ on $\Z\blambda(\SL_2(\Z)^\wedge)$. This action does not depend on the choice of the base point $b$. The operators $\ee_p$ ($p$ prime) and the Hecke operators
$$
T_N : \Z\blambda(\SL_2(\Z)^\wedge) \to \Z\blambda(\SL_2(\Z)^\wedge)
$$
are $\Gal(\Qbar/\Q)$-equivariant.
\end{theorem}

\begin{proof}
The first two assertions follow easily from the results above. To prove the last assertion, it suffices to show that $T_p$ and $\ee_p$ are Galois equivariant.
The relations in Theorem~\ref{thm:hecke_relns} combined with Theorem~\ref{thm:gal_equiv} and Proposition~\ref{prop:lambda_comp} imply that, to do this, we need only show that $Y_0(p)$ and its two projections to $\M_{1,1}$ are defined over $\Q$. But this follows from the fact that $Y_{0}(p)_{/\Q}$ is the \'etale covering of $\M_{1,1/\Q}$ corresponding to the inverse image of the upper triangular Borel subgroup $B$ of $\GL_2(\F_p)$ in $\pi_1^\et(\M_{1,1/\Q},b)$ under the homomorphism $\pi_1^\et(\M_{1,1/\Q},b) \to \GL_2(\F_p)$. It is geometrically connected as the canonical homomorphism $\pi_1^\et(Y_{0}(p)_{/\Q},\ast) \to \Gal(\Qbar/\Q)$ is surjective.
\end{proof}

\section{Local systems and connections on the modular curve}
\label{sec:loc_sys}

In preparation for defining the version of relative completion of $\SL_2(\Z)$ we shall need when discussing the Hecke action on iterated integrals, we introduce some local systems over $\M_{1,1}$ and the corresponding connections.

\subsection{Local systems}

Let $A$ be a commutative ring. A local system $\V$ of $A$-modules over $\M_{1,1}^\an$ is a local system $\Vtilde$ over $\h$ endowed with a left $\SL_2(\Z)$-action such the projection $\Vtilde \to \h$ is $\SL_2(\Z)$-equivariant. Since $\h$ is simply connected, $\Vtilde$ is isomorphic to $V\times \h$ as a local system, where $V$ is an $A$-module which is endowed with a natural left $\SL_2(\Z)$-action. The fundamental group $\pi_1(\M_{1,1}^\an,p)$ acts on $V$ on the right via $\gamma : v \to \gamma^{-1}v$. Consequently, local systems over $\M_{1,1}^\an$ correspond to right $\SL_2(\Z)$-modules.

Alternatively, a local system $\V$ over $\M_{1,1}^\an$ corresponds to a local system $\V'$ over $\M_{1,\uu}(\C)$ endowed with a $\Gm$-action. The restriction of $\V'$ to each $\Gm$ orbit in $\M_{1,\uu}(\C)$ is required to be trivial. The fiber is naturally a right $B_3$-module on which the square of the full twist acts trivially. As above, such local systems correspond to $\SL_2(\Z)$-modules.

\begin{remark}
The isotropy group at a point of the $\Gm$ action on $\M_{1,\uu}$ is isomorphic to the automorphism group of the corresponding elliptic curve. The isotropy group at a point can act non-trivially on the fiber of $\V'$ over it. The local system $\H$ defined below is an example where $-1 \in \Gm$ acts non-trivially on every fiber.
\end{remark}

When $\kk$ is a field of characteristic zero, the category of local systems of finite dimensional $\kk$-modules over $\M_{1,1}^\an$ is a $\kk$-linear neutral tannakian category.

\subsubsection{The local system $\H$}
\label{sec:loc_sys_H}

Denote the basis of $H_1(E_\tau)$ that corresponds to the basis $1,\tau$ of $\Lambda_\tau \cong H_1(E_\tau;\Z)$ by $\ba,\bb$. This basis is symplectic with respect to the intersection form $\langle \blank,\blank\rangle$.

The local system $\H_A$ over $\M_{1,1}^\an$ is the local system whose fiber over the moduli point of an elliptic curve $E$ is $H_1(E;A)$. The corresponding local system $\widetilde{\H}$ over $\h$ is naturally isomorphic to the trivial local system
$$
(A\ba \oplus A\bb)\times \h \to \h.
$$
It is convenient to set
$$
H_A = A\ba \oplus A\bb
$$
which is the first homology group $H_1(\cE_\h;A)$ of the universal elliptic curve over $\h$ and can also be regarded as the first homology of the fiber over the universal elliptic curve over $\partial/\partial q$.

The element
$$
\gamma = \begin{pmatrix} a & b \cr c & d \end{pmatrix}
$$
of $\SL_2(\Z)$, which we identify with $\pi_1(\M_{1,1}^\an,\partial/\partial q)$ via (\ref{eqn:pi_1-iso}), acts on the fiber $A\ba\oplus A\bb$ on the right by
\begin{equation}
\label{eqn:action_H}
\gamma : \begin{pmatrix} s & t \end{pmatrix}\begin{pmatrix} \bb \cr \ba \end{pmatrix}
\mapsto \begin{pmatrix} s & t \end{pmatrix} \gamma \begin{pmatrix} \bb \cr \ba \end{pmatrix}\quad s,t\in A.
\end{equation}

The dual local system $\H^\vee_A$ is the local system $R^1 f_\ast A$ associated with the universal elliptic curve $f : \cE \to \M_{1,1}^\an$. It has fiber $H^1(E;A)$ over the moduli point of $E$. Denote the basis of $H^1(E_\tau)$ dual to the basis $\ba,\bb$ of $H_1(E_\tau)$ by $\ba^\vee,\bb^\vee$. The natural left action of $\SL_2(\Z)$ on $H^1(E_\tau)$ is
$$
\begin{pmatrix} \bb^\vee & \ba^\vee \end{pmatrix}\begin{pmatrix} s \cr t \end{pmatrix} \mapsto \begin{pmatrix} \bb^\vee & \ba^\vee \end{pmatrix}\gamma \begin{pmatrix} s \cr t \end{pmatrix}
$$

Poincar\'e duality induces the isomorphism $\H_A \to \H^\vee_A$ that takes $u$ to $\langle u,\blank \rangle$. It identifies the frame $\begin{pmatrix} \bb^\vee & \ba^\vee \end{pmatrix}$ of $H^1(E_\tau)$ with $\begin{pmatrix} \ba & -\bb \end{pmatrix}$ of $H_1(E_\tau)$. Under this identification, the natural left $\SL_2(\Z)$-action on $H_A$ is given by
$$
\begin{pmatrix} \ba & -\bb \end{pmatrix} \mapsto \begin{pmatrix} \ba & -\bb \end{pmatrix} \gamma
$$
This is the same as the left action $x \mapsto x\gamma^{-1}$ obtained from the right action on $H_A$.

\subsubsection{The \'etale local system $\H_\Ql$ over $\M_{1,1/\Q}$}

Fix a prime number $\ell$. There are natural isomorphisms
$$
\pi_1^\et(\M_{1,1/\Q},\partial/\partial q) \cong \pi_1^\et(\M_{1,1/\Qbar},\partial/\partial q) \rtimes \Gal(\Qbar/\Q) \cong \SL_2(\Z)^\wedge \rtimes \Gal(\Qbar/\Q).
$$
This acts on $H_\Ql$, the fiber of $\H_\Ql$ over $\partial/\partial q$. It is well known (see, \cite[\S4]{nakamura:limits}, for example) that, as a Galois module,
$$
H_\Ql \cong \Q_\ell(0) \oplus \Q_\ell(1)
$$
where $\Q(0)$ is spanned by $\bb$ and $\Q(1)$ is spanned by $\ba$. Consequently, $\pi_1^\et(\M_{1,1/\Q},\partial/\partial q)$ acts on $H_\Ql$ on the right via (\ref{eqn:action_H}) and the homomorphism
$$
\SL_2(\Z)^\wedge \rtimes \Gal(\Qbar/\Q) \to \SL_2(\Zhat)\rtimes \Gal(\Qbar/\Q) \to \GL_2(\Z_\ell)
$$
where the last map takes $\sigma \in \Gal(\Qbar/\Q)$ to
\begin{equation}
\label{eqn:gal_action}
\begin{pmatrix} 1 & 0 \cr 0 & \chi_\ell(\sigma) \end{pmatrix} \in \GL_2(\Q_\ell)
\end{equation}
and $\chi_\ell$ is the $\ell$-adic cyclotomic character.

\subsection{Connections over the modular curve}

We recall the connection associated with the local system $\H$ over $\M_{1,1}$ and introduce the local systems $\V_N$ and their associated connections $\cV_N$ which will be used in the construction of the (large) relative completions of $\SL_2(\Z)$ in Section~\ref{sec:rel_comp_mod_gp}.

\subsubsection{The analytic version}

A vector bundle $\cV$ over $\M_{1,1}^\an$ is, by definition, a vector bundle $\cVtilde$ over $\h$ endowed with an $\SL_2(\Z)$-action such that the projection $\cVtilde \to \h$ is $\SL_2(\Z)$-equivariant. A connection on $\cV$ is an $\SL_2(\Z)$-invariant connection on $\cVtilde$.

\subsubsection{The algebraic version}
\label{sec:HDR}

Suppose that $\kk$ is a field of characteristic 0. A vector bundle on $\M_{1,1/\kk}$ is a vector bundle $\cVtilde$ over $\M_{1,\uu/\Q}$ whose restriction to each fiber has a trivialization that is $\Gm$-invariant. (The isotropy group of a point in a $\Gm$ orbit may act non-trivially on the fiber over it.)

A connection on $\cV$ is a $\Gm$-invariant connection $\nabla$ on the vector bundle $\cVtilde$ over $\M_{1,\uu}$ which is trivial on each $\Gm$ orbit. The connection on $\cV$ has regular singularities at infinity if the bundle $\cVtilde$ extends to a vector bundle $\cVbar$ over $\A^2-\{0\}$ and
$$
\nabla : \cVbar \to \cVbar \otimes_{\cO_{\A^2-\{0\}}} \Omega^1_{\A^2-\{0\}}(\log \Sigma)
$$
where $\Sigma$ is the discriminant divisor $D = 0$, where $D=u^3 - 27 v^2$. The restriction of $\cVbar$ to the discriminant divisor in $\A^2-\{0\}$ is trivialized by the $\Gm$-action. We define the fiber of $\cV$ over the cusp of the modular curve to be the vector space of $\Gm$-invariant sections of $\cVbar$ over $\Sigma$.

\subsubsection{The connection $\cH$}
\label{sec:conn_H}

This is the connection associated with the local system $\H$. The analytic version $\cH^\an$ is the trivial connection $\nabla = d$ on the trivial bundle $H_\C \times \h \to \h$. It is invariant under the left $\SL_2(\Z)$-action. Equivalently, it is the natural $\SL_2(\Z)$-invariant connection on $\H \otimes_\Q \cO_\h$ over $\h$.

The algebraic version $\cH$ is the connection
$$
\nabla = d +
\Big(
-\frac{1}{12}\frac{dD}{D}\otimes \sfT
+ \frac{3}{2}\frac{\alpha}{D}\otimes \sfS
\Big)\frac{\partial}{\partial \sfT}
+
\Big(
-\frac{u}{8}\frac{\alpha}{D}\otimes \sfT
+ \frac{1}{12}\frac{dD}{D}\otimes \sfS
\Big)\frac{\partial}{\partial \sfS},
$$
on the restriction $\cHbar$ of the trivial bundle $\cO_{\A^2_{/\Q}} S \oplus \cO_{\A^2_{/\Q}} T$ to $\A^2_{\Q}-\{0\}$, where $\alpha = 2udv-3vdu$ and $D = u^3 - 27 v^2$. It is defined over $\Q$ and invariant under the $\Gm$-action defined by $t\cdot \sfS = t^{-1}\sfS$ and $t\cdot \sfT = t\sfT$. The section $\sfT$ corresponds to $dx/y$ and the section $\sfS$ to $xdx/y$. The extended bundle has a Hodge filtration
\begin{equation}
\label{eqn:hodge_H}
\cHbar = F^0 \cHbar \supset F^1 \cHbar := \cO_{\A^2-\{0\}} \sfT \supset F^2 \cHbar = 0.
\end{equation}
The connection has fiber $H^\DR = \Q \sfS \oplus \Q \sfT$ over the cusp and residue $-\sfS\frac{\partial}{\partial \sfT} \in \End_\Q H^\DR$. After tensoring with $\cO_{\C^2-\{0\}}^\an$, this connection is isomorphic to $\cH^\an$.  Combined with the $\Q$ structure on $\H$, it defines a polarized variation of Hodge structure of weight 1 over $\M_{1,1}$. Details can be found in \cite[\S 19]{hain:kzb}. (Here we use the normalizations from \cite[\S4]{hain-matsumoto:mem}.)

\subsubsection{Betti-de~Rham comparison}
\label{sec:comp_H}

As in \cite[Prop.~5.2]{hain-matsumoto:mem}, we compare the Betti and de~Rham incarnations of $H$ on the lift $\D \to \A^2-\{0\}$ defined by $q \mapsto (g_2(q),g_3(q))$. The fiber of the pullback of $\cHbar$ to $\partial/\partial q$ is (by definition) the fiber of the pullback over $q=0$. It is $H^\DR = \Q\sfS \oplus \Q\sfT$. The fiber of $H^B$ of $\H$ over $\partial/\partial q$ is
$$
H^B = \Q\ba \oplus \Q\bb \cong \Q\bb^\vee \oplus \Q\ba^\vee.
$$
Here we are identifying $H^B$ with its dual via Poincar\'e duality. The inverse of the comparison isomorphism $H^\DR\otimes_\Q\C  \to H^B\otimes_\Q \C$ is
\begin{equation}
\label{eqn:comp_H}
\ba \mapsto \sfS - \sfT/12 \text{ and } \bb \mapsto -\sfT/2\pi i.
\end{equation}
For more details, see \cite[\S5.4]{hain-matsumoto:mem}.

\subsection{Distinguished local systems and connections over $\M_{1,1}$}

In this section we construct the various incarnations of the semi-simple local systems over $\M_{1,1}$ we shall need in the construction of the relative completion of $\SL_2(\Z)$ in Section~\ref{sec:rel_comp_mod_gp}. These are obtained by pulling back $S^m \H$ to its full level $N$ covering and then pushing the result forward to $\M_{1,1}$. We will construct the Betti, complex analytic de~Rham and algebraic de~Rham versions. These are needed to construct the various incarnations of the relative completion of $\SL_2(\Z)$ we shall need and the comparison isomorphisms between them.

Suppose that $N$ is a positive integer. Recall from Section~\ref{sec:prelims} that $\G(N)$ denotes the full level $N$ subgroup of $\SL_2(\Z)$ and that $Y(N)^\an$ is the modular curve $\G(N)\bs\h$. Its points correspond to isomorphism classes of elliptic curves $E$ over $\C$ together with a symplectic isomorphism $H_1(E,\Z/N) \overset{\sim}{\to} (\Z/N)^2$, where $(\Z/N)^2$ is given the standard inner product.

Let $\pi : Y(N)^\an \to \M_{1,1}^\an$ denote the projection. It is Galois with automorphism group $\SL_2(\Z/N)$. Denote the trivial rank 1 local system of $\kk$ vector spaces over the stack $\M_{1,1}^\an$ by $\kk_{\M_{1,1}}$. Denote the trivial rank 1 connection on the stack $\M_{1,1}^\an$ by $(\cO_{\M_{1,1}}^\an,d)$, or simply by $\cO_{\M_{1,1}}^\an$.

\begin{proposition}
Suppose that $\kk$ is a field of characteristic zero.
\begin{enumerate}

\item If $\V$ is a local system of $\kk$ vector spaces on $\M_{1,1}^\an$, then there is a natural isomorphism
\begin{equation}
\label{eqn:push-pull-B}
\pi_\ast\pi^\ast \V \overset{\simeq}{\To} \V \otimes_\kk \pi_\ast \pi^\ast \kk_{\M_{1,1}}.
\end{equation}
It is characterized by the property that the composition of the unit $\V \to \pi_\ast\pi^\ast\V$ of the adjunction with this isomorphism is the tensor product of the unit $\kk_{\M_{1,1}} \to \pi_\ast\pi^\ast \kk_{\M_{1,1}}$ with $\V$. As a $\pi_1(\M_{1,1},b)$-module, the fiber of $\pi_\ast\pi^\ast \V$ of the base point $b$ is naturally isomorphic to the right $\SL_2(\Z)$-module
$$
\CoInd_{\G(N)}^{\SL_2(\Z)} V_b := \Hom_{\G(N)}(\kk[\SL_2(\Z)],V_b),
$$
where $V_b$ denotes the fiber of $\V$ over $b$ and where $\Hom_{\G(N)}$ denotes right $\G(N)$-module homomorphisms.

\item If $\cV$ is a connection over $\M_{1,1}^\an$, then there is a canonical isomorphism
\begin{equation}
\label{eqn:push-pull-DR}
\pi_\ast\pi^\ast \cV \cong \cV\otimes_{\cO_{\M_{1,1}}^\an} \pi_\ast\pi^\ast \cO_{\M_{1,1}}^\an.
\end{equation}
It is characterized by the property that the composition of the unit $\cV \to \pi_\ast\pi^\ast\cV$ of the adjunction with this isomorphism is the tensor product of the unit $\cO^\an_{\M_{1,1}} \to \pi_\ast\pi^\ast \cO^\an_{\M_{1,1}}$ with $\cV$.

\item If $\kk$ is a subfield of $\C$ and $\cV = \V\otimes_\kk\cO_{\M_{1,1}}^\an$, then there is a natural isomorphism
$$
(\pi_\ast\pi^\ast \V)\otimes_\kk \cO_{\M_{1,1}}^\an \cong \pi_\ast\pi^\ast \cV
$$
which is compatible with the isomorphisms (\ref{eqn:push-pull-B}) and (\ref{eqn:push-pull-DR}) above. Consequently, there are canonical isomorphisms
\begin{align}
H^\bdot(\M_{1,1}^\an,\pi_\ast\pi^\ast\V) &\cong H^\bdot_\DR(\M_{1,1}^\an,\pi_\ast\pi^\ast\cV)
\cr
&\cong H^\bdot\big(\SL_2(\Z),\CoInd_{\G(N)}^{\SL_2(\Z)} V_b\big)
\cr
&\cong H^\bdot(\G(N),V_b)
\end{align}

\end{enumerate}
\end{proposition}

The proof is an exercise, but with one caveat. Namely, one has to work $\SL_2(\Z)$-equivariantly on $\h$ or $\Gm$-equivariantly on $\M_{1,\uu}$. This means that local sections of the sheaf $\cO_{\M_{1,1}}^\an$ are identified with either functions defined locally on the upper half plane, or else $\C^\ast$-invariant functions on saturated open subsets of $\C^2 - D^{-1}(0)$, where $D=u^3 - 27 v^2$. The last isomorphism follows from Shapiro's lemma.

\begin{corollary}
\label{cor:comp_isos_H}
The connection associated with $\pi_\ast\pi^\ast S^m \H$ is naturally isomorphic to $\pi_\ast\pi^\ast S^m \cH^\an$. There are natural compatible isomorphisms
$$
\pi_\ast\pi^\ast S^m \H \cong S^m \H \otimes_\Q \pi_\ast\pi^\ast\Q_{\M_{1,1}} \text{ and }
\pi_\ast\pi^\ast S^m \cH^\an \cong S^m \cH^\an \otimes_{\cO_{\M_{1,1}}^\an} \pi_\ast\pi^\ast\cO_{\M_{1,1}}^\an.
$$
For each $m\ge 0$, the monodromy representation factors through the representation
$$
\rho_N : \SL_2(\Z) \to \SL_2(\Q)\times \SL_2(\Z/N)
$$
which is the inclusion on the first factor and reduction mod $N$ on the second.
\end{corollary}

\begin{remark}
There is also an $\ell$-adic \'etale version of these results. But, due to a sleight of hand, we will not need them.
\end{remark}

\subsubsection{The tower of modular curves $Y(N)$}
\label{sec:mod_tower}

As in Section~\ref{sect:Qbar}, $\Qbar$ denotes the algebraic closure of $\Q$ in $\C$ and $\bmu_\infty$ its group of roots of unity. The field $\Q(\bmu_\infty)$ is the maximal abelian extension $\Q^\ab$ of $\Q$. We identify $\bmu_\infty$ with $\Q/\Z$ via the isomorphism $a/b \mapsto \exp(2\pi i a/b)$. Multiplication by $N$ induces an isomorphism of $\Z/N\Z$ with the subgroup $(1/N)\Z/\Z$ of $\Q/\Z$ and thus an isomorphism
$$
\bmu_N \cong \Z/N\Z.
$$
These isomorphisms commute with the inclusions $\bmu_N \hookrightarrow \bmu_{MN}$ and $\Z/N\Z \hookrightarrow \Z/MN\Z$.

Suppose that $N\ge 1$. We will regard the moduli stack $Y(N)$ of smooth elliptic curves with a full level $N$ structure as a stack over $\KDR$. Maps $S \to Y(N)$ from schemes $S/\KDR$ classify elliptic curves $E/S$ together with isomorphisms
$$
H^0(S,E[N]) \to (\Z_S/N)^{\oplus 2}
$$
under which the Weil pairing $H^0(S,E[N])^{\otimes 2} \to \bmu_N(S) \cong (\Z/N)_S$ corresponds to the standard symplectic inner product on $(\Z_S/N)^{\oplus 2}$. Note that $Y(1) = \M_{1,1}\times_\Q \KDR$. Denote it by $Y$. The projection $\pi : Y(N) \to Y$ that forgets the framings is an $\SL_2(\Z/N)$ torsor. The corresponding analytic stack $Y(N)^\an$ is $\G(N)\bbs \h$. When $N\ge 3$, $Y(N)$ is a geometrically connected scheme and $Y(N)^\an$ is a Riemann surface as $\G(N)$ is torsion free for all $N\ge 3$.

Denote the projective completion of $Y(N)$ by $X(N)$ and its set of cusps $X(N)-Y(N)$ by $C_N$. It is the quotient of $\P^1(\Q)$ by $\G(N)$. Denote the cusp corresponding to the orbit of $\infty \in \P^1(\Q)$ by $P_N$. The modular curves $X(N)$, $N\ge 1$, form a projective system and the cusps $\{P_N:N\ge 1\}$ form a compatible set of base points.

\begin{remark}
The moduli stacks $Y(N)$ and $X(N)$ can be defined over $\Q$. See \cite{deligne-rapoport}. As stacks over $\Q$, they are not geometrically connected. Their components correspond to embeddings $\bmu_N \to \C^\ast$. Each geometric component is defined over $\Q(\bmu_N)$. We work over the larger field $\KDR$ as we need to work with the tower of $X(N)$'s rather than any particular member of it.
\end{remark}

\subsubsection{Algebraic DR version}
\label{sec:sym_Q}

In order that the de~Rham incarnation of relative completion be defined over $\KDR$, we need to show that the connections $\pi_\ast\pi^\ast S^m \cH^\an$ over $\M_{1,1}^\an$ have a $\KDR$-de~Rham incarnation. Recall that $Y$ denotes $\M_{1,1}/\KDR$.

For each $N\ge 1$, define the connection $\sP_N$ over $Y$ by $\sP_N = \pi_\ast\pi^\ast \cO_Y$. There is a canonical isomorphism
$$
\sP_N \otimes_{\cO_Y} \cO_{\M_{1,1}^\an} \cong \pi_\ast \pi^\ast \cO_{\M_{1,1}^\an}.
$$
These connections are compatible in the sense that if $N|M$, there is a surjective morphism of connections
\begin{equation}
\label{eqn:morphism_conns}
\sP_M \to \pi_{M,N}^\ast \sP_N
\end{equation}
where $\pi_{M,N} : Y(M) \to Y(N)$ denotes the canonical projection.

\begin{proposition}
\label{prop:sS^m}
For each $m\ge 0$ and $N\ge 1$, there is a connection $\sS^m_N$ over $\M_{1,1/\KDR}$ such that
$$
\sS^m_N \otimes_{\cO_{\M_{1,1}/\KDR}} \cO_{\M_{1,1}^\an} \cong \pi_\ast \pi^\ast S^m \cH^\an.
$$
The comparison isomorphisms between $\pi_\ast\pi^\ast S^m\H$ and $\pi_\ast \pi^\ast S^m \cH^\an$ from Corollary~\ref{cor:comp_isos_H} define a polarized variation of Hodge structure of weight $m$ over $\M_{1,1}^\an$. When $N|M$, the surjection (\ref{eqn:morphism_conns}) is a morphism of variations of Hodge structure.
\end{proposition}

\begin{proof}
The connection $\sS^m_N$ is defined to be $\sP_N \otimes_Y S^m \cH$. The Hodge filtration (\ref{eqn:hodge_H}) extends to $S^m \cH$ and hence to $\sS^m_N$. Since the comparison of $\H$ and $\cH$ define a polarized VHS over $\M_{1,1}^\an$, it follows from standard Hodge theory that the comparison isomorphisms from Corollary~\ref{cor:comp_isos_H} define a polarized variation of Hodge structure of weight $m$ over $\M_{1,1}^\an$.
\end{proof}

\section{Relative unipotent completion in the abstract}
\label{sec:rel_comp}

This is a terse review of relative completion. Since we need to define it in several related contexts, we use Saad's efficient abstract setup \cite[\S5.2.1]{saad} (see also, \cite[\S12]{brown:mmv}) and establish several basic results about relative completion in this context.

Suppose that $\kk$ is a field of characteristic zero and that $\tC$ is a $\kk$-linear neutral tannakian category with fiber functor $\w : \tC \to \Vec_\kk$. Suppose that $\tS$ is a full tannakian subcategory of $\tC$, all of whose objects are semi-simple. Define $\sF(\tC,\tS)$ to be the full subcategory of $\tC$ that consists of all objects $\V$ of $\tC$ that admit a filtration
$$
0 = \V_0 \subseteq \V_1 \subseteq \dots \subseteq \V_{n-1} \subseteq \V_n = \V
$$
where each graded quotient $\V_j/\V_{j-1}$ is isomorphic to an object of $\tS$. This is a tannakian subcategory of $\tC$. The restriction of $\w$ to $\sF(\tC,\tS)$ is a fiber functor. Denote its tannakian fundamental group with respect to $\w$  by $\pi_1(\tC,\tS,\w)$:
$$
\pi_1(\tC,\tS;\w) := \Aut_{\sF(\tC,\tS)}^\otimes \w
$$

Since every object of $\tS$ is semi-simple, $\pi_1(\tS,\w)$ is a (pro)reductive $\kk$-group. (We will usually drop the pro and simply say that $\pi_1(\tS,\w)$ is reductive. The important fact for us is that every representation of $\pi_1(\tS,\w)$ be completely reducible.) Since $\tS$ is a full subcategory of $\tC$, there is a faithfully flat homomorphism $\pi_1(\tC,\tS;\w) \to \pi_1(\tS,\w)$ of affine $\kk$-groups. The kernel is the maximal prounipotent normal subgroup of $\pi_1(\tC,\tS;\w)$. Denote it by $\U$ so that $\pi_1(\tC,\tS;\w)$ is an extension
\begin{equation}
\label{eqn:tannakian_extension}
1 \to \U \to \pi_1(\tC,\tS;\w) \to \pi_1(\tS,\w) \to 1.
\end{equation}

For the time being, we set $\cG = \pi_1(\tC,\tS;\w)$ and $S=\pi_1(\tS,\w)$. Denote the Lie algebra of $\U$ by $\u$. It is pronilpotent. Suppose that $\V$ is an object of $\sF(\tC,\tS)$. Set $V=\w(\V)$. It is a left $\cG$-module. There are natural isomorphisms
\begin{equation}
\label{eqn:ext_iso}
\Ext^j_{\sF(\tC,\tS)}(\unit,\V) \cong \Ext^j_\cG(\kk,V) \cong H^j(\cG,V) \cong H^j(\u,V)^S,
\end{equation}
where $\unit$ denotes the unit object of $\tC$. The first follows from Tannaka duality, the second is the definition of cohomology of algebraic groups, and the third follows from the analogue for affine groups of the Hochschild--Serre spectral sequence of the extension (\ref{eqn:tannakian_extension}), the natural isomorphism $H^j(\u,V) \cong H^j(\U,V)$, and the fact that $S$ is reductive.

The inclusion functor $\sF(\tC,\tS) \to \tC$ is exact and therefore induces homomorphisms
\begin{equation}
\label{eqn:ext}
\Ext^j_{\sF(\tC,\tS)}(A,B) \to \Ext^j_\tC(A,B)
\end{equation}
for all objects $A$ and $B$ of $\sF(\tC,\tS)$ and all $j\ge 0$.

\begin{proposition}
\label{prop:ext}
The homomorphism (\ref{eqn:ext}) is an isomorphism when $j=0,1$ and an injection when $j=2$.
\end{proposition}

\begin{proof}
The case $j=0$ follows from the fact that $\sF(\tC,\tS)$ is a full subcategory of $\tC$. The case $j=1$ follows from fullness and the fact that $\sF(\tC,\tS)$ is closed under extensions. We prove the case $j=2$ using Yoneda's description \cite{yoneda} of Ext groups. Suppose that
$$
0 \to B \to E \to F \to A \to 0
$$
represents an element of $\Ext^j_{\sF(\tC,\tS)}(A,B)$. By \cite[p.~575]{yoneda}, it represents 0 in $\Ext^2_\tC(A,B)$ if and only if there is an object $M$ of $\tC$ with a filtration
$$
0 \subseteq M_0 \subseteq M_1 \subseteq M_2 = M
$$
in $\tC$ and an isomorphism of 2-extensions
$$
\xymatrix{
0 \ar[r] & B \ar[r] \ar[d]^\wr & E \ar[r]\ar[d]^\wr & F \ar[r] \ar[d]^\wr & A \ar[r] \ar[d]^\wr & 0 \cr
0 \ar[r] & M_0 \ar[r] & M_1 \ar[r] & M/M_0 \ar[r] & M/M_1 \ar[r] & 0
}
$$
To prove injectivity it suffices, by Yoneda's criterion, to show that $M$ is an object of $\sF(\tC,\tS)$. Set $C = \coker\{B \to E\}$. This is an object of $\sF(\tC,\tS)$. Since $M_1$ is an extension of $C$ by $B$, it is also in $\sF(\tC,\tS)$, and since $M$ is an extension of $A$ by $M_1$, it is a in $\sF(\tC,\tS)$, as required.
\end{proof}

Combining this with (\ref{eqn:ext}), where $A = \unit$ and $B = \V$, we obtain the following result which is useful for proving comparison theorems.

\begin{corollary}
\label{cor:ext}
For all objects $\V$ of $\sF(\tC,\tS)$, the homomorphism
$$
H^j(\cG,V) \to H^j(\pi_1(\tC,\w),V)
$$
induced by the inclusion $\sF(\tC,\tS) \to \tC$, where $V = \w(\V)$, is an isomorphism when $j\le 1$ and an injection when $j=2$.
\end{corollary}

The following useful criterion for the freeness of $\u$ is an immediate consequence of the isomorphisms (\ref{eqn:ext_iso}) and the well-known fact that a pronilpotent Lie algebra $\n$ is free if and only $H^2(\n)$ vanishes. (See \cite[\S18]{hain-matsumoto:mem}.)

\begin{proposition}
\label{prop:u_free}
The Lie algebra $\u$ is a free pronilpotent Lie algebra if and only if $\Ext^2(\unit,\V)$ vanishes for all simple objects $\V$ of $\tS$. Equivalently, $\u$ is free if and only if $H^2(\u) = 0$. \qed
\end{proposition}

\subsection{Computation of $H^1(\u)$}

In this section we give a computation of $H^1(\u)\otimes_\kk\KK$ for all extension fields $\KK$ of $\kk$. Fix $\KK$ and set $S_\KK = S\times_\kk \KK$. Suppose that $V$ is a simple left $S_\KK$-module. Its endomorphism ring is thus a division algebra naturally isomorphic to $\End_{S_\KK} V$. We shall denote it by $D_V$. The $\KK$-algebra $\End_{D_V} V$ is a right $S_\KK$-module via precomposition and a left $S_\KK$-module via post composition. Denote its dual by $\End_{D_V}^\vee V$. It has commuting left and right $S_\KK$-actions induced by pre and post composition. Note that the map $S_\KK \to \Aut_{D_V} V$ induces a map
$$
\End_{D_V}^\vee V \to \cO(S_\KK) = \cO(S)\otimes_\kk \KK
$$
into the coordinate ring of $S_\KK$. It is a right and left $S_\KK$-module homomorphism.

Suppose that $\{V_\alpha\}$ is a complete set of representatives of the isomorphism classes of simple $S_\KK$-modules. Set $D_\alpha = D_{V_\alpha}$. The next result is an analogue of the Peter--Weyl theorem. It follows from Tannaka duality (by taking matrix entries). Alternatively, it can be proved by first reducing to the case where $\KK$ is algebraically closed using the Artin--Wedderburn Theorem.

\begin{proposition}
\label{prop:coord_ring}
The map
$$
\bigoplus_\alpha \End_{D_\alpha}^\vee V_\alpha \to \cO(S_\KK)
$$
is an isomorphism of right and left $S_\KK$-modules. \qed
\end{proposition}

Suppose that $V$ is a left $S_\KK$-module. The right action of $S_\KK$ on $\End_{D_V} V$ gives the space $\Hom_{S_\KK}(\End_{D_V} V,V)$ of left $S_\KK$-invariant maps the structure of a left $S_\KK$-module.

\begin{lemma}
\label{lem:wedderburn}
If $V$ is a simple $S_\KK$-module, then the map
$$
V \to \Hom_{S_\KK}(\End_{D_V} V,V)
$$
that takes $v \in V$ to $\phi \mapsto \phi(v)$ is an isomorphism of left $S_\KK$-modules.
\end{lemma}

\begin{proof}
Tensor both sides with an algebraic closure of $\KK$ and apply the Artin--Wedderburn Theorem.
\end{proof}

The conjugation action $g : g \mapsto ghg^{-1}$ of $\cG$ on itself induces a left $\cG$-action on $\u$ and a right $\cG$-action on $H^\bdot(\u)$. This descends to a right $S$-action on $H^\bdot(\u)$.

Set $\cG_\KK = \cG\times_\kk \KK$. We will regard the coordinate ring $\cO(S_\KK)$ as a left $\cG_\KK$-module via right multiplication. We therefore have the cohomology group
$$
H^\bdot(\cG_\KK,\cO(S_\KK)) \cong H^\bdot(\cG,\cO(S))\otimes_\kk \KK.
$$
The right action of $S$ on $\cO(S)$ gives this the structure of a right $S_\KK$-module. Likewise, we regard each $\End_{D_\alpha}^\vee V_\alpha$ as a left $\cG_\KK$-module and a right $S_\KK$-module.

\begin{proposition}[See also {\cite[Prop.~6.1]{brown:mot_periods}}]
\label{prop:coho}
There are natural right $S_\KK$-module isomorphisms
$$
H^j(\u)\otimes_\kk \KK \cong \bigoplus_\alpha H^j(\cG_\KK,\End_{D_\alpha}^\vee V_\alpha) \cong H^j(\cG,\cO(S))\otimes_\kk \KK.
$$
In particular, if each $D_\alpha$ is isomorphic to $\KK$ (i.e, each $V_\alpha$ is absolutely irreducible), then there is a right $S_\KK$-module isomorphism
$$
H^j(\u)\otimes_\kk \KK \cong \bigoplus_\alpha H^j(\cG_\KK,V_\alpha)\otimes_\KK V_\alpha^\vee.
$$
\end{proposition}

\begin{proof}
This follows directly from Proposition~\ref{prop:coord_ring} and the isomorphisms (\ref{eqn:ext_iso}) using Lemma~\ref{lem:wedderburn}.
\end{proof}

\subsection{Criterion for isomorphism}

The following criterion for when certain homomorphisms of affine groups are isomorphisms is used to establish some comparison isomorphisms. It is well-known. A proof can be deduced from the discussion in \cite[\S18]{hain-matsumoto:mem}. Suppose that $G$ and $G'$ are affine $\KK$ groups that are extensions of a reductive group $R$ by prounipotent groups $U$ and $U'$, respectively. Denote the Lie algebras of $U$ and $U'$ by $\u$ and $\u'$.

\begin{proposition}
\label{prop:gp_isom}
Suppose that $\phi : G \to G'$ is a homomorphism that commutes with the projections to $R$. Then $\phi$ is an isomorphism if and only if the induced map $H^j(\u') \to H^j(\u)$ is an isomorphism when $j=1$ and injective when $j=2$.  Equivalently, $\phi$ is an isomorphism if and only if the induced map $H^j(G',\cO(R)) \to H^j(G,\cO(R))$ is an isomorphism when $j=1$ and injective when $j=2$.
\end{proposition}

\subsection{Comparison homomorphisms}

In order to prove the comparison theorems, we will need to compare the fundamental groups of neutral tannakian categories over fields that may not be isomorphic. In this section, we give a general construction which will be needed to construct comparison isomorphisms.

Suppose that $\KK$ is an extension field of $\kk$. Suppose that $\tC'$ is a $\KK$-linear tannakian category and $\tS'$ is a full subcategory whose objects are semi-simple. Suppose that $\w' : \tC' \to \Vec_\KK$ is a fiber functor. Set
$\cG' = \pi_1(\sF(\tC',\tS'),\w')$. It is an affine $\KK$-group.

\begin{proposition}
\label{prop:comp_iso}
Suppose that $F : \sF(\tC,\tS) \to \sF(\tC',\tS')$ is a $\kk$-linear tensor functor. If the diagram
$$
\xymatrix{
\sF(\tC,\tS) \ar[r]^F\ar[d]_\w & \sF(\tC',\tS') \ar[d]^{\w'} \cr
\Vec_\kk \ar[r]^{\otimes_\kk\KK} & \Vec_\KK
}
$$
commutes, then $F$ induces a homomorphism $\cG' \to \cG\times_\kk \KK$.
\end{proposition}

\begin{proof}
It suffices to show that $F$ induces a homomorphism $\cG'(\LL) \to (\cG_\kk\times\KK)(\LL)$ for all extensions $\LL$ of $\KK$. Suppose that the natural isomorphism $\eta \in \cG'(\LL)$ and that $f : \V \to \W$ is a morphism of $\sF(\tC,\tS)$. Then, since $w'\circ F = w\otimes_\kk \KK$, the diagram
$$
\xymatrix@C=42pt{
(\w(\V)\otimes_\kk \KK) \otimes_\KK \LL \ar[r]^{\eta_{F\V}\otimes \id_\LL}\ar[d] & (\w(\V)\otimes_\kk \KK) \otimes_\KK \LL \ar[d] \cr
(\w(\W)\otimes_\kk \KK) \otimes_\KK \LL \ar[r]^{\eta_{F\V}\otimes \id_\LL} & (\w(\W)\otimes_\kk \KK) \otimes_\KK \LL
}
$$
commutes. This implies that the natural isomorphism $V\mapsto \eta_{FV}\otimes\id_\LL$ is an element of $\cG(\LL) = (\cG\times_\kk \KK)(\LL)$.
\end{proof}

\subsection{Variant: completion with constrained $\Ext^1$}
\label{sec:constrained_exts}

This section is needed in Section~\ref{sec:filtrations} where it is used to define the ``modular filtration'' on the coordinate ring of the relative completion of $\SL_2(\Z)$. It can also be used to define the weighted and crystalline completions of arithmetic fundamental groups \cite{hain-matsumoto:weighted}. The construction is inspired by \cite[Def.~1.4]{deligne-goncharov}.

Let $\tC$, $\tS$ and $\w$ be as above. Suppose that $\{\V_\alpha\}$ is a collection of simple objects of $\tS$ and that $\tE$ is a collection of subspaces of $E_\alpha$ of $\Ext^1_\tC(\unit,\V_\alpha)$. Define the category $\sF(\tC,\tS;\tE)$ to be the full subcategory of $\sF(\tC,\tS)$ consisting of those objects $\V$ for which the 1-extensions that occur in every subquotient of $\V$ are equivalent, via pushout and pullback, to sums of elements of $\cup_\alpha E_\alpha$ tensored with an object of $\tS$. It is tannakian subcategory of $\sF(\tC,\tS)$.

If $\tE \subseteq \tE'$, then, $\sF(\tC,\tS;\tE)$ is a full tannakian subcategory of $\sF(\tC,\tS;\tE')$, so that
$$
\pi_1(\sF(\tC,\tS;\tE'),\w) \to \pi_1(\sF(\tC,\tS;\tE),\w)
$$
is faithfully flat. In particular,
$$
\cO(\pi_1(\sF(\tC,\tS;\tE),\w)) \subseteq \cO(\pi_1(\sF(\tC,\tS;\tE'),\w)).
$$

\section{Relative unipotent completion of the modular group}
\label{sec:rel_comp_mod_gp}

In this section we construct various incarnations (Betti, \'etale and de~Rham) of a relative unipotent completion $\cG$ of $\SL_2(\Z)$ that is more general and larger than the one defined and studied in previous works such as \cite{hain:modular,brown:mmv,hain-matsumoto:mem}. This enlargement has the property that its category of representations is closed under restriction to and pushforward from all congruence subgroups, a fact needed to show that the generalized Hecke operators act on its ring $\Cl(\cG)$ of class functions, as we shall see in Part~\ref{part:class_fns}.

We will view every finite group $G$ as an affine algebraic group scheme over $\Q$ in the standard way; its coordinate ring is the algebra of functions $G \to \Q$. By taking inverse limits, we will regard every profinite group as an affine group scheme over $\Q$. Its group of $\kk$ rational points is the original (pro)finite group.

\subsection{The Betti incarnation}
\label{sec:betti}

For each positive integer $N$, define
$$
\rho_N : \SL_2(\Z) \to \SL_2(\Q)\times \SL_2(\Z/N)
$$
to be the homomorphism that is the inclusion on the first factor and is reduction mod $N$ on the second. It has Zariski dense image in the $\Q$ group $\SL_2\times \SL_2(\Z/N)$. Consequently, the inverse limit
\begin{equation}
\label{eqn:def_rho}
\rho : \SL_2(\Z) \to (\SL_2\times \SL_2(\Zhat))(\Q) = \SL_2(\Q) \times \SL_2(\Zhat),
\end{equation}
of the $\rho_N$ has Zariski dense image. 

Suppose that $\kk$ is a field of characteristic zero. Denote the category of representations of $\SL_2(\Z)$ in finite dimensional $\kk$ vector spaces by $\tC_\kk$. Each object of $\tC_\kk$ is regarded as a left $\SL_2(\Z)$-module. Let $\tS_\kk$ be the full subcategory of $\tC_\kk$ consisting of those representations that factor through $\rho$. In other words, objects of $\tS_\kk$ are restrictions of rational representations of $\SL_2/\kk\times \SL_2(\Zhat)$ and are therefore completely reducible.

The category $\tC_\kk$ is a $\kk$-linear tannakian category. It is neutralized by the functor $\w_\kk: \tC_\kk \to \Vec_\kk$ that takes a representation to its underlying vector space.

\begin{definition}
The {\em (Betti incarnation of the) relative completion} $\cG^B$ of $\SL_2(\Z)$ is the affine $\Q$-group defined by
$$
\cG^B := \pi_1(\sF(\tC_\Q,\tS_\Q),\w_\Q).
$$
\end{definition}

It is an extension
\begin{equation}
\label{eqn:betti_extension}
1 \to \U^B \to \cG^B \to \SL_{2/\Q}\times \SL_2(\Zhat)_{/\Q} \to 1
\end{equation}
where $\U^B$ is prounipotent.  There is a canonical Zariski dense homomorphism
$$
\rhotilde : \SL_2(\Z) \to \cG^B(\Q)
$$
whose composition with the quotient map $\cG^B(\Q) \to (\SL_2\times \SL_2(\Zhat))(\Q)$ is the representation (\ref{eqn:def_rho}).

\begin{proposition}
\label{prop:B_base_change}
The functor $\underline{\blank}\otimes_\Q \kk : \sF(\tC_\Q,\tS_\Q) \to \sF(\tC_\kk,\tS_\kk)$ induces an isomorphism $\pi_1(\sF(\tC_\kk,\tS_\kk),\w_\kk) \to \cG^B \times_\Q \kk$.
\end{proposition}

We will denote $\cG^B\times_\Q \kk$ by $\cG^B_\kk$. The proposition implies that this notation is unambiguous.

\begin{proof}
Since the diagram
$$
\xymatrix{
\sF(\tC_\Q,\tS_\Q) \ar[r]^(.55){\underline{\blank}\otimes_\Q \kk} \ar[d]_{\w_\Q} & \sF(\tC_\kk,\tS_\kk) \ar[d]^{\w_\kk} \cr
\tC_\Q \ar[r]^{\underline{\blank}\otimes_\Q \kk} & \tC_\kk
}
$$
commutes, Proposition~\ref{prop:comp_iso} implies that the functor in the statement induces a homomorphism
$$
\pi_1(\sF(\tC_\kk,\tS_\kk),\w_\kk) \to \pi_1(\sF(\tC_\Q,\tS_\Q),\w_\Q)\times_\Q \kk = \cG^B\times_\Q \kk.
$$
Both groups have proreductive quotient $\SL(H_\kk) \times \SL_2(\Zhat)$. So we have to show that the homomorphism restricts to an isomorphism on prounipotent radicals. To do this we use Proposition~\ref{prop:gp_isom}.

For all objects $\V$ of $\sF(\tC_\Q,\tS_\Q)$, the diagram
$$
\xymatrix{
H^j(\cG^B,V)\otimes_\Q \kk \ar[r] \ar[d] & H^j(\pi_1(\sF(\tC_\kk,\tS_\kk),\w_\kk),V\otimes_\Q \kk) \ar[d]
\cr
H^j(\pi_1(\tC_\Q,\w_\Q),V)\otimes_\Q\kk \ar[r] & H^j(\pi_1(\tC_\kk,\w_\kk),V\otimes_\Q\kk)
}
$$
commutes, where $V=\w_\Q(\V)$, the top arrow is induced by the homomorphism above and the bottom arrow by $\tC_\Q \to \tC_\kk$. The bottom row can be replaced by
$$
H^j(\SL_2(\Z),V)\otimes_\Q \kk \to H^j(\SL_2(\Z),V\otimes_\Q\kk),
$$
which is clearly an isomorphism. Corollary~\ref{cor:ext} implies that the vertical arrows are isomorphisms when $j=0,1$ and injective when $j=2$. Consequently, the top arrow is an isomorphism when $j=0,1$ and injective when $j=2$.
 
Now take $V= \cO(\SL_2\times\SL_2(\Zhat))$. The result of the previous paragraph and Proposition~\ref{prop:coho} imply that
$$
H^j(\cG^B,\cO(\SL_2\times\SL_2(\Zhat)))\otimes_\Q\kk \to H^j(\pi_1(\sF(\tC_\kk,\tS_\kk),\w_\kk),\cO(\SL_2\times\SL_2(\Zhat))\otimes_\Q\kk)
$$
is an isomorphism in degree 1 and injective in degree 2. The result now follows from Proposition~\ref{prop:gp_isom}.
\end{proof}

Suppose that $\kkbar$ is an algebraically closed field. For each irreducible character $\chi : \SL_2(\Zhat) \to \kkbar$ of $\SL_2(\Zhat)$, choose an irreducible left $\SL_2(\Zhat)$-module $V_\chi$. We will regard its dual $V_\chi^\vee$ as a right $\SL_2(\Zhat)$-module.

Recall that $\G(N)$ denotes the full level $N$ subgroup of $\SL_2(\Z)$. For each $m\ge 0$, $H^1(\G(N),S^m H)\otimes\kk$ is a left $\SL_2(\Z/N)$-module. Recall that for a character $\psi : \SL_2(\Z/N) \to \kkbar$, $H^1(\G(N),S^m H)_\psi$ denotes the $\psi$-isotypical summand of $H^1(\G(N),S^m H)\otimes\kkbar$.

\begin{proposition}
\label{prop:H1u}
There are canonical right $\SL_2\times \SL_2(\Zhat)$ isomorphisms
\begin{align*}
H^1(\u^B)\otimes_\Q\kkbar & \cong \bigoplus_{m\ge 0}\bigoplus_{\chi} H^1(\SL_2(\Z);S^m H\otimes V_\chi)\otimes (S^m H^\vee \boxtimes V_\chi^\vee)
\cr
&\cong \varinjlim_N
\bigoplus_{m\ge 0}\bigoplus_{\psi} H^1(\G(N),S^m H)_{\psi^\vee}\otimes S^m H^\vee,
\end{align*}
where $\chi$ ranges over the irreducible $\kkbar$ characters of $\SL_2(\Zhat)$ and $\psi$ over the irreducible $\kkbar$ characters of $\SL_2(\Z/N)$. The cohomology of $\u^B$ vanishes in degree 2.
\end{proposition}

\begin{proof}
The first statement is an immediate consequence of Proposition~\ref{prop:coho}. Since $\SL_2(\Z)$ is virtually free, $H^j(\SL_2(\Z),M)$ vanishes for all $j > 1$ and all divisible $\SL_2(\Z)$-modules $M$. Proposition~\ref{prop:ext} now implies the vanishing of $H^2(\u)$ via the isomorphisms (\ref{eqn:ext_iso}).
\end{proof}

\begin{remark}
Later it will be convenient to make this into a left $S$-module isomorphism. This is easily achieved by instead taking the action of $\cG$ on $H^\bdot(\u)$ to be the one induced by the ``right conjugation'' action: $g : u \mapsto g^{-1}ug$.
\end{remark}

The next result follows from the vanishing of $H^2(\u)$ and Proposition~\ref{prop:u_free}.

\begin{corollary}
The Lie algebra $\u^B$ is a free pronilpotent Lie algebra.
\end{corollary}

\subsection{The $\ell$-adic \'etale incarnation and the Galois action}

Fix a prime number $\ell$. The homomorphism (\ref{eqn:def_rho}) extends to a continuous homomorphism
$$
\rho^\et_\ell : \SL_2(\Z)^\wedge \to \SL_2(\Ql) \times \SL_2(\Zhat).
$$
Denote the category of continuous (left) representations of $\SL_2(\Z)^\wedge$ in finite dimensional $\Ql$ vector spaces by $\tC^\cts_\Ql$. Denote the full subcategory of $\tC^\cts_\Ql$ consisting of those representations that factor through $\rho^\et_\ell$. Set
$$
\cG^\et_\ell := \pi_1(\sF(\tC^\cts_\Ql,\tS^\cts_\Ql),\w_\Ql)
$$
where $\w_\Ql$ takes a representation to its underlying $\Ql$ vector space. The homomorphism $\rho^\et_\ell$ lifts to a canonical Zariski dense homomorphism
$$
\rhotilde^\et_\ell : \SL_2(\Z)^\wedge \to \cG^\et_\ell(\Ql).
$$

The restriction functor
$$
\sF(\tC^\cts_\Ql,\tS^\cts_\Ql) \to \sF(\tC_\Ql,\tS_\Ql)
$$
is an exact tensor functor and thus induces a homomorphism $\cG^B_\Ql \to \cG^\et_\ell$.

\begin{proposition}
The homomorphism $\cG^B_\Ql \to \cG^\et_\ell$ induced by restriction is an isomorphism of $\Ql$ groups. Moreover, the diagram
$$
\xymatrix{
\SL_2(\Z) \ar[r]^\rhotilde \ar[d] & \cG^B(\Ql) \ar[d]^\simeq \cr
\SL_2(\Z)^\wedge \ar[r]^{\rhotilde^\et_\ell} & \cG^\et_\ell(\Ql)
}
$$
commutes.
\end{proposition}

\begin{proof}
Suppose that $V$ is an object of $\sF(\tC^\cts_\Ql,\tS^\cts_\Ql)$. There are natural isomorphisms
$$
H^j_\cts(\SL_2(\Z)^\wedge,V) \cong \Ext^j_{\tC^\cts_\Ql}(\Ql,V) \text{ and }
H^j(\SL_2(\Z),V) \cong \Ext^j_{\tC_\Ql}(\Ql,V)
$$
where $H^j_\cts$ denotes continuous group cohomology. We have the commutative diagram
$$
\xymatrix{
H^j(\cG^\et_\ell,V) \ar[r]\ar[d] & H^j(\cG^B_\Ql,V) \ar[d]
\cr
H^j_\cts(\SL_2(\Z)^\wedge,V) \ar[r] & H^j(\SL_2(\Z),V)
}
$$
whose horizontal maps are induced by restriction. By \cite[p.~16]{serre:galois}, $\SL_2(\Z)$ is a ``good group'', which implies that the bottom arrow is an isomorphism for all $j\ge 0$. Corollary~\ref{cor:ext} implies that the vertical arrows are isomorphisms when $j\le 1$ and are injective when $j=2$. It follows that the top arrow is an isomorphism when $j\le 1$ and injective when $j=2$. The result now follows from Proposition~\ref{prop:gp_isom}.
\end{proof}

Recall from Corollary~\ref{cor:galois_action} that the absolute Galois group $\Gal(\Qbar/\Q)$ acts on $\SL_2(\Z)^\wedge$ via the isomorphism $\pi_1(\M_{1,1/\Qbar},\partial/\partial q) \cong \SL_2(\Z)^\wedge$. This Galois action induces one on $\SL_2\times_\Q\Ql$ and also on $\SL_2(\Zhat)$ via its action on the Tate module. With these actions, the homomorphism
$$
\rho^\et_\ell : \SL_2(\Z)^\wedge \to \SL_2(\Ql)\times \SL_2(\Zhat)
$$
is Galois equivariant. The functoriality of relative completion yields the following result.

\begin{corollary}
\label{cor:gal_action}
For each prime number $\ell$, the absolute Galois group acts on $\cG^\et_\ell$ and the natural representation
$$
\rhotilde^\et_\ell : \SL_2(\Z)^\wedge \to \cG^\et_\ell(\Ql) 
$$
is $\Gal(\Qbar/\Q)$-equivariant.
\end{corollary}

\begin{remark}
The weight filtration on $\cO(\cG^\et_\ell)$ can be constructed using the fact that $\cO(\cG^\et_\ell)$ is a module over the weighted completion (as defined in \cite{hain-matsumoto:weighted}) of $\SL_2(\Zhat)$ relative to the natural homomorphism to the $\Ql$ points of $\GL_2\times \SL_2(\Zhat)$ and the central cocharacter $c : \Gm \to \GL_2$ defined by $c(t) = t^{-1}\id$. This weight filtration agrees, via the comparison isomorphism, with the one on $\cG^B$ that is constructed in \cite{hain:malcev} using Hodge theory. The details are omitted as we will not use the weight filtration on $\cO(\cG^\et_\ell)$ in this paper.
\end{remark}

\subsection{The de~Rham incarnation}

We use the notation of Section~\ref{sec:mod_tower}. In particular the modular stacks $X(N)$ and $Y(N)$ are defined over $\KDR$ and are geometrically connected. Suppose that $\KK\subseteq \C$ is an extension field of $\KDR$. Denote $X(N)\times_\KDR\KK$ and $Y(N)\times_\KDR\KK$ by $X(N)_\KK$ and $Y(N)_\KK$, respectively.

We will call a connection $(\cV,\nabla)$ over $\M_{1,1/\KK}$ {\em virtually locally nilpotent} if there is an $N\ge 1$ such that its pullback $\pi_N^\ast \cV$ to $Y(N)_\KK$ has an extension
$$
\nabla_N : \cVbar_N \to \cVbar_N \otimes \Omega^1_{X(N)_\KK}(\log C_N)
$$
to $X(N)_\KK$ that is locally nilpotent. That is, the residue of $\nabla_N$ at each cusp $P\in C_N$ is nilpotent. The connection $(\cVbar_N,\nabla)$ has a natural action of $\SL_2(\Z/N)$. It is a $\KK$-form of Deligne's canonical extension of the corresponding connection over $Y(N)^\an$ to $X(N)^\an$. If it exists, it is unique as, by \cite[Prop.~II.5.2]{deligne:diff_eq}, the canonical extension is unique after extending scalars to $\C$ and because the restriction of this isomorphism to the open subset $Y(N)_\KK$ is $\id : \pi_N^\ast \cV \to \pi_N^\ast \cV$, which is defined over $\KK$. The canonical extension of the tensor product of two locally nilpotent connections is the tensor product of their canonical extensions.

The virtually locally nilpotent connections over $\M_{1,1/\KK}$ form a $\KK$-linear tannakian category which we shall denote by $\tC^\DR_\KK$. The connections $\sS^m_N\otimes_\KDR\KK$ constructed in Proposition~\ref{prop:sS^m} are virtually locally nilpotent for all $m\ge 0$ and $N\ge 1$ and are thus objects of $\tC^\DR_\KK$. They are, in fact, semi-simple objects. Denote the full tannakian subcategory of $\tC^\DR_\KK$ generated by their simple summands by $\tS^\DR_\KK$. Let
$$
\w^\DR_\KK : \tC^\DR_\KK \to \Vec_\KK
$$
be the functor that takes a connection $(\cV,\nabla)$ to the fiber of the canonical extension $\cVbar_N$ of $\pi_N^\ast \cV$ over the distinguished cusp $P_N$ of $X(N)$, where $N$ is chosen to guarantee that $\pi_N^\ast \cV$ is locally nilpotent. It is a well defined, faithful tensor functor which neutralizes $\tC^\DR_\KK$. Define
$$
\cG^\DR = \pi_1(\sF(\tC^\DR_\KDR,\tS^\DR_\KDR),\w^\DR_\KDR).
$$

The goal of the rest of this section is to construct the comparison homomorphism (\ref{eqn:comp_DR}) below and prove that it is an isomorphism. The first step in the proof is to relate Yoneda extensions in $\tC^\DR_\KK$ to algebraic de~Rham cohomology. We begin in a simpler setting.

Suppose that $X$ is a smooth projective curve over $\KK$ and that $Y = X-D$ is a Zariski open defined over $\KK$. Denote the category of locally nilpotent connections over $Y$ by $\tC^\nilp_Y$. Objects of $\tC^\nilp_Y$ are connections $\cV$ over $Y$ that have an extension $\nabla : \cVbar \to \cVbar\otimes \Omega^1_X(\log D)$ to $X$, where $\nabla$ has nilpotent residue at each geometric point of $D$. As above, such an extension, if it exists, is unique.

\begin{lemma}
If $\cV$ is a locally nilpotent connection over $Y$, then there is a natural isomorphism $H^\bdot_\DR(Y,\cV) \cong \H^\bdot(X,\cVbar\otimes\Omega_X^\bdot(\log D))$.
\end{lemma}

\begin{proof}
As noted in \cite[Rem.~4.5]{steenbrink-zucker}, the map
$
\cVbar\otimes\Omega_X^\bdot(\log D) \to j_\ast (\cV \otimes_{\cO_Y}\Omega^\bdot_Y)
$
induced by the inclusion $j: Y \hookrightarrow X$ is a quasi-isomorphism after tensoring with $\C$. Since it is defined over $\KK$, it is a quasi-isomorphism. It thus induces an isomorphism
$$
\H^\bdot(X,\cVbar\otimes\Omega_X^\bdot(\log D)) \overset{\simeq}{\longrightarrow} \H^\bdot(Y,\cV\otimes_{\cO_Y}\Omega_Y^\bdot) = H^\bdot(Y,\cV).
$$
\end{proof}

\begin{proposition}
\label{prop:ext_DR}
There is a natural isomorphism
\begin{equation}
\label{eqn:ext_to_coho}
\Ext^\bdot_{\tC_Y^\nilp}(\cO_Y,\cV) \overset{\simeq}{\longrightarrow} H^\bdot_\DR(Y,\cV).
\end{equation}
Both sides vanish in degrees $\ge 2$ when $D$ is non-empty.
\end{proposition}

\begin{proof}
In degree 0 the map is the isomorphism
\begin{multline*}
\Ext^0_{\tC_Y^\nilp}(\cO_Y,\cV) = \Hom_{\tC_Y^\nilp}(\cO_Y,\cV) \cr
\overset{\simeq}{\longrightarrow}\ker\{\nabla : \cVbar \to \cVbar\otimes \Omega^1_X(\log D)\} = H^0_\DR(Y,\cV).
\end{multline*}
We now focus on positive degrees. We have the spectral sequence
$$
E_1^{s,t} = H^t(X,\cVbar\otimes\Omega^s_X(\log D)) \implies H^{s+t}_\DR(Y,\cV).
$$
Since $X$ is a curve, $E_1^{s,t}$ vanishes when either $s$ or $t$ is $>1$. This implies that the spectral sequence degenerates at $E_2$ and that there is an exact sequence
\begin{multline}
\label{eqn:H1}
0 \to \coker\{H^0(X,\cVbar) \to H^0(X,\cVbar\otimes \Omega^1_X(\log D)) \} \to H^1_\DR(Y,\cV)
\cr
\to \ker\{H^1(X,\cVbar) \to H^1(X,\cVbar\otimes \Omega^1_X(\log D))\} \to 0
\end{multline}
and an isomorphism
$$
H^2_\DR(Y,\cV) \cong \coker\{H^1(X,\cVbar) \to H^1(X,\cVbar\otimes \Omega^1_X(\log D))\}.
$$

Suppose that
\begin{equation}
\label{eqn:conn_extn}
0 \to \cV \to \cE \to \cO_Y \to 0
\end{equation}
is an extension in $\tC_Y^\nilp$. Let $0\to \cVbar \to \cEbar \to \cO_X \to 0$ be the corresponding exact sequence of canonical extensions in the category of vector bundles over $X$. Yoneda equivalent extensions in $\tC_Y^\nilp$ give rise to Yoneda equivalent extensions of holomorphic vector bundles. Taking the underlying extension of vector bundles defines a homomorphism
$$
\Ext^1_{\tC_Y^\nilp}(\cO_Y,\cV) \to H^1(X,\cVbar).
$$

The kernel of this homomorphism consists of Yoneda equivalence classes of connections on $\cEbar = \cO_X\ee \oplus \cVbar$ whose restriction to $\cV$ is the given connection. Each such connection $\nabla$ is determined by $\nabla \ee$ and each $\w \in H^0(X,\cVbar\otimes\Omega^1_X(\log D))$ determines a connection on $\cEbar$ by defining $\nabla \ee = \w$. The kernel is thus a quotient of $H^0(X,\cVbar\otimes\Omega^1_X(\log D))$. Two connections on $\cEbar$ are Yoneda equivalent in $\tC_Y^\nilp$ if and only if one is pulled back from the other along an automorphism $\Phi$ of $\cEbar$ whose restriction to $\cVbar$ is the identity and $\Phi(\ee) = \ee + \varphi$, where $\varphi \in H^0(X,\cVbar)$. The pullback of a connection $\nabla$ on $\cEbar$ along $\Phi$ satisfies
$$
\Phi^\ast \nabla - \nabla = \nabla \varphi \in H^0(X,\cVbar\otimes\Omega_X^1(\log D)).
$$
Thus two such connections on $\cEbar$ are Yoneda equivalent if and only if they differ by an element of $\nabla H^0(X,\cVbar)$. This implies that
$$
0 \to H^0(X,\cVbar\otimes \Omega^1_X(\log D))/\nabla H^0(X,\cVbar) \to \Ext^1_{\tC_Y^\nilp}(\cO_Y,\cV) \to H^1(X,\cVbar)
$$
is exact.

To understand the image of the right-hand map and to define the homomorphism (\ref{eqn:ext_to_coho}) we use \v{C}ech cochains. Let $\fU=\{U_0,U_1\}$ be an open covering of $X$ by two affine open subsets, both defined over $\KK$. By Leray's Theorem, the \v{C}ech cochains $C^\bdot(\fU,\cVbar\otimes \Omega^\bdot_X(\log D))$ compute the right hand side of (\ref{eqn:ext_to_coho}).

Consider the extension (\ref{eqn:conn_extn}). Since $U_j$ is affine, the restriction of $\cEbar$ to $U_j$ splits. Let $\ee_j : \cO_{U_j} \to \cEbar|_{U_j}$ be a splitting. The class of the extension $0\to\cVbar\to \cEbar \to \cO_X \to 0$ is represented by the 1-cocycle
$$
v_{0\, 1} := \ee_1 - \ee_0 \in H^0(U_0\cap U_1,\cVbar).
$$
The image of the class of $\cEbar$ in $H^1(X,\cVbar\otimes \Omega^1_X(\log D))$ is represented by $\nabla v_{0\, 1}$. It vanishes as it is the coboundary of $(\nabla \ee_j) \in C^0(\fU,\cVbar\otimes \Omega^1_X(\log D))$. Thus the image of $\Ext^1_{\tC_Y^\nilp}(\cO_Y,\cV)$ in $H^1(X,\cVbar)$ lies in the kernel of $H^1(X,\cVbar) \to H^1(X,\cVbar\otimes\Omega^1_X(\log D))$. Next we show that the image of $\Ext^1_{\tC_Y^\nilp}(\cO_Y,\cV)$ is the kernel of this map.

Every element of the kernel of $H^1(X,\cVbar) \to H^1(X,\cVbar\otimes\Omega^1_X(\log D))$ is represented by a 1-cocycle $(v_{0\, 1})\in C^1(\fU,\cVbar)$ for which we can find $(\w_j) \in C^0(\fU,\cVbar\otimes \Omega_X^1(\log D))$ satisfying $\nabla v_{0\, 1} = \w_1 - \w_0$. Now let $\cEbar$ be the vector bundle over $X$ whose restriction to $U_j$ is $\cO_X\ee_j \oplus \cVbar|_{U_j}$. These are glued on $U_0\cap U_1$ by identifying $\ee_1$ with $\ee_0 + v_{0\, 1}$. The connection on $\cEbar$ is defined by $\nabla \ee_j = \w_j$. These agree on $U_0\cap U_1$. This establishes surjectivity. We therefore have an exact sequence
\begin{multline}
\label{eqn:Ext}
0 \to \coker\{H^0(X,\cVbar) \to H^0(X,\cVbar\otimes \Omega^1_X(\log D)) \} \to \Ext_{\tC_Y^\nilp}^1(Y,\cV)
\cr
\to \ker\{H^1(X,\cVbar) \to H^1(X,\cVbar\otimes \Omega^1_X(\log D))\} \to 0.
\end{multline}

In terms of the notation above, the map $\Ext^1_{\tC_Y^\nilp}(\cO_Y,\cV) \to H^1_\DR(Y,\cV)$ takes the extension (\ref{eqn:conn_extn}) to the class represented by the cocycle
$$
\big((\w_j),(v_{0\,1})\big) \in C^0(\fU,\cVbar \otimes \Omega^1(\log D)) \oplus C^1(\fU,\cVbar).
$$
This maps the exact sequence (\ref{eqn:Ext}) to (\ref{eqn:H1}). Since it is the identity on the kernel and cokernel, it is an isomorphism. This completes the proof in degree 1.

Next we prove injectivity in degree 2. We restrict to the case where $D$ is non-trivial, which is all we shall need. In this case $Y$ is affine, which implies that $H^2_\DR(Y,\cV)$ vanishes. So, to prove the result, we need to show that $\Ext^2_{\tC_Y^\nilp}(\cO_Y,\cV) = 0$. To prove the vanishing of the class of the Yoneda 2-extension
\begin{equation}
\label{eqn:ext2}
0 \to \cV \to \cE^0 \to \cE^1 \to \cO_Y \to 0
\end{equation}
it suffices to construct $\cE \in \tC_Y^\nilp$ that is an extension of $\cO_Y$ by $\cE^0$ and has the property that $\cE/\cV$ is isomorphic to the extension
$$
0 \to \cF \to \cE^1 \to \cO_Y \to 0
$$
where $\cF = \ker\{\cE^1 \to \cO_Y\}$. This can be seen using the long exact sequence of Ext groups $\Ext_{\tC_Y^\nilp}^\bdot(\cO_Y,\underline{\blank})$ associated to the short exact sequence $0 \to \cV \to \cE^0 \to \cF \to 0$.

Let $\cEbar_j = \cVbar|_{U_j} \oplus \cEbar^0|_{U_j}$ endowed with the direct sum connection, denoted $\nabla_j$. As vector bundles, their restrictions to $U_0\cap U_1$ are isomorphic, but as connections, they differ by a 1-form:
$$
\nabla_1 - \nabla_0 = \psi_{0\, 1} \in H^0(U_0\cap U_1,\cVbar\otimes \Omega^1_X(\log D)) = C^1(\fU,\cVbar\otimes \Omega^1_X(\log D)).
$$
Since $H^2_\DR(Y,\cV)=0$, $(\psi_{0\,1})$ is the coboundary of a 1-cochain
$$
\big((\psi_j),(u_{0\, 1}) \big) \in C^0(\fU,\cVbar\otimes \Omega^1_X(\log D)) \oplus C^1(\fU,\cVbar).
$$
If we replace $\psi_j$ by $\nabla_j - \psi_j$ and glue $\cEbar_0$ to $\cEbar_1$ using the transition function
$$
\id + u_{0\, 1} \in H^0(U_0\cap U_1,\cVbar) \subset H^0\big(U_0\cap U_1,\Hom_{\cO_{U_0\cap U_1}}(\cEbar_0,\cEbar_1)\big),
$$
we obtain a logarithmic connection $(\cEbar,\nabla)$ over $X$ whose restriction to $U_j$ is $(\cEbar_j,\nabla_j - \psi_j)$. This is an element of $\Ext^1_{\tC_Y^\nilp}(\cO_Y,\cE^0)$ whose existence implies the vanishing of $\Ext_{\tC_Y^\nilp}^2(\cO_Y,\cV)$, as explained above.

This also implies the vanishing of $\Ext^2_{\tC_Y^\nilp}(\cA,\cB)$ for all $\cA,\cB \in \tC_Y^\nilp$ when $Y$ is affine. Every element
$
0\to \cV \to \cE^{j+1} \to \dots \to \cE^0 \to \cO_Y \to 0
$
of $\Ext^{j+2}_{\tC_Y^\nilp}(\cO_Y,\cV)$ can be written as the image of
$$
[0 \to \cV \to \cE^{j+1} \to \dots \to \cE^2 \to \cF \to 0] \otimes [0 \to \cF \to \cE^1 \to \cE^0 \to \cO_Y \to 0]
$$
under the Yoneda product
$
\Ext^2_{\tC_Y}(\cO_Y,\cF) \otimes \Ext^j_{\tC_Y^\nilp}(\cF,\cV) \to \Ext^{j+2}_{\tC_Y^\nilp}(\cO_Y,\cV),
$
where $\cF = \ker\{\cE^1 \to \cE^0\}$. Since $\Ext^2_{\tC_Y^\nilp}(\cO_Y,\cF)$ vanishes, $\Ext^{j+2}_{\tC_Y^\nilp}(\cO_Y,\cV)$ vanishes for all $j\ge 0$.
\end{proof}

\begin{corollary}
\label{cor:ext_DR}
For all $\cV$ in $\tC^\DR_\KK$ there is a natural isomorphism
$$
\Ext^\bdot_{\tC^\DR_\KK}(\cO_{\M_{1,1/\KK}},\cV) \overset{\simeq}{\longrightarrow} H^\bdot_\DR(\M_{1,1/\KK},\cV).
$$
Both sides vanish in degrees $\ge 2$.
\end{corollary}

\begin{proof}
Choose $N \ge 3$ for which $\cV_N := \pi_N^\ast\cV$ is locally nilpotent. Since $N\ge 3$, $X(N)_\KK$ is a smooth projective curve. The action of $\SL_2(\Z/nN)$ on $\cV_{nN}$ induces an action on $\Ext_{\tC^\nilp_{Y(nN)/\KK}}^\bdot(\cO_{Y(nN)/\KK},\cV_{nN})$. Since the pullback of every virtually nilpotent connection on $\M_{1,1/\KK}$ to $Y(nN)_\KK$ is nilpotent for some $n\ge 1$, the homomorphism
$$
\Ext^\bdot_{\tC^\DR_\KK}(\cO_{\M_{1,1}/\KK},\cV) \overset{\simeq}{\longrightarrow} \varprojlim_n \Ext^\bdot_{\tC^\nilp_\KK}(\cO_{Y(nN)/\KK},\cV_{nN})^{\SL_2(\Z/nN)}
$$
induced by pullback is an isomorphism. The result now follows from Proposition~\ref{prop:ext_DR} as pullback also induces an isomorphism
$$
H^\bdot_\DR(\M_{1,1/\KK},\cV) \overset{\simeq}{\longrightarrow}  \varprojlim_n H^\bdot_\DR(Y(nN)_\KK,\cV_{nN})^{\SL_2(\Z/nN)}.
$$
\end{proof}

The next result is the de~Rham analogue of Proposition~\ref{prop:B_base_change}.

\begin{proposition}
\label{prop:DR_base_change}
The functor $\underline{\blank}\otimes_\KDR \KK : \sF(\tC^\DR_\KDR,\tS^\DR_\KDR) \to \sF(\tC^\DR_\KK,\tS^\DR_\KK)$ induces an isomorphism $\pi_1(\sF(\tC^\DR_\KK,\tS^\DR_\KK),\w^\DR_\KK) \to \cG^\DR \times_\KDR \KK$.
\end{proposition}

We will denote $\cG^\DR\times_\KDR \KK$ by $\cG^\DR_\KK$. The proposition implies that this notation is unambiguous.

\begin{proof}
The proof is similar to that of Proposition~\ref{prop:B_base_change}, so we only sketch it. As in the Betti case, the functor $\underline{\blank}\otimes_\KDR \KK$ induces a homomorphism
$$
\pi_1(\sF(\tC_\KK,\tS_\KK),\w^\DR_\KK) \to \pi_1(\sF(\tC_\KDR,\tS_\KDR),\w^\DR_\KDR)\times_\KDR \KK.
$$
Both groups have proreductive quotient $\SL(H^\DR) \times \SL_2(\Zhat)$. So we have to prove that the homomorphism restricts to an isomorphism on prounipotent radicals.

Every simple object of $\tC^\DR_\KK$ is a summand of $\sS^m_N\otimes_\KDR\KK$ for unique $m$ and some $N>0$. The simple summands of $\sS^m_N\otimes_\KDR\KK$ correspond to simple $\KK[\SL_2(\Z/N)]$-modules. To prove that the homomorphism is an isomorphism, it suffices to show that 
$$
\Ext^j_{\tC^\DR_\KDR}(\cO_{\M_{1,1/\KDR}},\sS^m_N)\otimes_\KDR\KK \to \Ext^j_{\tC^\DR_\KK}(\cO_{\M_{1,1/\KK}},\sS^m_N\otimes_\KDR\KK)
$$
is an isomorphism when $j\le 1$ and injective when $j=2$ for all $m\ge 0$ and $N\ge 1$. But this follows from Corollary~\ref{cor:ext_DR} as
$$
H^j_\DR(\M_{1,1/\KDR},\sS^m_N)\otimes_\KDR \KK \to H^j_\DR(\M_{1,1/\KK},\sS^m_N\otimes_\KDR\KK)
$$
is an isomorphism for all $m$ and $N$.
\end{proof}

The comparison homomorphism
\begin{equation}
\label{eqn:comp_DR}
\cG^B\times_\Q \C \to \cG^\DR_\KDR \times_\KDR \C
\end{equation}
is induced by the tensor functor $\sF(\tC^\DR,\tS^\DR) \to \sF(\tC_\C,\tS_\C)$ that takes an object $\cV$ of $\tC^\DR$ to the monodromy representation
$$
\pi_1(\M_{1,1},\partial/\partial q) \to \Aut (\w^\DR(\cV)\otimes_\KDR\C)
$$
of $\SL_2(\Z)$ in the fiber over $\partial/\partial q$, which is naturally isomorphic to $\w^\DR(\cV)\otimes_\KDR\C$.

\begin{proposition}
The comparison homomorphism (\ref{eqn:comp_DR}) is an isomorphism.
\end{proposition}

\begin{proof}
In view of Propositions~\ref{prop:B_base_change} and \ref{prop:DR_base_change}, to prove that the comparison homomorphism (\ref{eqn:comp_DR}) is an isomorphism, it suffices to prove that $\sF(\tC^\DR_\C,\tS^\DR_\C) \to \sF(\tC_\C,\tS_\C)$ induces an isomorphism $\cG^B_\C \to \cG^\DR_\C$. The homomorphism induced on their proreductive quotients is the isomorphism $\SL(H_\C)\times\SL_2(\Zhat) \to \SL(H^\DR_\C)\times \SL_2(\Zhat)$ induced by the comparison isomorphism $H_\C \cong H^\DR_C$.

To prove that $\cG^B_\C \to \cG^\DR_\C$ is an isomorphism, it suffices, by Propositions~\ref{prop:coho} and \ref{prop:gp_isom}, to show that
$$
\Ext^j_{\tC^\DR_\C}(\cO_{\M_{1,1/\C}},\cV) \to \Ext^j_{\tC_\C}(\C,\w^\DR_\C(\cV))
$$
is an isomorphism for all $\cV$ in $\tC^\DR_\C$. But this is true as the Riemann--Hilbert correspondence \cite[Thm.~II.5.9]{deligne:diff_eq} implies that the functor $\tC^\DR_\C \to \tC_\C$ is an equivalence of categories.
\end{proof}

\subsection{The mixed Hodge structure}

The existence of a mixed Hodge structure on the relative completion of $\pi_1(\M_{1,1},b)$ for any base point $b\in\h$ follows from the main result of \cite{hain:malcev}. Here our preferred base point is $\partial/\partial q$. In this case, the MHS is the associated limit MHS. A proof of the existence of this limit MHS can be found in \cite[\S7]{hain:modular}.

\begin{theorem}
\label{thm:mhs}
The coordinate ring $\cO(\cG)$ of the relative completion of $\pi_1(\M_{1,1},\partial/\partial q)$ has a natural MHS.
\end{theorem}

The Hodge filtration is constructed using smooth forms.

\begin{remark}
One can construct the corresponding Hodge and weight filtrations Hodge on $\cO(\cG^\DR)$ using the bar construction on the Thom--Whitney resolution of the logarithmic de~Rham complex of $\M_{1,1/\KDR}$ with coefficients in the algebraic version of the connection denoted $\cO(P)$ in \cite[\S4]{hain:malcev}. In the simpler case of the completion of $\SL_2(\Z)$ with respect to the inclusion $\SL_2(\Z)\to\SL_2(\Q)$, this connection is
$$
\cO(P) = \bigoplus_{m\ge 0} S^m\cH \otimes (S^m H^\DR)^\vee.
$$
The augmentation is restriction $\delta : \cO(P) \to \cO(SL(H^\DR)) = \oplus \End^\vee(S^m H^\DR)$ to the fiber over the cusp. The construction for unipotent completion is explained in \cite[\S13.3]{hain:prospects}.
\end{remark}

\section{Variants}
\label{sec:variants}

There are many natural variants of this construction. Here we enumerate a few.
\begin{enumerate}

\item
\label{item:cong_subgps}
One can replace $\SL_2(\Z)$ by a congruence subgroup $\G$ and the proreductive group $\SL_2 \times \SL_2(\Zhat)$ by $\SL_2\times G$, where $G$ is the image of $\G$ in $\SL_2(\Zhat)$. This relative completion $\cG_\G$ of $\G$ will have compatible Betti, de~Rham and $\ell$-adic \'etale realizations. The inclusion $\G \hookrightarrow \SL_2(\Z)$ induces injective homomorphisms $\cG_\G^\w \to \cG^\w$ on all realizations $\w \in \{B,\DR,\etl\}$. In all cases we take the base point to be a lift of $\partial/\partial q$.

\item One can also replace $\SL_2(\Z)$ by $\PSL_2(\Z)$. This has the effect of ignoring all modular forms of odd weight.

\item Another useful construction is to complete with respect to the inclusion $\SL_2(\Z) \to \SL_2(\Zhat)$. In this the completion is an extension
$$
1 \to \U \to \cP \to \SL_2(\Zhat) \to 1
$$ 
where $\U$ is the inverse limit of the unipotent fundamental groups $\pi_1^\un(Y_\G^\an,\ast)$ of all modular curves. The coordinate ring of $\U$ consists of all closed iterated integrals of (not necessarily holomorphic) modular forms of weight 2.

\item The previous construction can be restricted to a congruence subgroup and also its quotient by $\pm \id$. One case which may be useful is to take $\G$ to be the level 2 subgroup $\G(2)/(\pm \id)$ of $\PSL_2(\Z)$. This is the fundamental group of $\Pminus$. Since all multizeta values occur as periods of $\Pminus$, the Hecke action in this case may shed light on the connection between MZVs and periods of modular forms.

\end{enumerate}

\part{Class functions from relative completions}
\label{part:class_fns}

In this part, we show that the Hecke correspondences act on the ring of class functions $\cG^\w \to \kk$ of each realization $\w \in \{B,\DR,\etl\}$ of the relative completion $\cG$ of $\SL_2(\Z)$ defined in the previous part. We also show that the space of such class functions is large and give a description of its elements.

To show that the Hecke correspondences $T_N$ act on conjugation-invariant iterated integrals of modular forms, we need to understand how iterated integrals behave under pullback and pushforward along finite coverings. The pushforward of a function $F : \lambda(Y) \to \kk$ under a finite unramified covering $\pi : Y \to X$ is the function $\pi_\ast F : \lambda(X) \to \kk$ defined by
$$
\langle \pi_\ast F, \alpha \rangle := \langle F, \pi^\ast \alpha \rangle,\quad \alpha \in \lambda(X).
$$

It is clear that pullbacks of iterated integrals along $\pi$ that are conjugation-invariant are again conjugation invariant iterated integrals. However, it is not so clear that pushforwards of such iterated integrals along finite unramified coverings are again iterated integrals, possibly twisted by characters of a finite quotient of the fundamental group of the base manifold $X$.

To illustrate the general setup considered in this section, we first consider a simple example which illustrates the shape of the general pushforward formula and why we are forced to consider the large relative completions of the type defined in the previous part if we want Hecke correspondences to act on their class functions.

In the example, the covering $\pi$ is the covering $\C^\ast \to \C^\ast$ defined by $z=\pi(w) = w^m$. Its automorphism group is canonically isomorphic to $\bmu_m$, the group of $m$th roots of unity in $\C$, which we consider to be the affine group whose set of $\Q$-points is $\bmu_m$. Denote the characteristic function of $\zeta\in \bmu_m$ by $\unit_\zeta \in \cO(\bmu_m)$ and the order of $\zeta$ by $d_\zeta$. Let $\rho : \pi_1(\C^\ast,1) \to \bmu_m$ be the monodromy homomorphism.

\begin{proposition}
\label{prop:example}
For all positive integers $r$ we have
$$
\pi_\ast \int \overbrace{\frac{dw}{w} \cdots \frac{dw}{w}}^r = \bigg(\sum_{\zeta \in \bmu_m} \bigg(\frac{m}{d_\zeta}\bigg)^{r-1} \unit_\zeta\circ \rho \bigg)\int \overbrace{\frac{dz}{z} \cdots \frac{dz}{z}}^r.
$$
\end{proposition}

Note that the right-hand side is an iterated integral on the base copy of $\C^\ast$ with coefficients in the coordinate ring $\cO(\bmu_m)$ of the Galois group of the covering. This is a special case of the pushforward formula given in Proposition~\ref{prop:pushforward}.

\begin{proof}
Recall the identity
$$
\int_\gamma \overbrace{\w \cdots \w}^r = \frac{1}{r!}\bigg(\int_\gamma \w\bigg)^r
$$
which holds for all 1-forms $\w$ on all manifolds $M$ and all paths $\gamma$ in $M$. We will prove $r!$ times the identity in the statement. Denote the positive generator of $\pi_1(\C^\ast,1)$ of the target copy of $\C^\ast$ by $\sigma$ and the positive generator of the fundamental group of the domain of $\pi$ by $\mu$. Suppose that $k\in \Z$. Set $g=\gcd(m,k)$. Note that the order $d_{\rho(\sigma^k)}$ of $\rho(\sigma^k)$ in $\bmu_m$ is $m/g$. Appealing to Corollary~\ref{cor:pullback}, we have
\begin{align*}
\Big\langle \pi_\ast \bigg(\int \frac{dw}{w} \bigg)^r,\sigma^k \Big\rangle
&= \Big\langle \bigg(\int \frac{dw}{w} \bigg)^r,\pi^\ast(\sigma^k) \Big\rangle \cr
&= \Big\langle \bigg(\int \frac{dw}{w} \bigg)^r,g \mu^{k/g} \Big\rangle \cr
&= g\bigg(\frac{k}{g}\bigg)^r \bigg(\int_\mu \frac{dw}{w} \bigg)^r \cr
&= \frac{1}{g^{r-1}}\bigg(\int_{\sigma^k} \frac{dz}{z} \bigg)^r \cr
&= \bigg(\frac{m}{d_{\rho(\sigma^k)}}\bigg)^{r-1}\Big\langle\bigg(\int \frac{dz}{z} \bigg)^r,\sigma^k \Big\rangle.
\end{align*}
This establishes the identity as $r!$ times the value of the right-hand side of the identity on $\sigma^k$ is the last expression in the calculation.
\end{proof}

\section{Class functions}
\label{sec:class_fns}

We extend the definition of the class functions of a discrete or profinite group given in Section~\ref{sec:dual_action} to affine groups in the obvious way. Suppose that $\kk$ is a field and that $G$ is an affine $\kk$-group with coordinate ring $\cO(G)$. The ring $\Cl(G)$ of {\em class functions} on $G$ is defined to be the ring of conjugation-invariant elements of $\cO(G)$. For a vector space $V$ over $\kk$, we define
$$
\Cl_V(\cG) = \Cl(\cG)\otimes_\kk V.
$$
This is the space of $V$-valued class functions.

By the ring of class functions $\Cl(\cG)$ on the relative completion $\cG$ of $\SL_2(\Z)$ constructed in Section~\ref{sec:rel_comp_mod_gp}, we will mean the collection of the class functions on each of its realizations --- Betti, de~Rham and $\ell$-adic. Set
$$
\Cl(\cG^B) = \Cl_\Q(\cG_B),\ \Cl(\cG^\DR) = \Cl_\KDR(\cG_\DR),\ \Cl(\cG^\et_\ell) =  \Cl_\Ql(\cG^\et_\ell).
$$
These are $\kk$ algebras with $\kk =\Q$, $\KDR$ and $\Ql$, respectively. There are comparison isomorphisms
$$
\Cl(\cG^B)\otimes_\Q \C \cong \Cl(\cG^\DR)\otimes_\KDR \C \text{ and }
\Cl(\cG^B) \otimes_\Q \Ql \cong \Cl(\cG^\et_\ell)
$$
for each prime number $\ell$. All are ring isomorphisms.

\begin{theorem}
The ring of class functions $\Cl(\cG^B)$ has a natural ind $\Q$ mixed Hodge structure. For each prime $\ell$, there is a natural action of $\Gal(\Qbar/\Q)$ on $\Cl(\cG^\et_\ell)$. These structures do not depend on the choice of a base point of $\M_{1,1}$. The Hodge and Galois structures are compatible with the ring structure.
\end{theorem}

\begin{proof}
This follows from Corollary~\ref{cor:gal_action}, Theorem~\ref{thm:mhs} and the theorem of the fixed part.
\end{proof}

\section{Pushforward of class functions}
\label{sec:pushforward}

Before we can show that Hecke correspondences act on the various incarnations of $\Cl(\cG)$, we need to show that an open inclusion $\cH \hookrightarrow \cG$ of a finite index subgroup of an affine group induces a well defined pushforward map $\pi_\ast : \Cl(\cH) \to \Cl(\cG)$. We will assume familiarity with the setup and notation introduced in Section~\ref{sec:cov_sp} and also with the discussion of the pullback map from Sections~\ref{sec:alg_pullback} and \ref{sec:gp_pullback}.

In this section, $\cG$ will be an affine $\kk$-group, where $\kk$ is a field of characteristic zero. Suppose that $G$ is a finite group, regarded as an affine $\kk$-group in the standard way, and that $\rho : \cG \to G$ is a surjective (i.e., faithfully flat) homomorphism. Let $H$ be a subgroup of $G$. Denote the inverse image of $H$ in $\cG$ by $\cH$ and the inverse image of $Hg \in H\bs G$ by $g^{-1}\cH g$. This notation is justified by the fact that, for all extension fields $\KK$ of $\kk$ for which $\cG(\KK) \to G$ is surjective, we have $(g^{-1}\cH g)(\KK) = \gtilde^{-1} \cH(\KK) \gtilde$, where $\gtilde \in \cG(\KK)$ is any lift of $g\in G$ to $\cG(\KK)$.

Define the ``conjugation map''
$$
\Phi_g : \Cl(\cH)\to \Cl(g^{-1} \cH g)
$$
to be the $\kk$-algebra homomorphism defined by $\Phi_g : \varphi \to \{x \mapsto \varphi(\gtilde x \gtilde^{-1})\}$.
It depends only on $g\in H\bs G$. For each $a\in G$ and each $g \in G$, we set
$$
d_a(g) = \min\{k \in \N^+ : g a^k g^{-1} \in H\}.
$$
Equivalently, it is the length of the orbit of $Hg$ under the right action of the subgroup $\langle a \rangle$ of $G$ on $H\bs G$.

Recall from Section~\ref{sec:gp_pullback} that, for all extensions $\KK$ of $\kk$, there is a pullback map
$$
\pi^\ast : \Z\blambda(\cH(\KK)) \to \Z\blambda(\cG(\KK)).
$$
Recall that $\psi^m$ denotes the $m$th Adams operator. The following result generalizes Proposition~\ref{prop:example}.

\begin{proposition}[pushforward formula]
\label{prop:pushforward}
There is a well defined pushforward map $\pi_\ast : \Cl(\cH) \to \Cl(\cG)$. It is a morphism of affine $\kk$-schemes and is characterized by the equality
$$
\langle \pi_\ast F, \alpha \rangle = \langle F, \pi^\ast \alpha \rangle
$$
for all $F\in \Cl(\cH)$ and $\alpha \in \cG(\kkbar)$. It is given by the formula
\begin{equation}
\label{eqn:alg_pushforward}
\pi_\ast F = \sum_{a\in G}\bigg(\sum_{g\in G} \frac{1}{d_a(g)} \psi^{d_a(g)}\circ \Phi_gF \bigg) \rho^\ast\unit_a \in \Cl(\cG),
\end{equation}
where $\unit_a \in \cO(G)$ denotes the characteristic function of $a\in G$.
\end{proposition}

\begin{proof}
Since $\cG(\kkbar) \to G$ is surjective, we can choose a lift $\gtilde \in \cG(\kkbar)$ of each $g\in G$. Suppose that $F\in \Cl(\cH)$ and that $\alpha \in \cG(\kkbar)$. The pullback formula (\ref{eqn:pullback}) in Section~\ref{sec:alg_pullback} can be written
$$
\pi^\ast \alpha = \sum_{g\in G} \frac{1}{d_{\rho(\alpha)}(g)}\, \gtilde\alpha^{d_{\rho(\alpha)}(g)} \gtilde^{-1}.
$$
Consequently,
$$
\langle \pi_\ast F,\alpha \rangle
= \sum_{g \in G} \frac{1}{d_{\rho(\alpha)}(g)}\Big\langle F,  \gtilde\alpha^{d_{\rho(\alpha)}(g)} \gtilde^{-1}\Big\rangle.
$$
On the other hand, since $\alpha^{d_a(g)} \in \gtilde^{-1} \pi_1(Y,y_o) \gtilde$, we have
\begin{align*}
\phantom{=}\bigg\langle \sum_{a\in G}\bigg(\sum_{g\in G} \frac{1}{d_a(g)} \psi^{d_a(g)}\circ \Phi_gF \bigg) \rho^\ast\unit_a,\alpha \bigg \rangle
&= \sum_{g\in G} \frac{1}{d_{\rho(\alpha)}(g)} \Big\langle\psi^{d_{\rho(\alpha)}(g)}\circ \Phi_gF,\alpha \Big\rangle
\cr
&= \sum_{g\in G} \frac{1}{d_{\rho(\alpha)}(g)} \Big\langle \Phi_gF,\alpha^{d_{\rho(\alpha)}(g)} \Big\rangle
\cr
&=\sum_{g\in G} \frac{1}{d_{\rho(\alpha)}(g)} \Big\langle F,\gtilde\alpha^{d_{\rho(\alpha)}(g)}\gtilde^{-1} \Big\rangle.
\end{align*}
Since $\cG$ is the disjoint union of the $g^{-1}\cH g$,
$$
\Cl(\cG) \subseteq \bigoplus _{g\in G} \Cl(g^{-1}\cH g) \subseteq \cO(\cG).
$$
The right hand side of (\ref{eqn:alg_pushforward}) is an element of $\bigoplus _{g\in G} \Cl(g^{-1}\cH g)$. Since it is conjugation invariant, it lies in $\Cl(\cG)$.
\end{proof}

We now specialize to the case where $\cG$ is one of the incarnations $\cG^\w$, $\w \in \{B,\DR,\etl\}$ of the relative completion of $\SL_2(\Z)$, $G=\PSL_2(\Fp)$, and $\rho$ is the homomorphism induced by the quotient map $\SL_2(\Z) \to \PSL_2(\Fp)$. The group $H$ is the Borel subgroup of $\PSL_2(\F_2)$ consisting of upper triangular matrices. The preimage $\cH^\w$ of $H$ in $\cG^\w$ is the corresponding relative completion of $\G_0(p)$, which is discussed in Section~\ref{sec:variants}(\ref{item:cong_subgps}).

\begin{corollary}
For each realization $\w \in \{B,\DR,\etl\}$, the inclusion $\pi : \G_0(p) \to \SL_2(\Z)$ induces a pushforward mapping
$$
\pi_\ast : \Cl(\cH^\w) \to \Cl(\cG^\w).
$$
These correspond under the comparison isomorphisms. \qed
\end{corollary}

Since the modular curve $Y_0(p)$ is defined over $\Q$, there is an action of $\Gal(\Qbar/\Q)$ on $\Cl(\cH^\et_\ell)$ for each prime number $\ell$.

\begin{proposition}
\label{prop:pushforward_et}
For each prime number $\ell$, the pushforward mapping
$$
\pi_\ast : \Cl(\cH^\et_\ell) \to \Cl(\cG^\et_\ell)
$$
is $\Gal(\Qbar/\Q)$-equivariant.
\end{proposition}

\begin{proof}
Consider the commutative diagram
$$
\xymatrix{
\Cl(\cH^\et_\ell) \ar[r]^{\pi_\ast} \ar[d] & \Cl(\cG^\et_\ell) \ar[d] \cr
\Hom(\Ql\blambda(\G_0(p)^\wedge),\Ql) \ar[r] & \Hom(\Ql\blambda(\SL_2(\Z)^\wedge),\Ql)
}
$$
in which the vertical maps are evaluation and the bottom arrow is induced by $\pi^\ast : \Z\blambda(\SL_2(\Z)^\wedge) \to \Z\blambda(\G_0(p)^\wedge)$. The bottom map is Galois equivariant by Theorem~\ref{thm:gal_equiv}. Since the two vertical maps are Galois equivariant injections, the top map $\pi^\ast$ is also Galois equivariant.
\end{proof}

\begin{lemma}
\label{lem:morphism}
For each $g\in \PSL_2(\Fp)$ the conjugation mapping
$$
\Phi_g : \Cl(g^{-1} \cH^B g) \to \Cl(\cH^B).
$$
is a morphism of MHS.
\end{lemma}

\begin{proof}
This is not immediately clear as the automorphism of $\cO(\cG)$ induced by conjugation by an element of $\cG^B(\Q)$ is almost never a morphism of MHS. To prove the result, we use the fact that $\cO(\PSL_2(\Fp))$ is a Hodge structure of type $(0,0)$. Fix a natural splitting of the Hodge and weight filtrations of the complexification of all MHSs (such as Deligne's bigrading \cite[Lem.~1.2.11]{deligne:h2}). Since this is compatible with tensor products, we have the Hopf algebra splitting
$$
\cO(\cG^B)\otimes\C \stackrel{\simeq}{\To} \bigoplus_{m\ge 0} \Gr^W_m \cO(\cG^B)\otimes\C 
\stackrel{\text{proj}}{\To} \Gr^W_0 \cO(\cG^B)\otimes\C \to \cO(\PSL_2(\Fp))\otimes\C
$$
of $\rho^\ast : \cO(\PSL_2(\Fp)) \to \cO(\cG^B)$. It is strict with respect to both the Hodge and weight filtrations. This implies that we can lift $g \in \PSL_2(\Fp)$ to $\gtilde \in F^0 W_0 \cG^B(\C)$. The automorphism of $\cO(\cG^B)\otimes\C$ induced by conjugation by $\gtilde$ thus preserves both the Hodge and weight filtrations, which implies that $\Phi_g : \Cl(\cH^B) \to \Cl(\cH^B)$ also preserves them. We can also choose a lift $\gamma \in \cG^B(\Q)$ of $g$. Conjugation by $\gamma$ also induces $\Phi_g$, which implies that it is defined over $\Q$. Since it also preserves the Hodge and weight filtrations, it is a morphism of MHS.
\end{proof}

Combining this with Proposition~\ref{lem:morphism}, we obtain:

\begin{proposition}
\label{prop:pushforward_MHS}
The pushforward map
$$
\pi_\ast : \Cl(\cH^B) \to \Cl(\cG^B)
$$
induced by $\G_0(p) \to \SL_2(\Z)$ is a morphism of MHS.
\end{proposition}

\section{Hecke action on $\Cl(\cG)$}
\label{sec:hecke_on_Cl}

Armed with the results of the previous section, we can now prove that the dual Hecke operators act compatibly on all incarnations (Betti, de~Rham and $\ell$-adic) of $\Cl(\cG)$ and that this action respects the mixed Hodge structure on the Betti incarnation and commutes with the Galois action on each $\ell$-adic incarnation. Recall that the dual Hecke algebra $\That^\op$ is the opposite ring of $\That$, which was defined in the introduction.

\begin{theorem}
\label{thm:hecke}
The dual Hecke algebra $\That^\op$ acts $\Cl(\cG)$. More precisely, it acts on $\Cl(\cG^B)$, $\Cl(\cG^\DR)$ and on each $\Cl(\cG^\et_\ell)$. These actions correspond under the comparison isomorphisms. Each element of $\That^\op$ acts on $\Cl(\cG^B)$ as a morphism of ind-MHS and on $\Cl(\cG^\et_\ell)$ as a $\Gal(\Qbar/\Q)$-equivariant endomorphism.
\end{theorem}

\begin{proof}
Proposition~\ref{prop:pushforward} implies that each $\Tdual_N$ and $\edual_p$ act compatibly on all incarnations of $\Cl(\cG)$. To prove that their actions on $\Cl(\cG^B)$ are morphisms of MHS and commute with the Galois action on each $\Cl(\cG^\et_\ell)$, we use the fact that $\That$ is generated by the $\edual_p$ and the $\Tdual_p$. The Galois equivariance of their actions on $\Cl(\cG^\et_\ell)$ follows from Proposition~\ref{prop:pushforward_et}. That they act as morphisms of MHS on $\Cl(\cG^B)$ follows from Proposition~\ref{prop:pushforward_MHS}.
\end{proof}

\section{Constructing class functions}
\label{sec:conj_invar_ints}

It is not immediately clear that the ring $\Cl(\cG)$ of a general affine group is large or interesting. In this section we give a general description of all elements of $\Cl(\cG)$ when $\cG$ is an affine group scheme whose prounipotent radical is free. This description will apply to the relative completions $\cG$ of $\SL_2(\Z)$ constructed in Section~\ref{sec:rel_comp_mod_gp}.

We first work in the following abstract setting: Assume that $\kk$ is a field of characteristic zero and that $\cG$ is an affine $\kk$-group that is an extension
\begin{equation}
\label{eqn:extn}
1 \to \U \to \cG \to R \to 1.
\end{equation}
of a proreductive group $R$ by a prounipotent group $\U$. In addition, we assume that the Lie algebra $\u$ of $\U$ is free as a pronilpotent Lie algebra.

\subsection{A reduction}

It is convenient to first reduce to the case where $R$ is an algebraic group and $H_1(\U)$ is finite dimensional. As above, $\u$ is assumed to be free.

\begin{proposition}
Every extension (\ref{eqn:extn}) is the inverse limit of extensions
$$
1 \to \U_\alpha \to \cG_\alpha \to R_\alpha \to 1
$$
where $R_\alpha$ is reductive, the Lie algebra $\u_\alpha$ of $\U_\alpha$ is free on a finite dimensional representation $V_\alpha$ of $R_\alpha$ and
$$
\u_\alpha \cong \LL(V_\alpha)^\wedge
$$
as a pronilpotent Lie algebra in the category of pro $R$-modules.
\end{proposition}

The significance of this result for us is that
$$
\cO(\cG) = \varinjlim \cO(\cG_\alpha) \text{ and } \Cl(\cG) = \varinjlim \Cl(\cG_\alpha).
$$

\begin{proof}[Sketch of proof]
A general version of Levi's theorem implies that the extension (\ref{eqn:extn}) is split. This can be proved using the fact that $\cG$ is the inverse limit of affine algebraic groups, to which one can apply the usual version of Levi's theorem, and use Zorn's Lemma to prove that there is a maximal quotient of $\cG$ on which there is a splitting. One then shows, as usual, that the maximal quotient has to be $\cG$ itself. The choice of a splitting makes the Lie algebra $\u$ of $\U$ into a pronilpotent Lie algebra in the category of pro $R$-modules. The choice of a continuous $R$-invariant splitting of the abelianization map
$$
\u \to H_1(\u)
$$
induces a Lie algebra homomorphism $\LL(H_1(\u)) \to \u$, which is continuous if we give $\LL(\H_1(\u))$ the natural topology. Since $\u$ is complete, this homomorphism extends to a continuous homomorphism
$$
\LL(H_1(\u))^\wedge \to \u
$$
which is surjective as both Lie algebras are pronilpotent and the homomorphism induces an isomorphism on abelianizations. It is injective as $\u$ is free. We conclude that there are isomorphisms
$$
\cG \cong \U \rtimes R \cong  \exp \LL(H_1(\U))^\wedge \rtimes R .
$$
To complete the proof, write $H_1(\u) = \varprojlim V_\alpha$, where each $V_\alpha$ is a finite dimensional $R$-module. We can write $R$ as the inverse limit of reductive affine algebraic groups $R_\alpha$, where $R_\alpha$ acts on $V_\alpha$. Then
$$
\cG \cong \varprojlim_\alpha \U_\alpha \rtimes R_\alpha,
$$
where $\U_\alpha$ is the quotient of $\U$ whose Lie algebra is $\LL(V_\alpha)^\wedge$.
\end{proof}

\subsection{A special case}

The results of the previous section reduce the problem of understanding $\Cl(\cG)$ to the case where $R$ is a reductive (and thus algebraic) group and where the abelianization $H_1(\u)$ of $\u$ is finite dimensional. We examine this case in this section. For convenience, we denote $H_1(\u)$ by $V$.

Every such extension is split and the splitting is unique up to conjugation by an element of $\U(\kk)$. Fix a splitting. It determines a left action of $R$ on $\u$ and an isomorphism
\begin{equation}
\label{eqn:splitting}
\cG \cong \U \rtimes R.
\end{equation}
Elements of $\cG(K)$, where $K$ is an extension of $\kk$, will be identified with pairs $(u,r) \in \U(K)\times R(K)$ with multiplication
$$
(u,r) (u',r') = \big(u (r\cdot u'),rr'\big)
$$
where $r : u \to r\cdot u$ denotes the left conjugation action of $R$ on $\U$. Since the map $\U\times R \to \cG$ defined by $(u,r) \mapsto ur$ is an isomorphism of affine schemes, we will sometimes denote $(u,r)$ by $ur$.

The choice of an $R$-invariant splitting $s: V \to \u$ of $\u\to H_1(\u)$ induces a continuous $R$-invariant homomorphism $\LL(V)^\wedge \to \u$ which induces an isomorphism on $H_1$ and is thus surjective. Since $\u$ is free, it is an isomorphism as both Lie algebras are pronilpotent.

\subsubsection{The coordinate ring of $\cG$}
A basic reference for this section is Appendix~A of Quillen's paper \cite{quillen}. Denote the tensor algebra on $V$ by $T(V)$ and its degree completion by $T(V)^\wedge$. It is a topological algebra that is complete in the topology defined by the powers of the closed ideal $I$ generated by $V$. It has additional structure; it is a complete Hopf algebra with diagonal
$$
\Delta : T(V)^\wedge \to T(V)^\wedge \comptensor T(V)^\wedge
$$
defined by $\Delta v = 1 \otimes v + v\otimes 1$ for all $v\in V$. Its space of primitive elements is $\LL(V)^\wedge$ and $\U$ is isomorphic to the set of group-like elements of $T(V)^\wedge$.

Since $\LL(V)^\wedge$ is a topological Lie algebra, its enveloping algebra is a topological Hopf algebra. The following result summarizes several well-known facts.

\begin{proposition}[{\cite[Appendix~A]{quillen}}]
The inclusion $\LL(V)^\wedge \to T(V)^\wedge$ induces a complete Hopf algebra isomorphism of the completed universal enveloping algebra of $\LL(V)^\wedge$ with $T(V)^\wedge$. Consequently, there is a complete coalgebra isomorphism
$$
T(V)^\wedge \cong \prod_{m\ge 0} \Sym^m \LL(V)^\wedge.
$$
Moreover, if we identify $\u$ with $\LL(V)^\wedge$, then the exponential mapping
$$
\exp :  I \to 1 + I
$$
restricts to an isomorphism of affine schemes $\u \to \U$.
\end{proposition}

The last statement of the proposition implies that the coordinate ring of $\U$ is the ring of continuous polynomials on $\LL(V)^\wedge$. The previous result implies that this is just the continuous dual of $T(V)^\wedge$, which is the graded dual
$$
\Hom^\cts_\kk(T(V)^\wedge,\kk) \cong T(\Vdual) = \bigoplus_{m\ge 0} \Vdual^{\otimes m}
$$
of $T(V)$, where $\Vdual$ denotes the dual of $V$. Multiplication is given by the shuffle product
$$
\shuffle : \Vdual^{\otimes a} \otimes \Vdual^{\otimes b} \to \Vdual^{\otimes(a+b)}
$$
which is defined by
\begin{equation}
\label{eqn:shuffle}
(\varphi_1 \dots \varphi_a) \shuffle (\varphi_{a+1} \dots \varphi_{a+b}) = \sum_{\sigma \in \Sh(a,b)} \varphi_{\sigma(1)} \dots \varphi_{\sigma(a+b)}
\end{equation}
where each $\varphi_j \in \Vdual$ and $\sigma$ ranges over the shuffles of type $(a,b)$.

\begin{corollary}
The chosen splitting of $\cG \to R$ induces an algebra isomorphism of the coordinate ring of $\cG$ with
$$
\cO(\U) \otimes \cO(R) \cong \bigoplus_{m\ge 0}\Vdual^{\otimes m} \otimes \cO(R).
$$
\end{corollary}

\subsubsection{Computation of $\Cl(\cG)$}

To compute $\Cl(\cG)$ first observe that
$$
\Cl(\cG) = \big[\cO(\U) \otimes \cO(R)\big]^\cG = \big[\cO(\U) \otimes \cO(R)\big]^R \cap \big[\cO(\U) \otimes \cO(R)\big]^\U
$$
where $\cG$ (and thus $R$ and $\U$ by restriction) acts on $\cO(\cG)$ on the left by conjugation:
$$
(gF)(h) = F(g^{-1}h g)\quad g,h \in \cG,\ F \in \cO(\cG).
$$
This convention will hold throughout this section.

We regard $\cO(\U)$ as an ind-scheme over $\kk$. Fix an algebraic closure $\kkbar$ of $\kk$. Elements of $\cO(\U)\otimes\cO(R)$ correspond to morphisms $R\to \cO(\U)$. The functions $f: R \to \cO(\U)$ and $F\in \cO(\U) \otimes \cO(R)$ correspond if and only if
$$
\langle F,(u,r) \rangle = \langle f(r),u\rangle
$$
for all $u \in \U(\kkbar)$, $r\in R(\kkbar)$.

\begin{lemma}
\label{lem:U-invariance}
Let $R$ act on itself by conjugation: $r : t \mapsto rtr^{-1}$. Under the correspondence above, elements of $\big[\cO(\U) \otimes \cO(R)\big]^R$ correspond to $R$-invariant morphisms $R \to \cO(\U)$.
\end{lemma}

\begin{proof}
Suppose that $f : R \to \cO(\U)$ and $F \in \cO(\U)\otimes\cO(R)$ correspond, that $r,t\in R(\kkbar)$ and that $u \in \U(\kkbar)$. On the one hand we have
$$
\langle f(trt^{-1}),u\rangle = \langle F, utrt^{-1}\rangle = \langle F,t(t^{-1}\cdot u) r t^{-1} \rangle = \langle t^{-1} F,t^{-1}ut r\rangle
$$
and, on the other, we have
$$
\langle t \cdot f(r), u \rangle = \langle f(r), t^{-1}ut \rangle = \langle F, t^{-1}u t r \rangle.
$$
Since this holds for all $u\in \U(\kkbar)$, it follows that for all $t \in R(\kkbar)$, $F = t^{-1} F$ if and only if  $f(trt^{-1}) = t\cdot f(r)$.
\end{proof}

The ring $\big[\cO(\U) \otimes \cO(R)\big]^R$ is graded by degree in $V$:
$$
\big[\cO(\U) \otimes \cO(R)\big]^R = \bigoplus_{m\ge 0} \big[\Vdual^{\otimes m} \otimes \cO(R)\big]^R.
$$
Elements of the summand $\big[\Vdual^{\otimes m} \otimes \cO(R)\big]^R$ correspond to $R$-invariant functions $R \to \Vdual^{\otimes m}$.

The following example, due to Florian Naef, should help motivate the statement and proof of the following two results.

\begin{example}[Naef]
\label{ex:naef}
In this example we consider the problem of computing the class functions on the semi-direct product $V \rtimes R$ that are linear on $V$. Suppose that $f: R \to \Vdual$. Define $F : V \rtimes R \to \kk$ by $F(v,r) = \langle f(r),v \rangle$. Since $t(v,r)t^{-1} = (tv,trt^{-1})$, $F$ is invariant under conjugation by $t\in R(\kkbar)$ if and only if $f$ is $R$-invariant. Since
$$
F(ru,u) = \langle f(r),ru \rangle = \langle f(r^{-1}rr),ru \rangle = \langle r\cdot f(r),ru \rangle = \langle f(r),u \rangle = F(u,r)
$$
and since $(u,1) (v,r) (u,1)^{-1} = (v+(1-r)u,r)$, we have
$$
F\big((u,1) (v,r) (u,1)^{-1}\big) = F(v,r) + F(u,r) - F(ru,r) = F(v,r).
$$
That is, $F$ is invariant under conjugation by $V$. Consequently, the class function $F: V \rtimes R \to \kk$ that are linear on $V$ correspond to $R$-invariant functions $f : R \to \Vdual$.
\end{example}

The first step in generalizing Naef's example is to compute the $\U$-invariants that are in $\big[\cO(\U)\otimes\cO(R)\big]^R$. For this, it is useful to introduce the infinite cyclic group $\Sigma$ generated by the symbol $\sigma$. It acts on the set $(V^{\otimes m}\otimes_\kk K) \times R(K)$, where $K$ is an extension of $\kk$, via the formula
$$
\sigma : (v_1 v_2\dots v_m,r) \mapsto (v_2 v_3 \dots v_m (r\cdot v_1),r)
$$
and on $[\Vdual^{\otimes m}\otimes \cO(R)]^R$ via the dual action.

The following result generalizes Naef's Example~\ref{ex:naef} from $m=1$ to all $m\ge 1$.

\begin{lemma}
\label{lem:cyclic}
If $F \in [\Vdual^{\otimes m}\otimes \cO(R)]^R$, then $\sigma^m F = F$, so that the $\Sigma$-action factors through an action of its cyclic quotient $C_m := \Sigma/\langle \sigma^m \rangle$. Moreover, the degree $m$ summand of $\Cl(\cG)$ consists of the $R$-invariant functions that are also $C_m$-invariant:
$$
[\Vdual^{\otimes m}\otimes \cO(R)]^\cG = \big[[\Vdual^{\otimes m}\otimes \cO(R)]^R\big]^{C_m}.
$$
\end{lemma}

\begin{proof}
Suppose that $v_1,\dots,v_m \in V$, $r\in R(\kkbar)$ and that $F \in [\Vdual^{\otimes m}\otimes \cO(R)]^R$. Then, since $F$ is $R$-invariant,
\begin{align*}
\langle \sigma^m F,(v_1\dots,v_m,r)\rangle
&= \langle F,\sigma^m(v_1\dots,v_m,r) \rangle \cr
&= \langle F,\big((rv_1)\dots(rv_m),r\big) \rangle \cr
&= \langle F,\big(r\cdot(v_1\dots,v_m),rrr^{-1}\big) \rangle \cr
&= \langle F,(v_1\dots,v_m,r)\rangle
\end{align*}
so that $\sigma^m F = F$. Since
$$
(v_2\dots v_m, r) (v_1,1) = \big(v_2\dots v_m (r\cdot v_1), r \big),
$$
in $(T(V)^\wedge\otimes \kkbar) \rtimes R(\kkbar)$, we see that a function $F \in [\Vdual^{\otimes m}\otimes \cO(R)]^R$ is $\U$-invariant (and hence $\cG$-invariant) if and only if for all $v_1,\dots,v_m \in V$ and $r\in R(\kkbar)$ we have
$$
\langle F, (v_1\dots v_m, r) \rangle = \langle F, (v_2\dots v_m (r\cdot v_1), r) \rangle  = \langle \sigma F, (v_1\dots v_m, r) \rangle.
$$
That is, if and only if $\sigma F = F$.
\end{proof}

\begin{corollary}
The choice of a splitting (\ref{eqn:splitting}) determines a graded ring isomorphism
$$
\Cl(\cG) \cong \bigoplus_{m\ge 0} \Cl_m(\cG)
$$
where
$$
\Cl_m(\cG) := \big[[\Vdual^{\otimes m}\otimes \cO(R)]^R\big]^{C_m}.
$$
The product $\Cl_m(\cG) \otimes \Cl_n(\cG) \to \Cl_{m+n}(\cG)$ is the extension
\begin{equation}
\label{eqn:product}
\sum_{j,k}(\varphi'_j\psi'_j) \otimes (\varphi_k''\psi_k'')
\mapsto \sum_{j,k} (\varphi_j'\shuffle \varphi_k'')\psi_j'\psi_k''
\end{equation}
of the shuffle product (\ref{eqn:shuffle}), where $\varphi_j' \in \Vdual^{\otimes m}$, $\varphi_k \in \Vdual^{\otimes n}$, $\psi_j',\psi_k'' \in \cO(R)$.
\end{corollary}

While the grading of $\Cl(\cG)$ depends on the splitting of the projection $\cG \to R$, the filtration
\begin{equation}
\label{eqn:len_filt}
0 \subset \Cl_0(\cG) \subseteq C_1 \Cl(\cG) \subseteq C_2\Cl(\cG) \subseteq \cdots
\end{equation}
does not, where
$$
C_m \Cl(\cG) := \bigoplus_{k\le m} \Cl_k(\U\rtimes R).
$$
However, since the splitting $s: R \to \cG$ is unique up to conjugation by an element of $\U(\kk)$, the projection
$$
\rho : \Cl(\cG) \to \Cl(R) \quad F \mapsto \{r \mapsto F(s(r))\}
$$
induced by $s$ does not depend on $s$. This means that there is a canonical decomposition
$$
\Cl(\cG) = \Cl(R) \oplus \ker \rho \cong \Cl(R) \oplus \Cl(\cG)/C_0.
$$

\begin{remark}
\label{rem:class_fns}
Several comments are in order:
\begin{enumerate}

\item \label{item:unipt} The prounipotent case (i.e., $R$ trivial) is well known, even when $\u$ is not free. See \cite{kawazumi-kuno} for the surface group case. The general case is almost identical. The ring $\Cl(\U)$ is the continuous dual $\Hom^\cts_\kk(|U\u|,\kk)$ of the cyclic quotient
$$
|U\u| := U\u/\{\text{subspace generated by } vw - wv: v\in V, w\in U\u\}
$$
of the enveloping algebra $U\u$.

\item Not all irreducible representations of $R$ occur in $\cO(R)$ when it is viewed as an $R$-module via conjugation. Only those representations that factor through the adjoint form $R/Z(R)$ of $R$ can appear. In particular, when $R=\SL_2$, only the even symmetric powers of the defining representation occur. However, each occurs a countable number of times.

\end{enumerate}
\end{remark}

\subsection{Constructing elements of $\Cl_m(\U\rtimes R)$}
\label{sec:cyclic}

Here we give a few explicit constructions of class functions on $\U\rtimes R$. We continue with the notation and setup of the previous section. We will also denote $\cO(R)$, regarded as a left $R$-module via right conjugation, by $\cO(R)^\conj$.

\subsubsection{Computing $\Cl_0(\U\rtimes R)$}
\label{sec:Cl_0}

This is just $\Cl(R)$. Denote the representation ring of $R$ by $\Rep(R)$. Proposition~\ref{prop:coord_ring} implies that the function
$$
\tr : \Rep(R) \to \Cl(R)
$$
that takes the isomorphism class of the $R$-module $A$ to the class function $R(\kkbar) \to \kkbar$ defined by $r\mapsto \tr\rho_A(r)$ is an isomorphism, where $\rho_A : R \to \Aut V$ is the corresponding homomorphism.

When $G$ is finite, $\Cl(G)$ is spanned by the irreducible characters of $G$. And when $G$ is the projective limit of $G_\alpha$, we have
$$
\Cl(G) = \varinjlim_\alpha \Cl(G_\alpha).
$$

\subsubsection{Computing $\Cl_1(\U\rtimes R)$}
\label{sec:length1}
Elements of $\Cl_1(\cG)$ correspond to $R$-module homomorphisms $\varphi: V \to \cO(R)^\conj$. The corresponding class function $F$ is defined by
$$
F(u,r) = \langle \varphi(\ubar), r \rangle
$$
where $\ubar$ denotes the image of $u$ under the projection $\U \hookrightarrow T(V)^\wedge \to V$, where the second map is projection.

The space $\Hom_R(V,\cO(\cG)^\conj)$ of such maps $\varphi$ can be computed by writing $V$ as a sum of its irreducible components and applying Proposition~\ref{prop:coord_ring}.

\begin{remark}
\label{rem:cocycle}
A function $\varphi : V \to \cO(R)^\conj$ extends to the function
$$
\xymatrix{
\U \rtimes R \ar[r] & V \rtimes R \ar[r] & V \ar[r] & \cO(R)^\conj.
}
$$
where the first map is the projection $\U \to V$ in the second factor. Denote it by $\phitilde$. The $R$ invariance of $\varphi$ is equivalent to $\phitilde$ being a 1-cocycle on $\U\rtimes R$ with values in $\cO(R)^\conj$ as
\begin{multline*}
\phitilde\big((v,r) (u,s)\big) = \phitilde\big((v(ru),rs)\big) = F(v + ru) \cr = \varphi(v) + rF(u) = \phitilde(v,r) + r\phitilde(u,s).
\end{multline*}
The class function corresponding to a coboundary is zero.

The pullback of a 1-cocycle along a homomorphism $\rho : \G \to (\U\rtimes R)(\kk)$ will define a class function on $\G$. Conversely, each 1-cocycle $\varphi : \G \to \cO(R)^\conj$ gives rise to the class function $\G \to \kk$ defined by
$$
\gamma \mapsto \langle \varphi(\gamma),\gamma \rangle.
$$
Coboundaries give the trivial class function, so there is a well defined function
$$
H^1(\G,\cO(R)^\conj) \to \Cl_\kk(\G)
$$
This observation explains how and why modular forms give class functions on $\SL_2(\Z)$. See Section~\ref{sec:even_wt}.
\end{remark}

\subsubsection{Constructing elements of $\Cl_m(\U\rtimes R)$ by averaging}

Suppose that $m>1$ and that $F\in [\Vdual^{\otimes m}\otimes \cO(R)^\conj]$. Define $\rot F \in [\Vdual^{\otimes m}\otimes \cO(R)]$ by
$$
\langle \rot F,(v_1\dots v_m,r)\rangle = \langle F,(v_2 v_3 \dots v_m v_1,r)\rangle
$$
and define $\nabla_j F \in [\Vdual^{\otimes m}\otimes \cO(R)]$ for each $j\in \{1,\dots,n\}$ by
$$
\langle \nabla_j F, (v_1\dots v_m,r)\rangle = \langle F, (v_1\dots (rv_j) \dots v_m,r)\rangle.
$$
In both cases $r\in R(\kkbar)$ and each $v_k \in V$.

\begin{proposition}
If $F$ is $R$-invariant, then so are $\rot F$ and each $\nabla_j F$.
\end{proposition}

\begin{proof}
It is clear that $F$ invariant implies that $\rot F$ is invariant. The other assertions follow from that the fact that for all $t,r\in R(\kkbar)$ we have
\begin{align*}
\langle\nabla_j F,((tv_1)(tv_2) \dots (t v_m),trt^{-1})\rangle
&= \langle F, ((tv_1) \dots (trv_j)\dots (t v_m),trt^{-1})\rangle
\cr
&= \langle F, (v_1 \dots (rv_j) \dots v_m,r)\rangle
\cr
&=  \langle \nabla_j F, (v_1 v_2 \dots v_m,r)\rangle.
\end{align*}
\end{proof}

The action of the cyclic group $C_m = \langle \sigma \rangle$ on an $R$-invariant function $F\in [\Vdual^{\otimes m}\otimes \cO(R)^\conj]^R$ can be expressed in terms of these:
$$
\sigma^\ast F = \rot \circ \nabla_1 F
$$
so that the average $\overline{F}$ of $F$ over the $C_m$-action is
$$
\overline{F} = \left(1 + \rot \circ \nabla_1 + \rot^2 \circ (\nabla_1 \nabla_2) + \dots + \rot^{m-1} \circ (\nabla_1 \dots \nabla_{m-1})\right) F.
$$
It is $C_m$-invariant and therefore an element of $\Cl_m(\U\rtimes R)$.

It is not always clear when $\overline F$ is non-zero. The following somewhat technical result gives a criterion for the cyclic average to be non-zero. It will be used in Section~
\ref{sec:odd_wt} to construct class functions from modular forms of odd weight.

\begin{proposition}
\label{prop:cyclic_nonzero}
Suppose that, $V$ and $V_1,\dots,V_n$ are $R$-modules. Suppose that $V$ contains $V_1 \oplus \dots \oplus V_m$. Let $q_j : \Vdual \to \Vdual_j$ be the projection dual to the inclusion $V_j \to V$. Suppose that
$$
\varphi : \Vdual_1 \otimes \dots \otimes \Vdual_m \to \cO(R)^\conj
$$
is $R$-invariant.  If $F$ is the composite
$$
\Vdual^{\otimes m} \overset{q}{\To} \Vdual_1 \otimes \dots \otimes \Vdual_m \overset{\varphi}{\To} \cO(R)
$$
where $q = q_1 \otimes \dots \otimes q_m$, then $F \neq 0$ implies that its cyclic average $\overline{F}$ is also non-zero.
\end{proposition}

\begin{proof}
Since $F\neq 0$, we have $v_j \in V_j$ such that $F(v_1 \dots v_m) \neq 0$. The definition of $F$ implies that for all non-identity permutations $\sigma$ of $\{1,\dots,m\}$,
$$
F(v_{\sigma(1)}\dots v_{\sigma(m)}) = 0.
$$
The definition of $\overline{F}$ implies that $\overline{F} (v_1 \dots v_m) = F(v_1 \dots v_m) \neq 0$.
\end{proof}

\begin{example}
\label{ex:cyclic_product}
Suppose that $F_1,\dots,F_m \in \Cl_1(\U\rtimes R)$. Define
$$
F \in [\Vdual^{\otimes m}\otimes \cO(R)^\conj]^R
$$
by
$$
\langle F, (v_1\dots v_m,r)\rangle = \prod_{j=1}^m \langle F_j,(v_j,r)\rangle 
$$
This function is $R$-invariant. Its cyclic average is the $R$-invariant function
$$
F_1 \cyl F_2 \cyl \cdots \cyl F_m \in \Cl_m(\U\rtimes R)
$$
defined by
$$
\langle F_1 \cyl F_2 \cyl \cdots \cyl F_m,(v_1\dots v_m,r) \rangle = \sum_{\tau\in C_m} \prod_{j=1}^m \langle F_{\tau(j)},(v_j,r)\rangle.
$$
We will call this the {\em cyclic average} of $F_1,\dots,F_m$. If $V$ contains $V_1 \oplus \dots \oplus V_m$ and if $F_j \in \Vdual_j$, then $F_1 \cyl F_2 \cyl \cdots \cyl F_m$ will be non-zero provided that $F\neq 0$.
\end{example}

\begin{remark}
\label{rem:prod}
When $m=2$, $F_1\cyl F_2$ is simply the extended shuffle product (\ref{eqn:product}) of $F_1$ and $F_2$. When $m>2$ it appears that, in general, cyclic products of linearly independent elements $F_j$ of $\Cl_1(\U\rtimes R)$ will not be decomposable. That is, they will not be expressible as a sum of products of elements of the $\Cl_j((\U\rtimes R))$ with $j<m$.
\end{remark}

\section{Examples of conjugation-invariant iterated integrals}
\label{sec:examples}

Our goal now is to give several constructions of interesting elements of $\Cl(\cG)$, where $\cG$ is the relative completion of $\SL_2(\Z)$ defined in Section~\ref{sec:rel_comp_mod_gp}. These class functions will be constructed from holomorphic modular forms. We begin with a few remarks about $\Cl_0(\cG)$ as $\Cl(\cG)$ is a module over it. The first step is the observation:

Recall from Section~\ref{sec:comp_H} that $H^B$ and $H^\DR$ denote the Betti and de~Rham realizations of $H^1(E_{\partial/\partial q})$. They are $\Q$ vector spaces related by the comparison isomorphism (\ref{eqn:comp_H}). It induces an isomorphism
$$
\SL(H^B)\times_\Q\C \cong \SL(H^\DR)\times_\Q \C.
$$
We will denote $\SL(H^\w)$ by $\SL_2^\w$, where $\w \in \{B,\DR,\etl\}$. Note that $\tr : \SL_2^B \to \Q$ and $\tr : \SL_2^\DR \to \Q$ correspond under the comparison isomorphism.

\begin{lemma}
For $\w\in \{B,\DR\}$, we have $\Cl_0(\cG^\w) \cong \Cl(\SL_2(\Zhat))[\tr]$, where $\tr : \SL_2 \to \Q$ is the trace.
\end{lemma}

\begin{proof}
The discussion in Section~\ref{sec:Cl_0} implies that $\Cl(\SL_2/\kk) = \kk[\tr]$, where $\tr$ is the trace function $\SL_2 \to \kk$. So
$$
\Cl_0(\cG^\w) \cong \Cl(\SL_2^\w \times \SL_2(\Zhat)) \cong \Cl(\SL_2^\w)\otimes \Cl(\SL_2(\Zhat)) \cong \Cl(\SL_2(\Zhat))[\tr].
$$
\end{proof}

After tensoring with $\Qbar$, the ring $\Cl(\SL_2(\Zhat))$ is spanned by the irreducible characters of $\SL_2(\Zhat)$. Note that this ring has many zero divisors as each $\SL_2(\Z/N)$ is totally disconnected.

\subsection{Modular forms as cohomology classes}
\label{sec:coho_mod-forms}

We begin by recalling how modular forms define cohomology classes. We continue with the setup and notation from Section~\ref{sec:loc_sys_H}. In particular, we will be working with Betti realizations so that, in this section, $\SL_2=\SL_2^B$ and $H=H^B$. More detailed references for this discussion are Sections~9 and 11 of \cite{hain:modular}.

Suppose that $V_\chi$ is an irreducible $\SL_2(\Z/N)$-module with character. Suppose that $f : \h \to V_\chi$ is a vector valued modular form of weight $m$, level $N$ and character $\chi$. By this, we mean that
\begin{equation}
\label{eqn:mod_form}
f(\gamma \tau) = (c\tau + d)^m \rho_\chi(\gamma) f(\tau),\quad \gamma \in \SL_2(\Z)
\end{equation}
where $\rho_\chi : \SL_2(\Z) \to \SL_2(\Z/N) \to \Aut V_\chi$ is the associated representation.

We take the diagonal maximal torus $t\mapsto \diag(t^{-1},t)$ in $\SL_2$. Set
$$
L = \begin{pmatrix} 0 & 1 \cr 0 & 0 \end{pmatrix}.
$$
This is an element of the Lie algebra $\sL_2$ of torus weight $-2$; and its transpose is the element of torus weight $2$. Denote the irreducible $\SL_2$-module with a highest weight vector $\ee$ of (torus) weight $k$ by $S^k(\ee)$. The holomorphic 1-form
$$
\w_f(\ee) := 2\pi i f(\tau)\, e^{2\pi i \tau L}(\ee)\, d\tau \in \Omega^1(\h)\otimes S^m(\ee)\otimes V_\chi
$$
on the upper half plane $\h$ with values in $S^m(\ee)\otimes V_\chi$ is $\SL_2(\Z)$-invariant in the sense that
$$
(\gamma^\ast \otimes 1) \w_f(\ee) = (1\otimes \gamma) \w_f(\ee)
$$
for all $\gamma \in \SL_2(\Z)$. It defines a class in
$$
H^1(\SL_2(\Z),S^m(\ee)\otimes V_\chi) \cong H^1(\G(N),S^m(\ee))_{\chi^\vee}.
$$

We need to know which characters of $\SL_2(\Z/N)$ occur in $H^1(\G(N),S^m(\ee))$. We have the following elementary necessary condition.

\begin{lemma}
If the character $\chi$ of $\SL_2(\Z/N)$ occurs in $H^1(\G(N),S^m H)$, then $\chi(-\id) = (-1)^m$.
\end{lemma}

\begin{proof}
If $\chi$ occurs in $H^1(\G(N),S^m H)$, there is a non-zero modular form $f$ satisfying (\ref{eqn:mod_form}). Taking $\gamma = -\id$ in this equation implies that $\rho(-\id) = (-1)^m$.
\end{proof}

Jared Weinstein \cite[Thm.~4.3]{weinstein} has shown that this is the only restriction.

\begin{theorem}[Weinstein]
\label{thm:weinstein}
An irreducible complex representation $\rho$ of $\SL_2(\Z/N)$ occurs in $H^1(\G(N),S^m H_\C)$ if and only if $\rho(-\id) = (-1)^m$.
\end{theorem}

\subsection{The conjugation action on $\cO(\SL_2)$ and $\cO(\SL_2(\Zhat))$}

To construct elements of $\Cl(\cG)$ from modular forms, we will need to know which representations of $\SL_2 \times \SL_2(\Zhat)$ occur in $\cO(\SL_2)^\conj\otimes \cO(\SL_2(\Zhat))^\conj$.

We regard $\cO(\SL_2)$ as a left $\SL_2\times \SL_2$-module via the action $(g,h)\varphi : x \mapsto \varphi(g^{-1}xh)$. It is isomorphic to
$$
\cO(\SL_2) = \bigoplus_{m\ge 0} \End^\vee_\kk (S^m H) \cong \bigoplus_{m\ge 0} S^m H\boxtimes S^m H.
$$
The conjugation action is obtained by restricting to the diagonal. Consequently, we have the isomorphism
$$
\cO(\SL_2)^\conj = \bigoplus_{m\ge 0} S^m H \otimes S^m H \cong \bigoplus_{m\ge 0} \big(S^{2m} H + S^{2m-2} H + \dots + S^2 H + S^0 H \big).
$$

\begin{proposition}
\label{prop:adjoint_reps}
Only even symmetric powers of $H$ occur in the conjugation representation of $\SL_2$ on its coordinate ring and each occurs with infinite multiplicity.
\end{proposition}

The Chinese remainder theorem implies that
$$
\cO(\SL_2(\Zhat)) = \sideset{}{'}\bigotimes_{p \text{ prime}}  \cO(\SL_2(\Z_p)).
$$
So, to understand $\cO(\SL_2(\Zhat))^\conj$, it suffices to understand $\cO(\SL_2(\Z/p^n))^\conj$ for all $p$ and $n > 0$.

Since $-\id$ is central in $\SL_2(\Z/p^n)$, it will act trivially on $\cO(\SL_2(\Z/p^n))^\conj$. The converse is true in most cases.

\begin{theorem}[Tiep]
\label{thm:tiep}
Suppose that $p > 3$ and $n>0$. An irreducible complex representation $\rho$ of $\SL_2(\Z/p^n)$ occurs in $\cO(\SL_2(\Z/p^n))^\conj\otimes \C$ if and only if $\rho(-\id) = 1$.
\end{theorem}

This generalizes the result \cite{HSTZ} in the $n=1$ case. The proof of the theorem appears in Appendix~\ref{sec:tiep}. An affirmative resolution of the remaining cases $p=2,3$ has recently appeared in the preprint \cite{monteiro-stasinski}.

\begin{remark}
The even symmetric powers $S^{2n}H$ of $\SL_2$ are precisely the representations on which $-\id$ acts trivially. So, in both the case of $\SL_2$ and $\SL_2(\Z/p^n)$, the only restriction on the representations $\rho$ that occur in the conjugation representation is that $\rho(-\id) = 1$, at least when $p>3$.
\end{remark}

\subsection{Class functions from modular forms}

Here we show all modular forms of all levels give rise to class functions. These constructions imply that their motives occur in the weight graded quotients of $\Cl(\cG)$. In view of the result in the previous two sections, we will need to separate the cases of odd and even weight.

\subsubsection{The even weight case}
\label{sec:even_wt}

Here we elaborate on Remark~\ref{rem:cocycle}. Suppose that $N\ge 1$, that $m \ge 2$ is even and that $\chi$ is a character of $\SL_2(\Z/N)$ that appears in its conjugation representation. Theorem~\ref{thm:weinstein} implies that there is a non-zero vector valued modular form $f : \h \to V_\chi$ of weight $m$ and level $N$. The corresponding form $\w_f(\ee)$ takes values in $S^m(\ee)\otimes V_\chi$.

The following result implies that a suitable Tate twist of the simple $\Q$-Hodge structure $V_f$ associated with a Hecke eigen cusp form $f$ of {\em even} weight appears in $\Gr^W_1 \Cl(\cG)$. Tate twists of the Hodge structure associated with a Hecke eigen cusp forms of odd weight do {\em not} appear in $\Gr^W_1 \Cl(\cG)$, but all do occur in $\Gr^W_m \Cl(\cG)$ for infinitely many $m>1$ as we show in Section~\ref{sec:odd_wt}.

\begin{proposition}
\label{prop:even_wt}
For each $\SL_2 \times \SL_2(\Z/N)$-invariant function
$$
\varphi : S^m(\ee) \otimes V_\chi \to \cO(\SL_2)^\conj\boxtimes \cO(\SL_2(\Z/N))^\conj
$$
(where the groups act on the target via conjugation), the function
$$
F_{f,\varphi} : \alpha \mapsto \Big\langle \int_\alpha \w_f(\varphi),\alpha \Big\rangle,\quad \alpha \in \SL_2(\Z)
$$
is a class function on $\SL_2(\Z)$ that is the restriction of an element of $\Cl_\C(\cG)$. 
\end{proposition}

Such class functions correspond to elements of
$$
H^1(\SL_2(\Z),\cO(\SL_2\times \SL_2(\Zhat))^\conj)\otimes\C
$$
by the discussion in Remark~\ref{rem:cocycle}. Although this result can be deduced from that discussion, we give a complete proof because of the centrality of the class functions $F_{f,\varphi}$ in the sequel.

\begin{proof}
We first explain why the value of $F_{f,\varphi}$ on $\alpha$ is well defined. To compute the integral, we first choose a base point $\tau_0 \in \h$ and a path $c_\alpha$ from $\tau_0$ to $\alpha\tau_0$. It is unique up to homotopy. Define
$$
F_{f,\varphi}(\alpha) = \Big\langle \int_{c_\alpha} \w_f(\varphi),\alpha \Big\rangle.
$$
We need to explain why it does not depend on $\tau_0$. Suppose that $\tau_1$ is another base point. Choose a path $c_1$ in $\h$ from $\tau_1$ to $\tau_0$. Then $c_\alpha' = c_1 c_\alpha (\alpha\cdot c_1^{-1})$ is a path from $\tau_1$ to $\alpha\tau_1$. Since $\w_f(\varphi)$ is invariant in the sense that
$$
(\gamma^\ast \otimes \id) \w_f(\varphi) = (1\otimes \gamma_\ast) \w_f(\varphi) \text{ for all } \gamma \in \SL_2(\Z).
$$
we have
\begin{align*}
\int_{c_\alpha'} \w_f(\varphi) &= \int_{c_1} \w_f(\varphi) + \int_{c_\alpha} \w_f(\varphi) + \alpha\int_{c_1^{-1}}\w_f(\varphi)
\cr
&= (1-\alpha) \int_{c_1} \w_f(\varphi) + \int_{c_\alpha} \w_f(\varphi)
\end{align*}
so that
$$
\Big\langle \int_{c_\alpha'} \w_f(\varphi),\alpha \Big\rangle - \Big\langle \int_{c_\alpha} \w_f(\varphi),\alpha \Big\rangle = \Big \langle (1-\alpha) \int_{c_1} \w_f(\varphi),\alpha \Big \rangle = 0
$$
as $\w_f(\varphi)$ takes values in $\cO(\SL_2)^\conj \boxtimes \cO(\SL_2(\Zhat)^\conj)$. Thus the definition of $F_{f,\varphi}$ does not depend on the choice of the base point.

Conjugation invariance can be proved similarly. Alternatively, Proposition~\ref{prop:H1u} implies that $\w_f(\varphi)$ represents an element of
$$
[H^1(\u_\C)\otimes \cO(\SL_2\times \SL_2(\Zhat))^\conj]^{\SL_2 \times \SL_2(\Zhat)}
$$
and thus to a continuous $\SL_2 \times \SL_2(\Zhat)$-invariant function
$$
H_1(\u_\C) \to \cO(\SL_2\times \SL_2(\Zhat))^\conj.
$$
This gives a class function on $\cG^B$ via the construction in Section~\ref{sec:length1}.
\end{proof}

For future use, we compute the action of the Adams operator $\psi^m$ on $F_{f,\varphi}$.

\begin{lemma}
\label{lem:adams_length_1}
With notation as above, we have, for all $m\ge 0$,
$$
\psi^m F_{f,\varphi} = m F_{f,\psi^m(\varphi)}
$$
\end{lemma}

\begin{proof}
Since integrating $\w_f(\varphi)$ over elements of $\SL_2(\Z)$ is a 1-cocycle, we have
$$
\int_{\alpha^m} \w_f(\varphi) = (1+\alpha + \cdots + \alpha^{m-1})\int_{\alpha} \w_f(\varphi).
$$
Since $\w_f(\varphi)$ takes values in $[\cO(\SL_2)\otimes\cO(\SL_2(\Zhat))]^\conj$, we have
\begin{align*}
(\psi^m F_{f,\varphi})(\alpha)
&= F_{f,\varphi}(\alpha^m)
\cr
&= \Big \langle \int_{\alpha^m} \w_f(\varphi),\alpha^m \Big\rangle
\cr
&=
\Big \langle (1+\alpha + \cdots + \alpha^{m-1})\int_{\alpha} \w_f(\varphi),\alpha^m \Big\rangle
\cr
&=  m \Big \langle \int_\alpha \w_f(\psi^m(\varphi)),\alpha \Big\rangle
= m F_{f,\psi^m(\varphi)}(\alpha).
\end{align*}
\end{proof}

\subsubsection{Odd weight}
\label{sec:odd_wt}

The construction given in the previous section does not work for forms of odd weight. This is because odd symmetric powers of $\SL_2$ do not occur in $\cO(\SL_2)^\conj$, nor do the representations of $\SL_2(\Zhat)$ that occur in spaces of modular forms of odd weight as they take the value $-1$ on $-\id$. To construct class functions from modular forms of odd weight, we need to consider iterated integrals of length 2 or more and use Proposition~\ref{prop:cyclic_nonzero}.

Suppose that $f$ and $g$ are non-zero vector valued modular forms of level $\le N$ of weights $m$ and $n$, both odd. Suppose that the corresponding characters of $\SL_2(\Z/N)$ are $\chi$ and $\psi$. These determine cohomology classes
$$
\w_f(\ee) \in H^1(\SL_2(\Z),S^m (\ee)\otimes V_\chi) \text{ and } \w_g(\ff) \in H^1(\SL_2(\Z),S^n(\ff)\otimes V_\psi).
$$
Assume that $\chi\psi$ occurs in $\cO(\SL_2(\Z/N))^\conj$ as should be guaranteed by Tiep's result and \cite{monteiro-stasinski}. Since $m+n$ is even, there is a non-zero $\SL_2 \times \SL_2(\Zhat)$-invariant map
$$
\varphi : S^m (\ee)\otimes S^n(\ff) \to \cO(\SL_2^\DR \times \SL_2(\Z/N)).
$$
We therefore have the twice iterated integral
$$
I_\varphi := \varphi \circ \bigg(\int \w_f(\ee) \w_g(\ff)\bigg)
$$
which takes values in $\cO(\SL_2^\DR)^\conj\otimes \cO(\SL_2(\Zhat))^\conj\otimes\C$. Its cyclic average $\overline{I}_\varphi$ is an element of $\Cl_2(\cG^\DR_\C)$. Its image in $\Cl_\C(\cG^B)$ is the function
$$
F_{f\times g,\varphi} : \alpha \mapsto \langle \overline{I}_\varphi(\alpha), \alpha \rangle.
$$
Proposition~\ref{prop:cyclic_nonzero} implies that it is non-zero.

\begin{proposition}
With these assumptions, the function $F_{f\times g,\varphi}$ is a non-zero element of $\Cl_\C(\cG^B)$. If $f$ and $g$ are Hecke eigenforms whose corresponding Hodge structures are $V_f$ and $V_g$, then for suitable choices of $d\in \Z$, $F_{f\times g,\varphi}$ lies in a copy of $(V_f\otimes V_g) (d)$ in $\Gr^W_\bdot\Cl_\C(\cG^B)$.
\end{proposition}

By taking $g$ to be an Eisenstein series, we conclude that Tate twists of the Hodge structure of a Hecke eigenform $f$ of odd weight occur in $\Gr^W_\bdot \Cl(\cG^B)$.

\subsection{Weight graded quotients}
\label{sec:wt_gr}

Denote the ``motive'' of a Hecke eigen cusp form $f$ of weight $m$ by $V_f$. By this, we mean the simple $\Q$-Hodge structure of weight $m-1$ associated with $f$. The following theorem follows by taking products (or cyclic products) of the class functions $F_{f,\varphi}$ constructed in Section~\ref{sec:even_wt} and using the fact that the motive $V_f$ associated with a cusp form $f$ is a simple Hodge structure.

\begin{theorem}
If $f_1,\dots,f_n$ are linearly independent Hecke eigenforms of $\SL_2(\Z)$ of level 1, then for all positive integers $r_1,\dots, r_n$, a Tate twist of the Hodge structure
$$
\Sym^{r_1} V_{f_1} \otimes \dots \otimes \Sym^{r_n} V_{f_n}
$$
appears in $\Gr^W_\bdot \Cl(\cG^B)$. In particular, the ring $\Cl(\cG_1^B)$ is not finitely generated.
\end{theorem}

One can obtain Tate twists of these Hodge structures by adding an Eisenstein series to the list $f_1,\dots,f_n$. Similarly, subject to finding the appropriate characters in the conjugation representation $\Cl(\SL_2(\Zhat))$, one can find the Hodge structures of all tensor products of symmetric powers of cusp forms of all levels in $\Gr^W_\bdot\Cl(\cG^B)$.

One can prove the analogous statement about the $\ell$-adic Galois representations that appear in $\Gr^W_\bdot \cO(\cG^\et_\ell)$. The details are left to the reader.

\subsection{The Hecke action on $\Cl_1(\cG)$ and its periods}
\label{sec:hecke}

Suppose that $f$ is a modular form of even weight $m$ and level 1. Fix an $\SL_2$-invariant function $\varphi : S^{m-2} H \to \cO(\SL_2)^\conj$ and a prime number $p$. Our goal in this section is to compute the action of $T_p$ on selected periods of the class function $F_{f,\varphi}$ that was constructed in Section~\ref{sec:even_wt}.

As in Proposition~\ref{prop:T_p}, we write the Hecke correspondence $T_p$ as
$$
\xymatrix{
& Y_0(p) \ar[dl]_\pi \ar[dr]^{\pi_\op} \cr
Y && Y
}
$$
where $\pi$ corresponds to the inclusion $\G_0(p) \hookrightarrow \SL_2(\Z)$ and $\pi_\op$ to the inclusion of $\G_0(p)$ into $\SL_2(\Z)$ defined by $\gamma \mapsto g_p^{-1} \gamma^{-\T} g_p$. For each $\alpha \in \blambda(\SL_2(\Z))$ we have $\langle T_p F_{f,\varphi}, \alpha \rangle = \langle \pi_\op^\ast F_{f,\varphi},\pi^\ast \alpha \rangle$.

To compute $\pi^\ast \alpha$ we lift $\alpha$ to $\SL_2(\Z)$. According to Proposition~\ref{prop:formula}
$$
\pi^\ast : \alpha \mapsto \sum_{j=0}^p (\gamma_j \alpha \gamma_j^{-1})^{d_j(\alpha)}/d_j(\alpha)
$$
where $\gamma_0,\dots, \gamma_p$ are the coset representatives given in Proposition~\ref{prop:coset_reps} and $d_j(\alpha)$ is the length of the orbit of $\G_0(p)\gamma_j$ under the right action by $\langle\alpha\rangle$. So
\begin{align}
\label{eqn:Tp_on_F}
\langle T_p F_{f,\varphi},\alpha\rangle
&= \sum_{j=0}^p \frac{1}{d_j} \Big\langle\int_{(\gamma_j \alpha \gamma_j^{-1})^{d_j}} \pi_\op^\ast\w_f(\varphi),(\gamma_j \alpha \gamma_j^{-1})^{d_j} \Big\rangle
\cr
&= \sum_{j=0}^p \frac{1}{d_j} \Big\langle\int_{(\gamma_j \alpha \gamma_j^{-1})^{d_j}} \pi_\op^\ast\w_f(\psi^{d_j}\varphi),\gamma_j \alpha \gamma_j^{-1} \Big\rangle
\end{align}
where we have abbreviated $d_j(\alpha)$ to $d_j$. These integrals are well defined by the discussion in the proof of Proposition~\ref{prop:even_wt}. Unfortunately, we {\em cannot} use conjugation invariance to replace
$$
\Big\langle \int_{\gamma_j \alpha \gamma_j^{-1}} \pi_\op^\ast\w_f(\psi^{d_j}(\varphi)),\gamma_j \alpha \gamma_j^{-1} \Big\rangle
$$
by
$$
F_{f,\psi^{d_j}(\varphi)}(\alpha) := \Big\langle \int_\alpha \pi_\op^\ast\w_f(\psi^{d_j}(\varphi)),\alpha\Big\rangle
$$
in this formula as $\gamma_j \notin \G_0(p)$.

At this stage, it is useful to relate $T_pF_{f,\varphi}$ to the standard action of $T_p$ on modular forms. As defined in \cite[VII\S5.3]{serre:arithmetic}, the image of $f$ under this action is 
\begin{equation}
\label{eqn:Tp_std}
T_p(f)(\tau) = p^{m-2}\Big(pf(p\tau) + p^{-m+1}\sum_{j=0}^{p-1} f\big((\tau+j)/p\big)\Big)
\end{equation}
The Hecke operator $T_p$ also acts on $\w_f(\varphi)$ as an element of $H^1(\SL_2(\Z),\cO(\SL_2))\otimes\C$. These actions are related:

\begin{lemma}
\label{lem:two_Tps}
If $f$ is a modular form of even weight $m$ and level 1, then
$$
p^{(m-2)/2}T_p\, \w_f(\varphi) =  \w_{T_p(f)}(\varphi).
$$
\end{lemma}

Note that, in general, this is not the same as the action of $T_p$ on $F_{f,\varphi}$, as we shall see in the example below.

\begin{proof}
We use the notation of Sections~\ref{sec:loc_sys_H} and \ref{sec:coho_mod-forms}. Additional details can be found in \cite[\S11]{hain:modular}. The expression
$$
\bw(\tau) := \ba^\vee + \tau \bb^\vee
$$
is the cohomology class of the holomorphic differential on $\C/(\Z\oplus \Z\tau)$ that takes the value 1 on the loop corresponding to the interval $[0,1]$. It should be regarded as a section of the Hodge bundle over the upper half plane. The group $\GL_2(\R)^+$ acts on this line bundle and
$$
(g^\ast \bw)(\tau) = (c\tau + d)^{-1} \bw(\tau), \quad \text{where }
g = \scriptstyle{\begin{pmatrix} a & b \cr c & d \end{pmatrix}}.
$$
This implies that
\begin{equation}
\label{eqn:transform}
g^\ast \bw^{m-2} d\tau = \frac{\det g}{(c\tau + d)^m} \bw^{m-2}d\tau.
\end{equation}
In the notation of Section~\ref{sec:coho_mod-forms}
$$
\w_f((\ba^\vee)^{m-2}) = f(\tau) \bw^{m-2}(\tau) d\tau.
$$
Denote it by $\w_f$. Let $\gamma_0,\dots,\gamma_p$ be the coset representatives from Proposition~\ref{prop:formula}. Recall from Section~\ref{sec:hecke_action} that $g_p = \diag(p,1)$. The pullback of $\w_f$ to $\M_{1,1}^\an$ along $T_p$ is the sum of the pullbacks along
$$
\begin{pmatrix} 0 & -1 \cr 1 & 0\end{pmatrix} g_p \gamma_j,\quad j = 0,\dots,p.
$$
The formulas (\ref{eqn:Tp_std}) and (\ref{eqn:transform}) imply that
$$
p^{m-2} T_p \w_f = p^{m-2}\Big(pf(\tau) + p^{-m-1} \sum_{j=0}^{p-1} f\big((\tau+j)/p\big)\Big) \bw^{m-2}d\tau = \w_{T_p(f)}.
$$
We have to adjust this as $(\ba^\vee)^{m-2} \in S^m H^\vee$, while $\varphi$ takes values in the highest weight part of $S^k H \otimes S^k H^\vee \cong \End(S^{m-2} H)^\vee$, where $k=(m-2)/2$. This has highest weight vector $(\ba^\vee)^k\bb^k$. Since $g_p^\ast \ba^\vee = \ba^\vee$ and $g_p^\ast \bb = p\bb$, we see that
$$
T_p \w_f(\varphi) = p^k\Big(pf(\tau) + p^{-m-1} \sum_{j=0}^{p-1} f\big((\tau+j)/p\big)\Big) \varphi(\bw^{m-2})(\tau) d\tau = p^{-k}\w_{T_p(f)}(\varphi).
$$
\end{proof}

\subsection{The Hecke action on periods: an example}
\label{sec:hecke_period}

The dual Hecke operators act on motivic periods of $\Cl(\cG)$. The action is defined by
$$
\Tdual_N : [\Cl(\cG);\alpha,F] \mapsto [\Cl(\cG);\alpha,\Tdual_N(F)] = [\Cl(\cG);T_N(\alpha),F],
$$
where $\alpha \in \SL_2(\Z)$ and $F \in \Cl(\cG^\DR)$. To illustrate how this works, we compute one example. Suppose that $\alpha$ is an element of $\SL_2(\Z)$ whose image in $\PSL_2(\Fp)$ has order $p+1$. (Such elements exist by the discussion in Section~\ref{sec:non-square}.)

\begin{proposition}
If $\alpha \in \SL_2(\Z)$ is as above and if $f$ is a modular form of level 1 and weight $m$, then
$$
\Tdual_p [\Cl(\cG);\alpha,F_{f,\varphi}] = \frac{\psi^{p+1}}{p^{(m-2)/2}(p+1)}[\Cl(\cG);\alpha,F_{T_p(f),\varphi}].
$$
\end{proposition}

\begin{proof}
We will abuse notation and write, for example,
$$
\Big \langle \int_\alpha \w_f(\varphi),\alpha \Big\rangle
$$
in place of $[\Cl(\cG);\alpha,F_{f,\varphi}]$. The proof holds for motivic periods as we appeal only to formal properties of integrals and not to their specific values.

The assumption implies that $\alpha^{p+1} \in \G_0(p)$. This implies that $\pi^\ast \alpha^{p+1} = (p+1)\alpha^{p+1}$. This, Lemma~\ref{lem:two_Tps} and the fact that integrating $T_p\w_{f,\varphi}$ defines a 1-cocycle on $\SL_2(\Z)$ implies that
\begin{align*}
\int_{\pi^\ast \alpha} \pi_\op^\ast \w_f(\varphi)
&= \int_{\alpha^{p+1}} \pi_\op^\ast \w_f(\varphi)
\cr
&= \frac{1}{p+1}\int_{\pi^\ast\alpha^{p+1}}\pi_\op^\ast \w_f(\varphi)
\cr
&= \frac{1}{p+1} \int_{\alpha^{p+1}} \pi_\ast \pi_\op^\ast \w_f(\varphi)
\cr
&= \frac{1}{p+1} \int_{\alpha^{p+1}} T_p \w_f(\varphi)
\cr
&= \frac{1}{p+1}(1 + \alpha + \cdots + \alpha^p) \int_\alpha T_p \w_f(\varphi)
\cr
&= \frac{p^{-(m-2)/2}}{p+1}\,(1 + \alpha + \cdots + \alpha^p)  \int_\alpha  \w_{T_p(f)}(\varphi).
\end{align*}
Plugging into the formula (\ref{eqn:Tp_on_F}) and applying Lemma~\ref{lem:adams_length_1}, we get
\begin{align*}
(T_p F_{f,\varphi})(\alpha)
&= \Big\langle \int_{\pi^\ast\alpha} \pi_\op^\ast \w_f(\varphi), \alpha^{p+1} \Big\rangle
\cr
&= \frac{p^{-(m-2)/2}}{p+1}\Big\langle(1 + \alpha + \cdots + \alpha^p) \int_\alpha  \w_{T_p(f)}(\varphi),\alpha^{p+1}\Big\rangle
\cr
&= p^{-(m-2)/2} \Big\langle \int_\alpha  \w_{T_p(f)}(\psi^{p+1}(\varphi)),\alpha\Big\rangle
\cr
&= p^{-(m-2)/2}F_{T_p(f),\psi^{p+1}(\varphi)}(\alpha)
\cr
&= \frac{1}{p^{(m-2)/2}(p+1)} \big(\psi^{p+1}F_{T_p(f),\varphi}\big)(\alpha).
\end{align*}
\end{proof}

\section{Filtrations}
\label{sec:filtrations}

The ring $\cO(\cG)$ and its subring $\Cl(\cG)$ have several natural filtrations which are defined below, where we also describe their behaviour under Hecke correspondences. We will omit the decoration $\w \in \{B,\DR,\etl\}$ when a filtration is defined on all realizations and these filtrations correspond under the comparison maps.

\subsection{The (relative) coradical filtration}

This is also called the {\em length filtration} as it corresponds to the filtration of iterated integrals by length. It is an increasing filtration
$$
0 = C_{-1} \cO(\cG) \subset C_0 \cO(\cG) \subset C_1 \cO(\cG) \subset C_2 \cO(\cG) \subset \cdots 
$$
which is defined on all realizations. It is defined by setting $C_0 \cO(\cG) = \cO(R)$, where $R = \SL_2\times \SL_2(\Zhat)$. When $n>0$, $C_r \cO(\cG)$ is defined to be the kernel of the $r$th ``reduced diagonal''
$$
\Deltabar^r : \cO(\cG) \to (\cO(\cG)/\cO(R))^{\otimes (r+1)}
$$
It restricts to the filtration
$$
0 = C_{-1} \Cl(\cG) \subset C_0 \Cl(\cG) \subset C_1 \Cl(\cG) \subset C_2 \Cl(\cG) \subset \cdots 
$$
This agrees with the definition (\ref{eqn:len_filt}).

This filtration is preserved by the Hecke operators:
$$
T_N : C_r \Cl(\cG) \to C_r \Cl(\cG).
$$
The shuffle product formula (\ref{eqn:product}) implies that multiplication induces a map
$$
C_r \Cl(\cG) \otimes C_s \Cl(\cG) \to C_{r+s} \Cl(\cG).
$$
Note, however, that its graded quotients are infinite dimensional.

\subsection{Hodge and weight filtrations}

As previously noted, the weight filtration is defined on all realizations. It satisfies
$$
0 = W_{-1} \Cl(\cG) \subset W_0 \Cl(\cG) \subset W_1 \Cl(\cG) \subset \cdots
$$
and is finer than the length filtration:
$$
W_r \Cl(\cG) \subseteq C_r \Cl(\cG).
$$
This inclusion is strict, except when $r=0$ when we have
$$
W_0 \Cl(\cG) = C_0 \Cl(\cG) = \Cl(\SL_2 \times \SL_2(\Zhat)).
$$

The Hodge filtration
$$
\cdots \supset F^{-1} \Cl(\cG^\DR) \supset F^0 \Cl(\cG^\DR) \supset \cdots
\supset F^p \Cl(\cG^\DR) \supset F^{p+1} \Cl(\cG^\DR) \supset \cdots 
$$
is defined on the de~Rham realization. It extends infinitely in both directions. To see why, consider the class functions $F_{f,\varphi}$ constructed from cusp forms $f$ of weight $2n$ and level 1 in Section~\ref{sec:even_wt}. The Hodge structures they generate are all isomorphic to that on the cuspidal cohomology group
$$
H^1_\cusp(\SL_2(\Z),S^{2n-2} H)(n-1)
$$
which has weight 1 as the coefficient module $S^{2n-2}H(n-1)$ has weight 0. Zucker's work \cite{zucker} implies that it has Hodge numbers $(1-n,n)$ and $(n,1-n)$. (See also \cite[Thm.~11.4]{hain:modular}.) This implies that the dimension of $F^{1-n} \Cl_1(\cG)$ becomes infinite as $n\to \infty$.

The Hodge and weight filtrations are both preserved by the Hecke operators as they act as morphisms of MHS by Theorem~\ref{thm:hecke}.

\subsection{Filtration by level}

This is defined for all realizations. Recall that $\cG_N$ is the relative completion of $\SL_2(\Z)$ with respect to $\SL_2\times \SL_2(\Z/N)$ and that $\cG$ is the inverse limit of the $\cG_N$. This implies that
$$
\Cl(\cG) = \varinjlim_N \Cl(\cG_N).
$$
The {\em level filtration} (which is actually a net indexed by the partially ordered set of levels, ordered by division) of $\Cl(\cG)$ is the net defined by $L_N \Cl(\cG) = \Cl(\cG_N)$. Each $L_N\Cl(\cG)$ is a subring of $\Cl(\cG)$. Elements of $\Cl(\cG_N^\DR)$ are closed iterated integrals of (not necessarily holomorphic) modular forms of level dividing $N$. They are ordered by divisibility:
$$
L_N \Cl(\cG) \subseteq L_M \Cl(\cG)
$$
when $N|M$. This ``filtration'' is not preserved by the Hecke operators. Rather, we have
$$
T_p L_N \Cl(\cG) \subseteq
\begin{cases}
L_N \Cl(\cG) & p | N,\cr
L_{Np}\Cl(\cG) & p \nmid N.
\end{cases}
$$

\subsection{The modular filtration}

Roughly this is the filtration by modular weight. To make this precise, we use the construction and notation from Section~\ref{sec:constrained_exts}. We will construct it as a filtration of $\Cl(\cG^B)$ by subalgebras. We take $\tC$ and $\tS$ as in Section~\ref{sec:betti}. For each positive integer $m$, we consider all simple $V_\alpha$ in $\Rep(\SL_2\times \SL_2(\Zhat))$ of the form $S^n H \otimes V$, where $n\le m$ and $V$ is an arbitrary simple $\SL_2(\Zhat)$-module. We take
$$
E_\alpha = H^1(\SL_2(\Z),S^n H \otimes V).
$$
In other words, we allow arbitrary extensions that are sums of extensions of the form
$$
0 \to (S^n H \otimes V)\otimes A \to E\otimes A \to A \to 0.
$$
where $n\le m$ and $A$ is simple. As explained in Section~\ref{sec:constrained_exts}, the homomorphism $\cG^B \to \pi_1((\tC,\tS;\tE),\w)$ is faithfully flat. Set
$$
M_m \cO(\cG^B) = \im\{\cO(\pi_1((\tC,\tS;\tE),\w)) \to \cO(\cG)\}.
$$
This defines an increasing filtration $M_\bdot$ of $\cO(\cG^B)$ satisfying
$$
\bigcup_{m\ge 0} M_m \cO(\cG^B) = \cO(\cG^B).
$$
There are similar compatible constructions for $\cO(\cG^\DR)$ and $\cO(\cG^\et_\ell)$, so we will consider it to be a filtration of $\cO(\cG)$. The de~Rham realization, $M_m \cO(\cG^\DR)$ consists of all closed iterated integrals of (not necessarily holomorphic) modular forms of all levels of weight $\le m+2$.

This filtration restricts to the {\em modular filtration}
$$
\Cl(\SL_2\times \SL_2(\Zhat)) = M_0\Cl(\cG) \subseteq M_1\Cl(\cG) \subseteq \M_2(\Cl(\cG)) \subseteq \cdots
$$
of $\Cl(\cG)$. Each $M_m\Cl(\cG)$ is a subring and is preserved by the Hecke operators. Its graded quotients are not finite dimensional.

\subsection{Finiteness properties}

The coradical, weight, level and modular filtrations are all filtrations by mixed Hodge structures and (after tensoring with $\Ql$), Galois representations. Even though the individual terms of these filtrations are infinite dimensional, certain of their intersections are finite dimensional.

\begin{proposition}
For each $N,m,r \ge 0$, the subspaces
$$
L_N \cap M_m \cap C_r \Cl(\cG) \text{ and } L_N \cap M_m \cap W_r \Cl(\cG)
$$
of $\Cl(\cG)$ are finite dimensional. Each is a mixed Hodge structure and, after tensoring with $\Ql$, a $\Gal(\Qbar/\Q)$-module.
\end{proposition}

\begin{example}
If $f$ is a $V_\chi$-valued modular form of even weight $m$ and level $N$. If $\varphi : S^{m-2}H\otimes V_\chi \to \cO(\SL_2)^\conj$ is $\SL_2(\Z)$-invariant, then the class function $F_{f,\varphi}$ defined in Section~\ref{sec:even_wt} satisfies:
$$
F_{f,\varphi} \in M_m \cap L_N \cap C_1 \Cl(\cG).
$$
This will lie in $F^{m/2} \cap W_1\Cl(\cG)$ if $f$ is a cusp form and in $F^{m/2}W_{m/2}\Cl(\cG)$ if $f$ is an Eisenstein series.

More generally, if $f_1,\dots,f_r$ are vector valued modular forms of even weights $m_1,\dots,m_r$, all of level dividing $N$, and if $\varphi_j : S^{m_j-2}H \otimes V_{\chi_j} \to \cO(\SL_2)^\conj$ is $\SL_2$ for each $j$, then
$$
F_{f_1,\varphi_1} \cyl F_{f_2,\varphi_2} \cyl \cdots \cyl F_{f_r,\varphi_r}
\in M_{m-2} \cap L_N \cap C_r \Cl(\cG),
$$
where $m = \max\{m_1,\dots,m_r\}$.
\end{example}

\appendix

\renewcommand{\thefootnote}{\fnsymbol{footnote}}

\section{{\sc The conjugation representation of $\SL_2(\Z/p^n)$}}
\begin{center}
 by Pham Huu Tiep\footnote[2]{Rutgers University, Piscataway, NJ 08854; {\sf tiep@math.rutgers.edu}}
\footnote[3]{The author gratefully acknowledges the support of the NSF (grant DMS-2200850), the Simons Foundation, and the Joshua Barlaz Chair in Mathematics. Part of this work was done while the author was visiting Princeton University and MIT. It is a pleasure to thank both institutions for their generous hospitality and stimulating environment. Finally, the author thanks Gabriel Navarro and Eamonn O'Brien for helpful conversations on the problem and for several computer calculations in the cases $p=2,3$.}

\end{center}
\label{sec:tiep}

\bigskip

\end{document}